\documentclass[11pt, reqno]{amsart}
\usepackage[margin=1in]{geometry}
\usepackage{amsmath, bbm}
\usepackage{amsthm}
\usepackage{amsfonts}
\usepackage{amssymb}
\usepackage{enumerate}
\usepackage{graphicx}
\usepackage{todonotes}
\usepackage{stmaryrd}
\usepackage{subcaption}
\newcommand{\norm}[1]{\left\lVert#1\right\rVert}

\usepackage{comment}
\usepackage{pbox}
\DeclareMathOperator{\Binom}{Binom}

\usepackage{hyperref}
\hypersetup{
    colorlinks,%
    citecolor=black,%
    filecolor=black,%
    linkcolor=black,%
    urlcolor=black
}

\newcommand\Bigger[2][7]{\left#2\rule{0mm}{#1truemm}\right.}

\newcommand\abs[1]{\ensuremath{\lvert#1\rvert}}

\DeclareMathOperator{\nnz}{nnz}
\DeclareMathOperator{\sgn}{sign}

\usepackage[toc]{appendix}
\usepackage[shortlabels]{enumitem}

\newcommand{\E}{\mathbb{E}}

\newcommand{\R}{\mathbb{R}}

\newcommand{\vect}{\mathbf}
\newcommand{\mycomment}[1]{}

\theoremstyle{plain}
  \newtheorem{theorem}{Theorem}

  \newtheorem{assumption}[theorem]{Assumption}

  \newtheorem{proposition}[theorem]{Proposition}
  \newtheorem{fact}[theorem]{Fact}
  \newtheorem{lemma}[theorem]{Lemma}
  \newtheorem{corollary}[theorem]{Corollary}
  
  \newtheorem{question}[theorem]{Question}
  \newtheorem{remark}[theorem]{Remark}
  \newtheorem{algorithm}[theorem]{Algorithm}

\theoremstyle{definition}
  \newtheorem{definition}[theorem]{Definition}
  
\usepackage{theoremref}
\usepackage[sortcites,maxnames=50,backend=biber,sorting=nty, defernumbers=true]{biblatex}

\usepackage{amsaddr}

\begin{document}
\date{}

\title[Matrix Perturbation: Davis-Kahan in the Infinity Norm]{Matrix Perturbation: \\ Davis-Kahan in the Infinity Norm}

\author[Bhardwaj]{Abhinav Bhardwaj}
\address{Department of Mathematics, Yale University}
\email{abhinav.bhardwaj@yale.edu}


\author[Vu]{Van Vu }
\address{Department of Mathematics, Yale University}
\email{van.vu@yale.edu}
\thanks  {Van Vu is supported by NSF grant DMS 2311252.}
\thanks { An extended abstract of this paper appears in the proceeding of SODA 2024.}

\begin{abstract}
Perturbation theory is developed to analyze the impact of noise on data and has been an essential part of numerical analysis. Recently, it has played an important role in designing and analyzing matrix algorithms. One of the most useful tools in this subject, the Davis-Kahan sine theorem,  provides an $\ell_2$ error bound on the perturbation of the leading singular vectors (and spaces). 

We focus on the case when the signal matrix has low rank and the perturbation is random, which occurs often in practice. In an earlier paper, O'Rourke, Wang, and the second author showed that in this case, one can obtain an improved theorem. In particular, the noise-to-gap ratio condition in the original setting can be weakened considerably. 

In the current paper, we develop an infinity norm version of the O'Rourke-Vu-Wang result. The key ideas in the proof are a new bootstrapping argument and the so-called iterative leave-one-out method, which may be of independent interest. 

Applying the new bounds, we develop new, simple, and quick algorithms for several well-known problems, such as finding hidden partitions and matrix completion. The core of these new algorithms is the  fact that one is now able to quickly approximate certain key objects in the infinity norm, which has critical advantages over 
approximations in the $\ell_2$ norm, Frobenius norm, or spectral norm. 
\end{abstract} 

\maketitle

\section{Introduction}

\subsection{The classical Davis-Kahan theorem}
Perturbation theory is developed to analyze the impact of noise on 
 data and has been an essential part of numerical analysis. 
The general setting is that we have a (signal or data) matrix $A$, a noise matrix $E$, and a matrix functional $f$. Our goal is a compare $f(A)$ with $f(A+E)$. A typical perturbation bound provides a upper bound for the difference $ f(A+E)- f(A) $ in some norm. 
 
 For the sake of presentation, in most of the paper, we assume that  both $A$ and $E$ are symmetric and of dimension $n$. All results in this paper can be extended to the asymetric case by a simple symmetrization trick. 
We assume that $n$ is sufficiently large, whenever needed, and asymptotic notations are used under the assumption that $n$ tends to infinity. 

Assume that $A$ has rank $r$ and let   $\sigma_i$ be the non-trivial singular values of $A$ in decreasing order, for $1 \leq i \leq r$. 
Let $\vect u_{i}$ be the corresponding singular vector of $\sigma_i$ with entries $u_{il}$. 
For the sake of presentation we assume that $\sigma_i$ are different so $\vect u_i$ are well defined, up to 
sign. Let $\tilde A= A+E$, and  use notation $\tilde \sigma_i$, $\vect{\tilde u}_i$, and $\tilde{u}_{il}$. Notice that because $\vect u_i$ and $\vect{\tilde u}_i$ are 
unique up to sign, we may always choose  the signs so that the the angle between them is at most $\pi/2$. 
Let $\Delta_i = \sigma_i - \sigma_{i+1}$, and let $\delta_i = \min \{ \Delta_{i-1}, \Delta_i\} $ be the distance from $\sigma_i$ to the nearest singular value. We take $\sigma_0 = \infty$ for the sake of consistency.

One of the most useful tools in perturbation theory is the Davis-Kahan bound, which provides a perturbation bound 
for singular vectors.

\begin{theorem}[Davis-Kahan] \label{DKW}There is a constant $C >0$ such that, provided $\delta_i \ge 2\norm{E},$ 

\begin{equation} \label{DKW1} \| \vect u_i - \vect{\tilde u}_i \|_2 \le C \frac{ \| E\| } {\delta_i } . \end{equation}

\end{theorem} 

The first version of this theorem, by Davis and Kahan 
\cite{davis-kahan} was stated for eigenvectors (and eigenspaces). Later, Wedin \cite{wedin1972} extended the results to singular vectors. In this paper, we use singular vectors, which are  simpler to handle. The more general version of the Davis-Kahan theorem gives a perturbation bound for the spaces spanned by a set of singular vectors. In this paper, we focus on individual singular vectors, but the results can be extended into that direction with simple modifications.

\noindent It is important to notice  that for the RHS of \eqref{DKW1} be small, one needs

\begin{equation} \label{delta1} \gamma\| E \| \le \delta_i,  \end{equation}  for some large $\gamma >0$. In other words, the noise to gap ratio 
$\frac{\|E\| }{\delta _i} $ has to be small. We will refer to this as the {\it noise-to-gap ratio} assumption.

\vskip2mm 

{\it Notation.} We use the conventional asymptotic notations, such as $o, O, \Omega, \Theta$.  We will also use the notation $f(n) = \tilde{O}(g(n))$ if there exists some absolute $c$ such that $f(n) = O(g(n)\log^c n)$;  $\tilde{\Theta}$ and $\tilde{\Omega}$ are defined similarly.

\subsection{Low rank data with random noise and an improved version of Davis-Kahan theorem} 

In modern studies, the following two assumptions come up frequently. First the data matrix $A$ has low rank, and second, the noise matrix $E$ is random. 

The low rank (or approximately low rank) phenomenon is automatic in a number of theoretical settings, such as the clustering problem discussed in Section \ref{applications}. It also occurs in so many real life problems that researchers have even tried to give a theoretical explanation for this; see \cite{udell2018big}.

Under the assumption that  $A$ has low rank $r$ and $E$ is random, the second author discovered  that one can improve the Davis-Kahan bound  \cite{Vu1}. In particular, we can replace the noise-to-gap ratio assumption \eqref{delta1} by much weaker ones; for related results, see  \cite{Vu1,OVW1,OVW2,mbsfjqrobust,zhong2017eigenvector,abbesurvey,ellptheory,eldridge,ling_2022,bao_ding_wang_2021,cape_tang_priebe_2019,koltchinskii_xia_2016}.

\begin{assumption} In what follows, we assume that $A$ is a symmetric, deterministic matrix with rank $r$.
 $E$ will be a random symmetric matrix  with independent
(but not necessarily iid) upper triangular entries $\xi_{ij}$. The $\xi_{ij}$ will be $K$-bounded random variables with mean 0. A random variable $\xi$ is $K$-bounded if $|\xi| \le K$ with probability 1. 
\end{assumption}

\noindent  Following \cite{Vu1}, about 10 years ago,  O'Rourke,  Wang, and the second author \cite{OVW1} obtained the following theorem.

\begin{theorem}
    \label{OVW1eigenvector}
For any constants $\tau, r >0$, there exists  a constant $C_0$ such that with probability at least $1 - \tau$, 

\begin{equation}
\label{OVW1bound}
    \norm{\tilde {\vect u}_1 - \vect u_1} _2 \leq C_{0}\Big[\frac{Kr^{1/2}}{\delta_1} + \frac{\norm{E}}{\sigma_1} + \frac{\norm{E}^{2}}{\delta_1\sigma_1}\Big].
\end{equation}  
\end{theorem}

\noindent The theorem holds for other singular vectors as well, with simple modifications. A more quantitative form of this theorem \cite{OVW1} allows one to take $\tau \rightarrow 0$ with $n$.

The key point in Theorem \ref{OVW1eigenvector} is that for the RHS to be small, we only need to assume 
 $\frac{\| E\|^2 }{ \sigma_i \delta_1 } $  
and  $\frac{1}{\delta_1}$ are small, 
which is much weaker than  \eqref{delta1},  where we   need to 
require that $\frac{\| E\| }{\delta_1}$ is small. (In standard settings, $\|E\|$ is of order $\sqrt n$). 

In many applications
(see \cite{OVW1}), the gap $\delta_1$ is   smaller than $\| E \|$, so the "noise to gap ratio is small" assumption \eqref{delta1} is violated. On the other hand, even if $\delta_1 \le \| E\|$, it is still often the case that   the product $\delta_1 \sigma_1$ is larger than $\| E\| ^2$, as $\sigma_1$ can be way larger than both $\delta_1$ and $\| E \| $, and our bound applies. Let us illustrate this with an example. 

\vskip2mm 

{\it \noindent Example.} Let $A$ be a matrix whose entries are of order $\Theta (1)$ with constant rank $r$, and $E$ be a matrix whose entries are iid standard Gaussian. 
Since $\sum_{i=1}^r \sigma_i^2 = \| A\|_F^2= \Theta (n^2)$, we expect that the non-trivial singular values of $A$ are of order $\Theta (n)$.
On the other hand, it is well known that 
(with high probability), $\| E \| = (2+o(1)) \sqrt n$. Furthermore, by a  simple truncation trick, we can set $K=20 \sqrt {\log n}$, as this holds with overwhelming probability. Thus, in Theorem \ref{OVW1eigenvector}, we only need to require the gap 
$\delta_i$ to be $\Omega (\sqrt {\log n})$ to have a meaningful conclusion (making the RHS of \eqref{OVW1bound} going to zero). On the other hand, an application of Davis-Kahan theorem would require 
$\delta_i =\Omega (\sqrt n)$, a significantly stronger assumption,   to achieve the same conclusion. 

\vskip2mm 

For more recent progress in this direction, see \cite{OVW2}.

\subsection {The infinity norm version}

Theorem \ref{DKW} provides an an $\ell_2$ estimate. It  is  natural and important to obtain similar results in the infinity norm. Going from $\ell_2$ to $\ell_{\infty} $ is always a non-trivial task and progress has only been made in the last 10 years or so. There are many papers in this topic considering either the infinity norm or the $\ell_{2 \rightarrow \infty}$ norm \cite{fan2017ellinfty,abbefanentrywise,eldridge,Agterberg2021EntrywiseEO,Cape2017TheTN,caietal}. However, in all of these papers, one needs to use the original noise to gap ratio assumption \eqref{delta1}. Our study in this paper will go  beyond this setting, as our goal is to obtain an infinity 
norm version of Theorem \ref{OVW1eigenvector}, which holds under  weaker assumptions. 

\vskip2mm

While finishing this paper,  we became aware of a result in 
\cite{asymmetriccheng2}. In this paper, the authors studied a {\it hybrid} model where a {\it symmetric}  low rank matrix is perturbed with {\it asymmetric} random noise. This paper also studied the infinity norm and does not need to assume \eqref{delta1}.  On the other hand, the  analysis  relies strongly on the  hybrid model and is totally different from the methods in this paper; see \cite{asymmetriccheng2}  for details. The hybrid model does not seem to occur very often in applications, as it is natural to assume that $A$ and $E$ have the same type of symmetry.

\section {New results } 
\label{newresults}

\subsection{An optimal guess}

Our goal is to find an infinity norm analogue of Theorem \ref{OVW1eigenvector}, the improved version of Davis-Kahan theorem, with a significantly weakened noise to gap assumption. Consider a singular vector ${\bf u}_i$, its perturbed counterpart 
$\tilde {\bf u}_i$, and the infinity norm difference  $  \| \tilde{\vect u}_i - \vect u_i  \| _{\infty} $. To start, let us raise a question. 

\begin{question} What is the best possible bound we can hope to achieve for $  \| \tilde{\vect u}_i - \vect u_i  \| _{\infty} $?
\end{question}

Consider the $\ell_2$ norm difference $\| \tilde{\vect u}_i - \vect u_i  \|_2 $.
It is apparent that 

$$ \| \tilde{\vect u}_i - \vect u_i  \| _{\infty} \ge \frac{1} {\sqrt n} 
\| \tilde{\vect u}_i - \vect u_i  \|_2. $$ Thus, the best bound one would hope for (up to a poly-logarithmic factor, which is usually unavoidable in a random setting)  is 
\begin{equation} \label{best1}
 \| \tilde{\vect u}_i - \vect u_i  \| _{\infty} = \tilde O ( \frac{1} {\sqrt n}\| \tilde{\vect u}_i - \vect u_i  \|_2  ). \end{equation}
However, \eqref{best1} may be too optimistic. In practice, it is natural to expect that coordinate-wise errors are proportional to the magnitude of the coordinates. Simply speaking, the error at a higher magnitude coordinate is likely to be larger. Thus, a more realistic version of \eqref{best1} is

\begin{equation} \label{best2}
\| \tilde{\vect u}_i - \vect u_i  \| _{\infty} = \tilde O ( \|\vect u _i \| _{\infty} \| \tilde{\vect u}_i - \vect u_i  \|_2 ). \end{equation} We are going  to prove that under certain conditions,  
 a slightly weaker version  of \eqref{best2} holds. The 
 precise form of the result is a bit technical, but in essence it shows
 (see Remark \ref{promise} for discussion)

\begin{equation} \label{newbound}   \| \tilde{\vect u}_i - \vect u_i \| _{\infty}  = \tilde O (  \| U \|_{\infty} \| \tilde{\vect u}_i - \vect u_i  \|_2 ).  \end{equation} The parameter $\| U \| _{\infty}$ is not an adhoc one. It has played an important role in many applications of spectral methods in statistics, through the notion of {\it incoherence}  \cite{candestaomatrixcompletion,recht} \cite{compressivesampling,chensurvey}. Our main theorem roughly asserts that under a modest  condition on the singular values and the gaps, \eqref{newbound} holds. Let us  illustrate with a corollary of our main results.

\begin{theorem}[Leading singular vector perturbation]
\label{main-corollary}
Let $E$ be a symmetric, $K$-bounded random matrix with independent upper triangular entries. Then with probability $1 - o(1)$,

\begin{equation} \label{main-corollary-eq} 
        \norm{\tilde{\vect u}_1 - \vect{u}_1}_{\infty}  \le c\norm{U}_{\infty} \Big[ \frac{\norm{E}}{\sigma_1} + \frac{K\sqrt{\log n}}{\delta_1} + \frac{\norm{E}^2}{\sigma_1\delta_1}\Big]+\frac{cK\sqrt{\log n}}{\sigma_1}.
\end{equation}

\end{theorem}

\vskip2mm 
{\noindent}Notice that in the main term $c\norm{U}_{\infty} \Big[ \frac{\norm{E}}{\sigma_1} + \frac{K\sqrt{\log n}}{\delta_1} + \frac{\norm{E}^2}{\sigma_1\delta_1}\Big]$, the term  $\Big[ \frac{\norm{E}}{\sigma_1} + \frac{K\sqrt{\log n}}{\delta_1} + \frac{\norm{E}^2}{\sigma_1\delta_1}\Big]$ is essentially the RHS of 
the bound \eqref{OVW1bound} for $\| \tilde{ \vect u}_1 - \vect u_1 \| _2$. In the case $K$ grows slowly with $n$, the second term $\frac{cK\sqrt{\log n}}{\sigma_1}$ is often negligible compared to the main term. 

\vskip2mm 

In the next three sections, we present our main theorems.

\subsection {Main theorems: The deterministic setting} 
Our first main theorem is  a deterministic one (Theorem \ref{coordinatetheorem}) where we can measure the difference of the eigenvectors between two matrices $A$ and $A+H$, where both $A$ and $H$ are deterministic. This theorem asserts a relation between the eigenvectors of $A$ and $A+H$ through information on the eigenvalues of $A, A+H$ and  those of the principal minors of $A+H$. This is somewhat  close, in spirit, to the eigenvector-eigenvalue identity, discovered several times in linear algebra, most recently through the study of neutrino oscillations by 
Denton,  Parke, Tao, and Zhang \cite{dpz}; see  \cite{Taosurvey} for a survey.

We keep the definition of all  parameters such as $\epsilon_1(i), \epsilon_2(i)$
 (with $H$ playing the role of $E$).  We denote by $H^{\{l\}}$ the matrix obtained by zeroing out the $l$th row and column of $H$. 

  Let $A$ and $H$ be  symmetric matrices of size $n$, where $A$ has rank $r$.  For any $1\le l \le n$,  set $ A^{\{l\}} = A + H^{\{l\}}$ and $\tilde A= A+H$. Notations such as $\sigma_i, \tilde \sigma_i$ and $\sigma_i ^{\{l\}}$ are self-explanatory.

\begin{theorem}
\label{coordinatetheorem}
Consider $A,H, \tilde A, A^{\{l\}}, H^{\{l\}}$ as above. 
     Let  $U^{\{l\}}$ denote the $n \times r$ matrix of $r$ leading singular vectors of $A^{\{l\}}$ and  $\vect x = \vect x(l)$ be the $l$th row of $H$, except with the lth entry of $\vect x$ reduced to $H_{ll}/2$. Set $C_0 := 272\times 4r^{3/2}$, and define $a_l := \norm{U^{\{l\}T} \vect x}_{2}$. Suppose that 
    \begin{itemize}
        \item $\sigma_i > C_0\norm{H}$
        \item $\delta_i > C_0\max\{a_l, \kappa_i\norm{H}\norm{U}_{\infty}\}$
        \item $\min\{\abs{\tilde{\sigma}_i - \sigma_{i+1}^{\{l\}}}, \abs{\tilde{\sigma}_i - \sigma_{i- 1}^{\{l\}}}\} > \delta_i/2 .$
    \end{itemize}
    
      Then, 
    \begin{equation}
    \label{coordinatebound}
        \abs{\tilde{u}_{il} - u_{il}} \le C_0\norm{U_{l, \cdot}}_{\infty}\Big[\kappa_i\norm{\tilde{\vect u}_i - \vect{u}_i}_2 + \epsilon_1(i) + a_l\kappa_i\epsilon_2(i)\Big] + 256r\frac{\abs{\langle \vect u^{\{l\}}_i, \vect x \rangle}}{\sigma_i},
    \end{equation}

where $\tilde{\vect u}_i$ is the $i$th singular vector of $A + H$. 
\end{theorem}

\begin{remark} \label{stability-deterministic}
The first assumption $\sigma_i > C_0\norm{H}$ is a signal to noise assumption.

The second assumption $\delta_i > C_0\kappa_ir^{1/2}\max\{a_l, \norm{H}\norm{U}_{\infty}\}$ is a gap assumption (replacing the stronger assumption \eqref{delta1} 
from the original Davis-Kahan theorem). In many applications (including all applications in this paper), $\kappa_i =O(1)$ and $\| U\| _{\infty} = n^{-1/2 +o(1) }$, thus $\kappa_i \| H\| \| U\| _{\infty} = \| H\| n^{-1/2 +o(1) } $, improving \eqref{delta1}  by a factor of $n^{-1/2 +o(1)} $. 
The term $a_l$  measures the correlation between $H$ and $A$. This is small if $\vect x$ does not align with any non-trivial eigenvector of $A^{(l)}$. If $H$ is random then $\vect x$ is a random vector independent of $A$ and this holds trivially.

The last assumption is a stability assumption. Intuitively, we expect that 
$\tilde \sigma_i$ is close to $\sigma_i$ and $\sigma_{i+1} ^{\{l\}} $ close to $\sigma_{i+1} $, which would imply that $\tilde \sigma_i - \sigma_{i+1} ^{\{l\}} $ is close to $\sigma_i -\sigma_{i+1} \ge \delta _i$. Our stability assumption guarantees a weaker bound $|\tilde \sigma_i - \sigma_{i+1} ^{\{l\}} | \ge \delta_i/2$.
 
\end{remark}

\subsection {Main theorems: The random setting with small $K$ } 
Now we consider the random model $A+E$, where $E$ is random matrix whose entries are $K$ bounded random variable.  The result in this section holds for any $K$, but for large $K$, we have a better result (under a slightly stronger assumption) in the next section.

To ease the presentation, we introduce the following definition. 

 


\begin{definition}
\label{stable}
For a matrix $A$, we say that a singular value and its gap $(\sigma_i, \delta_i)$ is $(c, \tau, \nu)$ {\it stable} under E if the following are all true.  Let $T = \inf\{t > 0: \mathbb{P}(\norm{E} > t) \le \tau\}$.
 
\begin{enumerate}[(a)]
\item $\sigma_i > cT$.
\item $\delta_i > c(K\log^{\nu/2}n + \sigma_{i}^{-1}T^2)$. 
\item $\delta_i > c\kappa_iT\norm{U}_{\infty}$; where $\kappa_i := \sigma_1/\sigma_i$.\
\end{enumerate}

\end{definition}

The conditions 
in this definition will guarantee that $\sigma_i, \delta_i$
are stable, in that they do not vary too much after the perturbation by $E$. This gives us control on $\tilde \sigma_i, \tilde \delta_i$, which is important in the analysis. Most importantly, it guarantees that  the stability assumption in Theorem \ref{coordinatetheorem} hold; see Remark 
\ref{stability-deterministic}.

\vskip2mm 

We will assume that our signal matrix $A$ has a singular value and gap $(\sigma_i, \delta_i)$ that is $(c , \tau, \nu)$ stable under $E$ for properly chosen $c, \tau, \nu$, where $\tau$ is a parameter that goes to zero,  $\nu$ is a small constant (like 1 or 2), and $c$ is a large constant.

Let us comment on each condition. We can think of $T$ as basically $\norm{E}$, since $T$ is a stand-in for a high probability bound of $\norm{E}$, which is often strongly concentrated;  see \cite{Vu2005SpectralNO}.

\begin{remark}[Interpretation of Stability] \label{stability-condition} The conditions here run parallel with those in Theorem \ref{coordinatetheorem}.
\begin{itemize}
\item
The first condition is simply the assumption that the signal-to-noise ratio is large. This is absolutely necessary because if the intensity of the noise  is larger than that of the signal, the signal will most likely be destroyed \cite{BBP}.

\item
The second condition  essentially asks for the gap to be at least 
polylogrithmic and the product  $\delta_i \sigma_{i+1} $ to dominate $\| E\|^2$ (which is consistent with the improved $\ell_2$ bound in Theorem \ref{OVW1eigenvector}).  

\item
The third condition requires 
the gap to be at least $\kappa_i \| U\|_{\infty}\| E\| $. This is better than \eqref{delta1} by a factor $\kappa_i \|U\| _{\infty}$, which can be as small as $n^{-1/2 }$, a large improvement.  As a matter of fact, in all applications in this paper, $\kappa_i= O(1)$ and $\| U\|_{\infty} = n^{-1/2+o(1)}$.

\end{itemize}
\end{remark}{\noindent}As the  role of $E$ is consistent through the paper,  we will simply say $(\sigma_i, \delta_i)$ is  $(c, \tau, \nu)$ stable, instead of saying  that  $(\sigma_i, \delta_i)$ is  $(c, \tau, \nu)$ stable under $E$.

\vskip2mm 

\noindent Set $\epsilon_1(i) := \| E\|/ \sigma_i$, $\epsilon_2(i) := 1/ \delta_i$,  $C(r) = 1000 \times 9^{2r}$,  $c_1(r) = 2500r^{3/2}$. Big constants like 1000 are for definiteness and are rather adhoc. The function $9^{2r} $ in the definition of $C(r)$ can be replaced by a polynomial function of $r$. However, for the sake of a simpler presentation, we make no attempt to optimize  these parameters here, and are going to use them throughout the paper. 

\begin{theorem} \label{main-result} 
Let $c_0, \tau > 0$, where $c_0$ is a constant and $\tau$ may tend to zero with $n$. Set $c= 2^{11}(c_0+1)r^3$. Assume that $(\sigma_i, \delta_i)$ is $(c, \tau, 1)$ stable. Then with probability at least $1 - C(r)n^{-c_0} - 2\tau$,
\begin{equation} \label{main-result-eq} 
        \norm{\tilde{\vect u}_i - \vect{u}_i}_{\infty}  \le c\norm{U}_{\infty} (\kappa_i\norm{\tilde{\vect{u}}_i - \vect{u}_i}_2 + \epsilon_1(i) + \kappa_i\epsilon_2(i)K\sqrt{\log n})+\frac{cK\sqrt{\log n}}{\sigma_i}.
\end{equation}
\end{theorem}

Theorem \ref{main-corollary} follows easily from Theorem \ref{main-result} and Theorem \ref{OVW1eigenvector}. 


\begin{remark} [Optimality] \label{promise}

Consider the bound in Theorem \ref{main-result} for the first singular vector. Let us  compare it to the desired bound \eqref{newbound}, which is $\norm{\vect{\tilde{u}}_{1} - \vect u_{1}}_{\infty} = \tilde O(
\| U\|_{\infty}\norm{\vect{\tilde{u}}_{1} - \vect u_{1}}_{2} ). $
We notice that the  bound \eqref{main-result-eq} is off by 2 terms: 
$\frac{cK\sqrt{\log n}}{\sigma_i} $ and 
$\tilde O( \| U\| _{\infty} (\epsilon_1 + \epsilon_2)). $

\vskip2mm

In many applications, $\sigma_1$ is sufficiently large and $K$ is sufficiently small  that the first term 
$\frac{cK\sqrt{\log n}}{\sigma_1} $ is negligible.  Furthermore, when applying \eqref{main-result-eq}, we typically do not know $\norm{\vect{\tilde{u}}_{1} - \vect u_{1}}_{2}$. In this case, the best we can do is to use the bound from Theorem \ref{OVW1eigenvector}, which contains both $\epsilon_1$ and $\epsilon_2$. Moreover, a new study \cite{OVW2} reveals that both $\epsilon_1 $ and $\epsilon_2$  are necessary in \eqref{OVW1bound}. Thus, technically \eqref{main-result-eq} has achieved what \eqref{newbound} promised. The case when $K$ is relatively large will be discussed in the next section. 

\vskip2mm 

Our proof also reveals that  if we consider the local estimate $|\tilde u_{il} - u_{il}|$, then we can replace $\| U\|_{\infty}$ by $\norm{U_{l, \cdot}}_{\infty}$ where 
$U_{l, \cdot}$ is the $l$th row of $U$. So for a local estimate, we only need 
local information from $U$.

\end{remark}


\subsection{Main theorems: Random setting with large $K$} 
\label{largeK}
Let us discuss the last term   $ K\sqrt {\log n}/\sigma_i $
in \eqref{main-result-eq}. We stated that in many cases, this term is negligible. It is indeed so when the entries of $E$ have a fixed distribution, which does not depend on $n$, as illustrated by the following two examples. 

\vskip2mm 

{\it \noindent  Example.} If the entries of $E$ are iid Rademacher ($\pm 1$), then $K=1$.

\vskip2mm 

{\it \noindent Example.} If the entries of $E$ are iid $N(0,1)$, then 
we can use the following simple truncation argument. Notice that with probability $1- o(n^{-100})$, a standard gaussian variable is bounded by $20 \sqrt {\log n} $ (with room to spare). Thus, we can replace $N(0,1)$ by 
its truncation at $20 \sqrt {\log  n}$, and set $K =20 \sqrt {\log n}$ and pay an extra term $n^{-100}$ in the probability bound. 
One can apply this trick to any distribution with a light tail. 

\vskip2mm

However, for some important applications, the entries of  $E$ 
can be reasonably large. 
A typical example here is the matrix completion problem, a fundamental problem in data science (see Section \ref{applications} for more details). 
In this case, $E$ is not really noise in the traditional sense, but a random matrix created artificially
from the setting of the problem. In this setting, the magnitude of $K$ will make the error term in question become too big. 

\vskip2mm

{\it \noindent Example: Matrix completion.} Let $A$ be matrix with rank $r$ and 
non-zero entries of order $O(1)$.  Let $B$ be a matrix obtained by keeping each entry of $A$ with probability $p$, independently. 
We call these entries {\it observed}. For a non-observed entry, write 0. 
The task is  to recreate $A$ from $B$.

If we consider $\tilde A = \frac{1}{p} B$, then $\tilde A$ is a 
random matrix with expectation equalling $A$, thanks to the normalization. Thus, we can write  $\tilde A= A +E$, where $E$ is a random matrix with independent entries with mean 0.  The entries $\xi_{ij}$  of $E$ have different distributions. For the $ij$ entry, $\xi_{ij}= a_{ij}$ with probability $1-p$, and $(1-1/p) a_{ij}$ with probability $p$. Thus $K$ is roughly $\max_{ij} |a_{ij}|/p$ which is of order $1/p$. In the matrix completion problem, one often wants to make $p$ as small as possible, typically $n^{-1 +\epsilon}$ or even $\log n/n$. 
Thus, $K$ can be close to $n$  and the error in question becomes too big.

\vskip2mm

In this section, we develop new bounds to overcome this 
deficiency.  We will need the notion of strong stability,  which is a refinement of the notion of stability introduced earlier. 

\begin{definition} \label{strong-stable} We say that $(\sigma_i, \delta_i)$  is $(c,\tau, \nu)$ {\it strongly stable} under $E$ if in addition to being $(c, \tau, \nu)$ stable, $\sigma_i$ satisfies  $\sigma_i > c\sqrt{Kn}\log^{\nu + 0.01}n$.
\end{definition}


\begin{assumption}
\label{assumption-K}
We assume for the rest of this  section that $\E[\xi_{ij}^2] \leq K \leq n$.\end{assumption}

This assumption is satisfied by random variables which take a large value $K$ with a small probability of order $1/K$. This is exactly the situation with the matrix completion problem discussed above.

\begin{remark} 
For a random matrix of size $n$, whose entries 
have zero mean and variance $K$, the spectral 
norm is typically $\Omega (\sqrt {Kn})$. Thus, in this case, the last (new) condition  in  Definition \ref{strong-stable} is only marginally stronger than the signal to noise assumption 
$\sigma_i \gg \| E\|.$
\end{remark}

\begin{theorem}
\label{refined-main}
Let $c_0, \tau > 0$, with $c_0$ constant and $\tau$ potentially tending to $0$ with n. Set $c= 2^{17}(c_0+1) r^3$. Assume that
 $(\sigma_i, \delta_i)$ is $(c, \tau, 2)$ strongly stable. Then with probability at least $1 - C(r)n^{-c_0} - \tau\log n$, 

\begin{equation} \label{refined-result-1} 
        \norm{\tilde{\vect u}_i - \vect{u}_i}_{\infty}  \le c\norm{U}_{\infty} (\kappa_i\norm{\tilde{\vect{u}}_i - \vect{u}_i}_2 + \epsilon_1(i) + \kappa_i\epsilon_2(i)K\sqrt{\log n})+\frac{c\sqrt{Kn}\kappa_i\norm{U}_{\infty}\log n}{\sigma_i}.
\end{equation}
\end{theorem}

\begin{remark} Results in random matrix theory show that  $\norm{E}$ often concentrates  around $\sim \sqrt{Kn}$ \cite{Vu2005SpectralNO}. In this case, the last term on the RHS of \eqref{refined-result-1} is basically $\kappa_i\norm{U}_{\infty}\epsilon_1(i)$. Thus, the bound essentially  becomes \begin{equation}
        \norm{\tilde{\vect u}_i - \vect{u}_i}_{\infty}  = \tilde{O}\Big(\kappa_i\norm{U}_{\infty} (\norm{\tilde{\vect{u}}_i - \vect{u}_i}_2 + \epsilon_1(i) + K\epsilon_2(i))\Big).
\end{equation} By the discussion in Remark \ref{promise}, the $\epsilon_1(i)$ and $\epsilon_2(i)$ terms are necessary to bound $\norm{\tilde{\vect u}_i - \vect u_i}_2$. The bound is thus \begin{equation}
        \norm{\tilde{\vect u}_i - \vect{u}_i}_{\infty}  = \tilde{O}\Big( \kappa_i\norm{U}_{\infty}\norm{\tilde{\vect{u}}_i - \vect{u}_i}_2\Big),
\end{equation} which is only off from \eqref{newbound} by a factor of $\kappa_i$. Thus, we basically achieve \eqref{newbound} even if $K$ is large.
\end{remark}


In a setting where $\| U\| _{\infty}= n^{-1/2 +o(1)}$ and  the condition number $\kappa_i= O(1)$, this new result essentially reduces $K$ to $\sqrt K$. The bound on $\| U\|_{\infty}$ (incoherence bound) is typical and necessary  in the analysis of the matrix completion problem; see for instance \cite{candesnoise,candestaomatrixcompletion,recht,kmo}.

The key step in our analysis, the so-called {\it delocalization} lemma below, is new and could be of independent interest.  

\begin{lemma}[Delocalization Lemma]
\label{delocalization-result-1}
Let $c_0, \tau > 0$, with $c_0$ constant. Set $c= 2^{11}(c_0+1)r^3$.  If $(\sigma_i, \delta_i)$ is $(c, \tau, 2)$ strongly stable, then 

\begin{equation}
\norm{\tilde{\vect u}_i}_{\infty} \leq c_1\kappa_i\norm{U}_{\infty},
\end{equation}
with probability at least $1 - C(r)n^{-c_0} - \tau\log n$.
\end{lemma}

\subsection{The rectangular case} 
The results are easy to generalize to the rectangular case, where $A \in \R^{m \times n}$, by a standard symmetrization trick. Let $N = m+n$.  Define the $N \times N$ matrix $S(A) := 
\begin{pmatrix}
0& A \\
A^{T} & 0 
\end{pmatrix}.
$ Notice that $S(A)$ is symmetric, and it is easy to show that if $A^{T}\vect u = \sigma \vect v$ and $A\vect v = \sigma \vect u$, then $S(A)(\vect u, -\vect v^{T}) = \sigma (\vect u, -\vect v^{T})$. Thus, we can apply the main result to $S(A)$ to obtain $\ell_{\infty}$ perturbation bounds for $\vect u_i$ and $\vect v_i$, the $i$th left and right singular vectors of $A$. For this theorem, we will assume $A$ has singular value decomposition $A = U\Sigma V$. We have the following rectangular analogue of Theorem \ref{main-result}.

\begin{theorem}
\label{rectangular1}
Let $m_{2}(i) = \max\{\norm{\tilde{\vect u}_i - \vect u_i }_{2}, \norm{\tilde{\vect v}_i - \vect v_i}_{2}\}$, and let $W = [U, V]$, the concatenation of $U$ and $V$.  Let $c_0, \tau > 0$, with $c_0$ constant and $\tau$ potentially tending to $0$ with $n$. Set $c = 2^{12}(c_0+1)r^3$. If $(\sigma_i, \delta_i)$ is $(c, \tau, 1)$ stable (with $\norm{W}_{\infty}$ instead of $\norm{U}_{\infty}$), we have with probability at least $1 - C(r)N^{-c_0} - 2\tau$,

\begin{equation} \label{rectangular} 
\norm{\tilde{\vect u}_i - \vect u_i}_{\infty} \le c
\norm{U}_{\infty}[\kappa_i m_2(i) + \epsilon_1(i) + \kappa_i\epsilon_2(i)K\sqrt{\log N}] + \frac{cK\sqrt{\log N}}{\sigma_i}
\end{equation}

The same holds for $\tilde{\vect v}_i - \vect{v}_i$, with $\norm{U}_{\infty}$ replaced with $\norm{V}_{\infty}$.  
\end{theorem}

Here is the analogue for Theorem \ref{refined-main}.

\begin{theorem}
\label{rectangular2}
Let $c_0, \tau > 0$, with $c_0$ constant and $\tau$ potentially tending to $0$ with $n$. Set $c = 2^{18}(c_0+1)r^3$. If $(\sigma_i, \delta_i)$ is $(c, \tau, 2)$ strongly stable (replacing $n$ in the definition of stable pair with $N$, and $\norm{U}_{\infty}$ with $\norm{W}_{\infty}$), then with probability at least $1 - C(r)N^{-c_0} - \tau\log N$, 

\begin{equation} \label{main1} 
\norm{\tilde{\vect u}_i - \vect u_i}_{\infty} \le c
\norm{U}_{\infty}[\kappa_im_2(i) + \epsilon_1(i) + \kappa_i\epsilon_2(i)K\sqrt{\log N}] + \frac{c\sqrt{KN}\kappa_i\norm{W}_{\infty}\log N}{\sigma_i}.
\end{equation}

The same holds for $\tilde{\vect v}_i - \vect{v}_i$, with $\norm{U}_{\infty}$ replaced with $\norm{V}_{\infty}$.  
\end{theorem}


\vskip2mm

\subsection{Sketch of the proofs and main new ideas} \label{mainideas} 

One can derive Theorem \ref{main-result} from Theorem \ref{coordinatetheorem} by verifying that the assumptions of Theorem \ref{coordinatetheorem} hold with high probability in the setting of Theorem \ref{main-result}. This part requires  a technical, but rather routine, computation.  

In order to prove Theorem \ref{coordinatetheorem}, we apply the leave-one-out strategy, which is a popular method to control the coordinates of an eigenvector. The starting  observation here is the following. Let $H^{\{l\}}$ be the matrix obtained from $H$ by zeroing out its $l$th row and column, then the $l$th row and column of $A+H^{\{l\}}$ and $A$ are the same. On the other hand, the $l$th row of the matrix, thanks to the eigenvector equation $Av =\lambda v$, has a direct influence on the $l$th coordinate of any eigenvector $v$. From here, it is not hard to deduce a strong bound  for the difference between the $l$th coordinates of an eigenvector of $A$ and its counterpart of $A+ H^{\{l\}}$. 

By the triangle inequality, it remains to bound  for the difference between the $l$th coordinates of the eigenvector   of $A+ H^{\{l\}}$
and its counterpart of $A+H$. It is often enough to just replace it by the $\ell_2$ distance between the vectors.  In many previous treatments, authors used the original Davis-Kahan to obtain a bootstrapping inequality \cite{abbefanentrywise,chensurvey}. Our new idea here is to exploit the special structure of the difference matrix $H- H^{\{l\}}$, which has exactly one non-trivial row and column. Using a series of linear algebra manipulations, we obtain a  more effective bound, laying the  ground for a 
stronger bootstrap argument which results in the conclusion of Theorem \ref{coordinatetheorem}.

To prove Theorem \ref{refined-main}, we introduce  the so-called 
{\it iterative leave-one-out} argument, which is a refinement 
of the original leave-one-out argument and could be of independent interest. The basic idea is as follows. Starting with the deterministic Theorem \ref{coordinatetheorem}, we observe that quality of the bound provided by this theorem  depends on 
the $\ell_{\infty}$ norm of the eigenvectors of the (leave-one-out) matrix
$A+ H^{\{l\}}$.  To control this later quantity, we apply Theorem \ref{coordinatetheorem} again, but now on $A+ H^{\{l\}}$. It will lead to a leave-two-out matrix, obtained by zeroing out two rows and columns of $H$. We keep continuing  this process (leave-three-out and so on)  and obtain a series of improvements, which converges to the desired bound. The key point here is that the more rows and columns we  leave out, the weaker the requirements on the $\ell_{\infty}$ bound become, until a point that it is automatically satisfied.


Using this new argument, we first prove the Delocalization Lemma \ref{delocalization-result-1}. It is then relatively simple to derive Theorem \ref{refined-main} from 
this lemma.

\section{Applications} 
\label{applications}

In this section, we apply our new results to a  number of  well known algorithmic problems, leading to fast and very simple algorithms. 

\subsection{Finding hidden partition}
Finding hidden partition is a popular problem in statisitics and theoretical computer science (also goes under the name of statistical block model). Here is the setting:  a
vertex set $V$ of size $n$ is partitioned into $r$ subsets $V_1, \dots V_r$, and between each pair $V_i, V_j$ we draw edges independently with probability $p_{ij} $  (we allow $i=j$). The task is to find 
a particular subset $V_j$ or all the subsets $V_1, \dots, V_r$  given one instance of the random graph. See  \cite{drineas2004,fernandez,Feige2005SpectralTA,feige2000,feige,mcsherry2001,alon,blumspencer,Condon1999AlgorithmsFG,vuhidden,KVempala} and the references therein. We think of $r$ as a constant and $n$ tends to infinity. 

The most popular approach to this problem is the spectral method (see \cite{KVempala} for a survey), which typically 
consists of two steps. In the first step, one considers the coordinates of an eigenvector of the adjacency matrix of the graph (or more generally the projection of the row vectors of the adjacency matrix onto a low dimensional eigenspace), and runs a standard clustering algorithm on these low dimensional $n$ points. The output of this step is an approximation of the truth. In the second step, one applies adhoc combinatorial techniques to clean the output to recover the mis-classified  vertices.

The input of the problem is the adjacency matrix of the (random) graph. 
Let us call this matrix $\tilde A$. Now let $A$ be the matrix of expectation
(thus the entries will be $p_{ij})$. Since there are $r$ vertex sets in the partition, this matrix  has $r$ identical blocks and thus has rank at most 
$r$. The difference  $E= \tilde A- A$ is a random matrix with independent 
upper diagonal entries. Since $p_{ij}$ is the expectation of the $ij$ entry of $\tilde A$, $E$ has zero mean.

It has been speculated  that in many cases, the cleaning step is not necessary. Our result makes a contribution towards solving this problem. The critical point here is that the existence of mis-classified vertices, in many settings, is just an artifact of the analysis in the first step, which typically relies on $\ell_2$ norm estimates. It is clear that any $\ell_2$ norm estimate, even sharp, could only imply that a majority of the vertices are well classified, which  leads to the necessity of the second step. On the other hand, if  we have a strong $\ell_{\infty}$  norm estimate, 
then we can classify all the vertices at once.
Our new infinity norm estimates will enable us to do exactly this in a number of settings, resulting in simple and fast new algorithms. In Section \ref{clustering},  we apply this idea to 
many problems in this area, including the hidden clique problem, 
the planted coloring problem, the hidden bipartition problem, and the general hidden partition problem. 

All of these problems 
have been studied heavily, with numerous treatments using different tools. 
On the other hand, our  treatment is very simple and universal for all settings considered. Moreover, in certain ranges,
the algorithm works under the weakest assumption known to date.

\vskip2mm 
Let us close this section with an illustrative example. 

\vskip2mm 

\noindent {\it The hidden clique problem.} The (simplest form) of the hidden clique problem is the following: Hide a clique $X$ of size $k$ in  the random graph 
 $G(n,1/2)$. Can we find $X$ in polynomial time? 

 Notice that the largest clique in $G(n,1/2)$, with overwhelming probability, has size approximately 
$2 \log n$ \cite{ASbook}. Thus, for any $k$ bigger than $(2+\epsilon) \log n$, 
with any constant $\epsilon >0$, X would be abnormally large and therefore detectable, by brute-force at least.  For instance, one can check all vertex sets of size $k$ to see if any of them form a clique. However, finding $X$ in polynomial time is a different matter, and the best current bound for $k$ is $k \ge c \sqrt n$, for any constant $c >0$. This was first  achieved by Alon, Krivelevich, and Sudakov  \cite{alon}; see also 
 \cite{feige}\cite{dekel} for later developments concerning faster algorithms for certain values of $c$. 

The Alon-Krivelevich-Sudakov algorithm runs as follows. It first finds $X$ when $c$ is sufficiently large, then  uses a simple  sampling  trick to 
reduce the case of small $c$ to this case. 

To find the clique for a large $c$, they first compute the second eigenvector of the adjacency matrix of the graph and  locate  the first largest $k$  coordinates in absolute value. Call this set $Y$. This is an approximation of the clique $X$, but not yet totally accurate. 
In the second, cleaning, step, they  define  $X$ as the  vertices in the graph with at least $3/4k$ neighbors in $Y$. The authors then proved that with high probability, $X$ is indeed the hidden clique. 

With our new results,  we can find $X$ immediately  by a slightly modified version of the first step, omitting the cleaning step, as promised. Before starting the main step of the algorithm, we change all zeros in the adjacency matrix to $-1$.

 \vskip2mm 

\begin{algorithm}[First singular vector clustering-FSC]

Compute the first singular vector. Let $x$ be the largest value of the coordinates and let $X$ be the set of all coordinates with  value at least $x/2$.
    
\end{algorithm}

\vskip2mm 
This is perhaps the simplest algorithm for this problem. Implementation is trivial as 
computing the first singular vector of a large matrix 
is a routine operation that appears in all standard numerical linear algebra packages.

 \begin{theorem} 
There is a constant $c_0$ such that for all $k \ge c_0 \sqrt n$, FSC outputs the hidden clique correctly with probability at least $.99$. 
\end{theorem}

\subsection{Matrix Completion} \label{completion-description}

A major problem in data science is the  matrix completion problem, which  asks to recover a large matrix from  
a sparse, random, set of observed entries. Formally speaking, let $A$ be a large $m \times n$
matrix where each entry is revealed with probability $p$, independently (thus roughly 
$pmn $ entries are observed). The goal is to recover $A$ from the set of observed entries. 

One of the key motivations for this problem is to build  rating/recommendation systems.  Assume  that 
a company wants to know customers' opinions
about the entire catalog of their products. 
They can achieve this by constructing the rating matrix of their products,  where the rows of represent customers and the columns are indexed by products, and each entry represents a rating.
Clearly, entries of high ratings suggest a natural recommendation strategy. 

The problem here is that  only part of 
the matrix is known, as most customers have used and rated only few products. Thus, one needs to complete the matrix based on these few observed entries.
A famous example here is the
Netflix problem where the entries are the ratings of movies (from 1 to 5). In fact, matrix completion has become  a public event thanks to the Netflix competition; see \cite{matrixrecommender}. 

It is clear that the task is feasible only if there is some condition on the matrix, and the most popular 
condition is that $A$ has low rank. There is a vast literature on the problem with this 
assumption; see  \cite{Chatterjee_2015, candestaomatrixcompletion, candestaomatrixcompletion, recht} and the references therein. 

A natural try for matrix completion is to find the matrix of minimal rank agreeing with the observed entries. However, 
this problem is NP-hard. 
The  idea here is to use the following relaxation 

    \begin{equation}
    \label{nuclearnormprogram}
    \text{minimize     }  \norm{X}_{*}
    \newline
    \text{ subject to.  }  X_{ij} = [P(A)]_{ij}, \text{ for all observed $(i,j)$}
    \end{equation}
    where $\norm{X}_{*}$ is the sum of the singular values of $X$. In words, the task is: among all matrices whose entries agree with the observed matrix entries, find the one with the smallest nuclear norm.  A series of papers \cite{candesrecht,candestaomatrixcompletion,candesnoise,recht},
    by Cand\`es and many coauthors show that (under various assumptions) the solution to the convex program \eqref{nuclearnormprogram} recovers $A$ exactly, with high probability.



Another idea is to use  the spectral method. Consider a matrix $\tilde A$, where $\tilde A_{ij} = p^{-1} A_{ij}$  if the entry 
$A_{ij}$ is observed, and 0 otherwise. Thus, $\tilde A_{ij}$ is a random variable with mean $A_{ij}$, as each entry is observed with probability $p$. So we can write $\tilde A = A +E$, where $E$ is a random matrix  with independent entries having zero mean. 
One can see $\tilde A$ as an unbiased estimator of $A$.   A well known work in this direction is \cite{kmo}.
In this paper, Keshavan, Montanari, and Oh first use a low rank approximation of $\tilde A$ to obtain 
an approximation of $A$ in Frobenius norm. Next, they solve an optimization problem to 
clean the output, and achieve exact recovery with high probability.
This step relies on gradient descent  performed over the cross product of two Grassmann manifolds.
See Table \ref{mc-table} for a summary of the discussed  results.

In both approaches above, one needs to solve a non-trivial  optimization problem. For a more detailed  discussion,  we refer to \cite{Li2019ASO}. 

As an application of our new results, we design a  simple spectral algorithm, whose cleaning step is simply rounding the output of the spectral step. Assume for a moment that the entries, as in the Netflix problem, are non-zero integers. We are going to show that a properly chosen low rank approximation  $B$ of $\tilde A$ satisfies 
$\| A- B \| _{\infty} < 1/2$. Thus, one can recover 
$A$ from $B$ by simply rounding the entries to the nearest integer. 
This is thanks to the fact that we now can prove that 
 a properly defined low rank approximation of $\tilde A$  approximates  $A$ in the {\it infinity norm}, compared to approximation in Frobenius norm or spectral norm in previous works. 

\begin{table}[!h]
\resizebox{\textwidth}{!}{%
\begin{tabular}{||c | c | c ||} 
\hline
Result & Algorithm & Lowest possible density\\
\hline
Cand\'es, Recht '09 \cite{candesrecht} & convex optimization & $p = \Omega(N^{-0.8}\log N)$\\
 \hline
 Cand\'es, Tao '10 \cite{candestaomatrixcompletion} & convex optimization & $p = \Omega(N^{-1}\log^2N)$  \\ 
 \hline
  Recht '11 \cite{recht} & convex optimization & $p = \Omega(N^{-1}\log^2 n)$ \\ 
  \hline
  Keshavan, Montanari, Oh '10 \cite{kmo} & spectral + cleaning &  $p = \Omega(N^{-1}\log N)$ \\ 
 \hline
  \end{tabular}}
\caption[A survey of exact matrix completion.]{A survey of results for {\it exact recovery} in the matrix completion problem for a $m \times n$ matrix, with $N = m + n$. All results are stated under the assumptions that $r = O(1)$ and $\norm{U}_{\infty} = O(n^{-1/2})$ in order to minimize the sampling density $p$.}
\label{mc-table}
\end{table}

We now describe the algorithm. 
 First compute the leading singular values and singular vectors of $\tilde A$, $(\tilde \sigma_1, \tilde u_1, \tilde v_1), 
(\tilde \sigma_2 , \tilde u_2, \tilde v_2) , \dots, (\tilde \sigma_{\tilde{s}}, \tilde u_{\tilde{s}}, 
\tilde v_{\tilde{s}})$ 
where $\tilde{s} := \max_i\{i: \tilde{\sigma}_{i} \ge  \frac{1}{8r}\| W\|_{\infty}^{-2}\}$. Let $s := \max_i\{i: \sigma_{i} \ge  \frac{1}{16r}\| W\|_{\infty}^{-2}\}.$ Observe that $\tilde{s}$ is random because it is computed from the observed $\tilde{A}$, while $s$ is deterministic. $B$ will be the low rank approximation given by

$$B:= \sum_{j=1}^{\tilde s} \tilde \sigma_j \tilde u_j \tilde v_j^T. $$The formal code is as follows.
\begin{algorithm}[Approximate-and-Round]
\label{alg:thresholdandround}
\leavevmode
\begin{enumerate}
    \item Take SVD of $\tilde{A} = \sum_{i = 1}^{\min\{m,n\}}\tilde{\sigma}_{i}\tilde{u}_{i}\tilde{v}_{i}^{T}$.
    \item Approximate: let $B = \sum_{i \leq \tilde{s}} \tilde{\sigma}_{i}\tilde{u}_{i}\tilde{v}_{i}^{T}$
    \item Round:  round the entries of  $B$ to the nearest integer.
\end{enumerate} 
\end{algorithm}

\begin{theorem}
\label{threshould-round-theorem}
Let $A$ be a $m \times n$ matrix of rank $r$ whose entries are non-zero integers, where both $ r, \norm{A}_{\infty} = O(1).$
    Let $N = m + n$ and $\overline{\delta} = \min_{i \leq s}\delta_i$.  Then, there exists $c = c( r, \norm{A}_{\infty})$ such that if 
    \begin{itemize}

    \item (signal to noise) $\sigma_s > c(\sqrt{Np^{-1}})\log^{2.01}N$ 
\vspace{2mm}
    \item (gap) $\overline{\delta} > cp^{-1}\log N$.
\vspace{2mm}
    \item (incoherence) $\norm{W}_{\infty} \le c N^{-1/2}$
\vspace{2mm}
    \item (density) $p > N^{-1}\log^{4.03}N$,
\vspace{2mm}
    \end{itemize}  then Algorithm \ref{alg:thresholdandround} recovers all of the entries of $A$ exactly with probability at least $1 - N^{-1}$.
\end{theorem}
The optimal value for the density is $p = O(\log N/N) $, which has been essentially achieved in \cite{kmo} (under various assumptions). 
In  this paper, we focus on the simplicity of both the algorithm and the proof, so do not try do optimize $p$. A more sophisticated analysis
will bring us close to the optimal bound, while keeping the algorithm essentially the same. This will be the topic of a future paper. 

Our algorithm does not require the knowledge of the rank $r$. 
Low rank approximation is a routine operation and run very fast in practice. With an input matrix of size 20,000, our algorithm takes a few minutes on a laptop; see Figure 1. 

Finally, let us comment on our 
new assumptions. The assumption that 
the entries are integers is common for  recommendation systems, as we have alluded to. Furthermore, in real life most data matrices become integral by multiplying by a relatively small constant.
For instance, if all entries have at most 2 decimal places, then 
$100 A$ is integral, and our algorithm  works with an obvious re-scaling. 

The assumption that the entries are non-zero is for convenience, and can be achieved by 
simply shifting the matrix. If we know that all entries are in the interval 
$[-L, L]$, for some integer $L >0$,  then $A + (L+1) J $ 
(where $J$ is the all-one matrix) have non-zero
entries in the interval $[1, 2L+1]$. Furthermore, the rank would change by at most 1. Thus,  the shifted matrix basically has the same parameters as the original one. 

\subsection{ Matrix completion with noise} 

In a  more realistic setting, many authors considered a model when the data matrix $A$ is already corrupted by (random) noise, and we only observe a few entries from the corrupted matrix \cite{chenchifanma,kmonoisy,candesnoise,abbefanentrywise}.

This problem looks more technical than the original (noiseless) one. On the other hand, with respect to our approach, it is still exactly the same problem.  Assume that each entry $a_{ij}$ from  $A$ is corrupted by noise $x_{ij}$ with mean zero. Thus, the corrupted matrix is $A' = A+X$. As argued before, the  observed matrix is 

\begin{equation}\label{mc-corrupted} \tilde A = A' + E', \end{equation}
where $E'$ is a random matrix with independent entries $\xi'_{ij}$ which are equal $(\frac{1}{p}-1) (a_{ij} + x_{ij}) $ with probability $p$, and 
$- (a_{ij} + x_{ij} )$ with probability $1-p$. Since $x_{ij}$ are independent bounded random variables with mean zero, $\xi_{ij}'$ are independent, zero mean and $O(1/p)$ bounded. This is still under the assumption of Theorem \ref{threshould-round-theorem}. 
Thus,  we can easily deduce the following "noisy" version. 

\begin{theorem}
\label{threshould-round-theorem-withnoise}
Let $A$ be a $m \times n$ matrix of rank $r$ whose entries are non-zero integers, where both $ r, \norm{A}_{\infty} = O(1).$ Let $X$ be a random $m \times n$ matrix with entries  $x_{ij}$ being  $B$-bounded  independent random variables with zero mean, where $B = O(1)$. Let $N = m + n$ and $\overline{\delta} = \min_{i \leq s}\delta_i$.  Then, there exists $c = c(r, \norm{A}_{\infty}, B)$ such that if 
    \begin{itemize}

    \item (signal to noise) $\sigma_s > c\sqrt{Np^{-1}}\log^{2.01}N$ 
\vspace{2mm}
    \item (gap) $\overline{\delta} > cp^{-1}\log N$.
\vspace{2mm}
    \item (incoherence) $\norm{W}_{\infty} \le c N^{-1/2}$
\vspace{2mm}
    \item (density) $p > N^{-1}\log^{4.03}N$,
\vspace{2mm}
    \end{itemize}  then Algorithm \ref{alg:thresholdandround} (given $A+X+E$ as input) recovers all of the entries of $A$ exactly with probability at least $1 - N^{-1}$.
\end{theorem}

See Figure \ref{figure:approx-round-X} for a numerical example.

\begin{figure}
  \centering
\begin{subfigure}{.475\textwidth}
\centering
  \includegraphics[width=.8\linewidth]{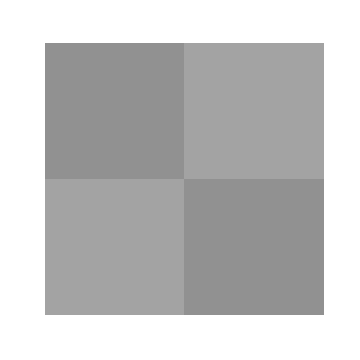}
  \caption{Original $A$}
  \label{original}
\end{subfigure}%
\begin{subfigure}{.475\textwidth}
\centering
  \includegraphics[width=.8\linewidth]{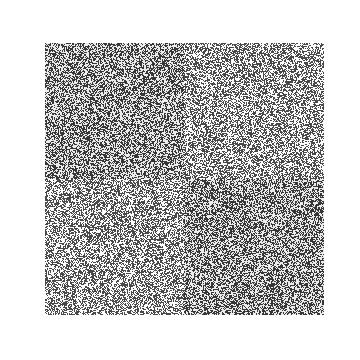}
  \caption{Noisy  $A' = A + X$}
  \label{corrupted}
\end{subfigure} \\
\begin{subfigure}{.475\textwidth}
\centering
  \includegraphics[width=.8\linewidth]{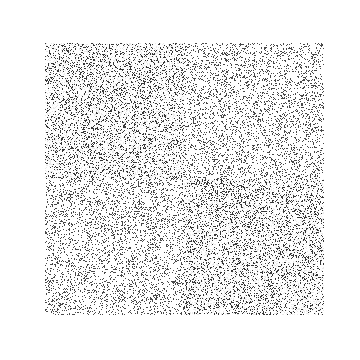}
  \caption{Sampled version of $A'$}
  \label{samples}
\end{subfigure}
\begin{subfigure}{.475\textwidth}
\centering
  \includegraphics[width=.8\linewidth]{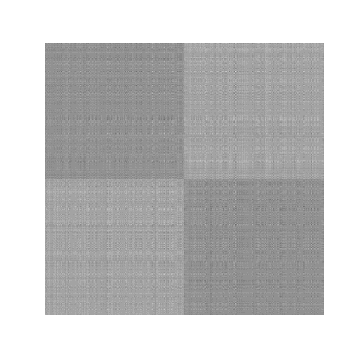}
  \caption{Recovered}
  \label{recovered}
\end{subfigure}
\caption[Numerical results for Approximate-and-Round.]{Approximate-and-Round run on $\tilde{A} = A' + E'$ as defined in \eqref{mc-corrupted}, where $A$ is a $N \times N$ block matrix of $16$'s and $18$'s and $X$ is a $\pm 10$, mean zero, random matrix. Here, $N = 20,000$ and the sampling density is $p = 0.35$.}
\label{figure:approx-round-X}
\end{figure}


\subsection*{The outline of the rest of the paper.}
We first  collect preparatory tools from linear algebra and probability in Section \ref{preparation}. Then, we prove the deterministic Theorem \ref{coordinatetheorem} in Sections \ref{deterministicproof}, \ref{sectionfirststep}, and \ref{sectionsecondstep}. The proof of Theorem \ref{main-result} from Theorem \ref{coordinatetheorem} is given in Section \ref{mainresultproof}. Afterward, we present the proof of the delocalization lemma, Lemma \ref{delocalization-result-1}, in Sections \ref{section: delocalizationproof1}, \ref{section: delocalizationproof2}, and \ref{section: delocalizationproof3}. Next, we  demonstrate in Section \ref{refined-proof} that Theorem \ref{refined-main} follows fairly easily from the proof of Lemma \ref{delocalization-result-1}. To conclude, we discuss in detail (with proofs) our applications in Sections \ref{clustering} and \ref{completion}.

\section{Preparation}
\label{preparation}
Throughout the paper, we will make repeated use of a few facts from linear algebra and probability. To make the exposition easier, we will collect these facts and some of their consequences here.

\subsection{Linear Algebra}
\begin{fact}[Weyl inequality]
\label{weylfacts}
Let $A$ be a symmetric matrix with eigenvalues $\lambda_1 \geq \lambda_2 \geq ... \geq \lambda_n$ and singular values $\sigma_1 \geq \sigma_2 \geq ... \geq \sigma_n$. Define $\tilde{A} = A + H$ for any symmetric matrix $H$. Assume that $\tilde{A}$ has eigenvalues and singular values $\tilde{\lambda}_i$ and $\tilde{\sigma}_i$, again ordered decreasingly. For all $1 \leq i \leq n$,

\begin{equation}
\begin{split}
\sigma_{i} - \norm{H} &\leq \tilde{\sigma}_{i} \leq \sigma_{i} + \norm{H}, \text{ and } \\
\lambda_{i} - \norm{H} &\leq \tilde{\lambda}_{i} \leq \lambda_{i} + \norm{H}.
\end{split}
\end{equation}

As an immediate consequence, we have the following two results.
\begin{enumerate}
\item If $\sigma_i \geq 2\norm{H}$, then $\tilde{\sigma}_i \geq \frac{\sigma_i}{2}$. 

\item If $|\lambda_i| > \norm{H}$, then $\tilde{\lambda}_i$ has  the same sign as $\lambda_i$.

\end{enumerate}
\end{fact}

 For any symmetric matrix $H$ and index set $\alpha$, let $H^{\alpha}$ be equal be the matrix obtained from $H$ by 
 zeroing out the rows and columns indexed by $\alpha$ (replacing all entries in these rows and columns by zeros). 
 The following fact is well known and easy to prove.

\begin{fact}
\label{interlacingfacts}
For any index set $\alpha$, $\norm{H^{\alpha}} \leq \norm{H}$. 
\end{fact}


\subsection{Probability}

\begin{lemma}[Hoeffding's Inequality, Theorem 2.2.6 in \cite{vershyninHighdimensionalProbabilityIntroduction2018}]
\label{Hoeffding}

Let $X_{1}, X_{2}, ... X_{n}$ be independent zero-mean 
random variables such that $a_{i} < X_{i} < b_{i}$ with probability $1$. Then, 

\begin{equation}
    \mathbb{P}\Big\{\abs{\sum_{i = 1}^{n}X_{i}} > t\Big\} \leq 2\exp\Big(\frac{-2t^{2}}{\sum_{i = 1}^{n}(b_{i} - a_{i})^{2}}\Big).
\end{equation}

\end{lemma}




\begin{corollary}
\label{hoeffcor}

Let $\bf x$ be a random vector whose entries are independent, zero-mean,
K-bounded random variables. Then for any fixed unit vector  $\vect u$ and any $C >0$,

\begin{equation*}
 \mathbb{P} \{ |  {\bf x} ^T \vect u |  \ge C K\sqrt{\log n} \} \le 2\exp(-\frac{C^2}{2}\log n) = 2n^{-C^{2}/2}.
\end{equation*}

The same bound holds if $\vect u$ is a random unit vector from which $\vect x$ is independent.

\end{corollary}.

\begin{lemma}[Bernstein's Inequality, Theorem 2.8.4 in \cite{vershyninHighdimensionalProbabilityIntroduction2018}]
\label{bernstein}
Let $X_1, \dots X_n$ be independent, $K$ bounded, mean zero,  random variables. Then

\begin{equation}
\mathbb{P}\Big\{\abs{\sum_{i = 1}^n a_iX_i} \geq t\Big\} \leq 2\exp\Big(-\frac{t^2/2}{\sum_{i = 1}^{n}\E[X_{i}^2] + Kt/3}\Big).
\end{equation}
\end{lemma}

The following is a corollary of a result 
from \cite{OVW1}; see Appendix \ref{sing-val-perturb} for the proof. 

\begin{theorem}
\label{ovw-singular-K}
Suppose that $E$ is a symmetric random matrix with $K$-bounded, mean zero, independent entries above the diagonal. Suppose that $A$ has rank $r$, and let $1 \leq k \leq r$ be an integer. Then, for any $t \geq 0$, the following hold.

\begin{equation}
\begin{split}
    \mathbb{P}\{\tilde{\sigma}_k < \sigma_k - t\} &\leq 4 \times 9^k\exp\Big(-\frac{t^{2}}{128K^2}\Big), \text{ and } \\
    \mathbb{P}\Bigg\{\tilde{\sigma}_k > \sigma_k + t\sqrt{r} + 2\sqrt{k}\frac{\norm{E}^{2}}{\tilde{\sigma}_k} + k\frac{\norm{E}^3}{\tilde{\sigma}^2_k}\Bigg\}& \leq  4 \times 9^{2r}\exp\Big(-r\frac{t^2}{128K^2}\Big).
    \end{split}
\end{equation}

In particular, 

$$\mathbb{P}\Big\{\abs{\tilde{\sigma}_k - \sigma_k} > 2r\Big(t + \frac{\norm{E}^2}{\tilde{\sigma_k}} + \frac{\norm{E}^3}{\tilde{\sigma}_k^2}\Big)\Big\} \leq 8\times 9^{2r}\exp\Big(-\frac{t^2}{128K^2}\Big).$$
\end{theorem}

\section{Proof of Theorem \ref{coordinatetheorem}}
\label{deterministicproof}

In this proof, both $A$ and $H$ are deterministic. We will  examine the effect of the full perturbation $H$ on entry $u_{il}$ by first considering the auxiliary perturbation $H^{\{l\}}$ (which is obtained from $H$ by leaving out the $l$th row and column). This is an example of the so-called leave-one-out strategy,  which has been used by many researchers in recent studies \cite{abbefanentrywise,leaveoneoutcompletion,chensurvey,zhang2022leaveoneout,zhong2017eigenvector}. Next, we need to add the $l$th row and column back 
and consider the impact of these. This is the more technical part of the proof,  which requires a careful analysis.

 Let $H_{\{l\}} = H - H^{\{l\}}$. By definition, the entries outside the $l$th row and column of $H_{\{l\}}$ are all zero. We set $A^{\{l\}} := A + H^{\{l\}}$ and call the singular values and  singular vectors of this matrix $\sigma^{\{l\}}_i$ and  $\vect u^{\{l\}}_{i}$, respectively. 
The $l$th entry of $\vect u^{\{l\}}_i$ is $u^{\{l\}}_{il}$.

 First, we show that the effect of $H^{\{l\}}$ on $u_{il}$ is extremely small. This is the content of Lemma \ref{firststep}. 
Once this is established, we view  $\tilde{A}$ as a perturbation of $A^{\{l\}}$, $\tilde A=  A^{\{l\}}+ H_{\{l\}}$. 
The structure of $H_{\{l\}}$ (now viewed as noise) will allow us to deduce a strong $\ell_{2}$ bound for the leading singular vectors of $\tilde{A}$. The key here is that this bound will be so strong that even when we use it to upper bound 
the entry-wise perturbation, the result still leads to the 
claim of our theorem. 
This bound is the content of Lemma \ref{secondstep}, which is the most technical part of the proof and requires some novel ideas, going 
far beyond applying the standard Davis-Kahan bound.


\begin{lemma} \label{firststep} 
Under the conditions of Theorem \ref{coordinatetheorem}, for any $1 \leq l \leq n$,
$$\abs{u^{\{l\}}_{il} - u_{il}} \le 2r\norm{U_{l, \cdot}}_{\infty}\Big[\kappa_i\norm{\tilde{\vect u}_i - \vect u_i}_2 + \kappa_i\norm{\tilde{\vect u}_i - \vect u^{\{l\}}_i}_2 + \epsilon_1(i)\Big].$$
\end{lemma}


\begin{lemma} \label{secondstep} 
Under the conditions of Theorem \ref{coordinatetheorem}, for any $1 \leq l \leq n$,

\begin{equation*}
\begin{split}
     \norm{\tilde{\vect u}_i - \vect u^{\{l\}}_i}_{2}&\leq 68r^{1/2}\Big[[\epsilon_1(i) + a_l\epsilon_2(i)]\abs{\tilde{u}_{il}}  + \kappa_ia_l\norm{U_{l, \cdot}}_{\infty}\epsilon_2(i)\Big] + 32\frac{\abs{\langle \vect u^{\{l\}}_i, \vect x \rangle} }{\sigma_i}.
\end{split}
\end{equation*}
\end{lemma}

\begin{proof}[Proof of Theorem \ref{coordinatetheorem} given the lemmas] 

By the triangle inequality and the fact that $\ell_2$ norm dominates the  $\ell_{\infty}$ norm, we have 
\begin{equation*}
\begin{split}\abs{\tilde{u}_{il} - u_{il}} &\leq \abs{u^{\{l\}}_{il} - u_{il}} + \norm{\tilde{\vect u}_i - \vect u^{\{l\}}_i}_{2}. \\
\end{split}
\end{equation*} By Lemma \ref{firststep}, we have \begin{equation*}
\begin{split}\abs{\tilde{u}_{il} - u_{il}} &\leq \abs{u^{\{l\}}_{il} - u_{il}} + \norm{\tilde{\vect u}_i - \vect u^{\{l\}}_i}_{2} \\
&\leq 2r\norm{U_{l, \cdot}}_{\infty}(\kappa_i\norm{\tilde{\vect{u}}_i - \vect{u}_i}_2 + \kappa_i\norm{\tilde{\vect{u}}_{i} - \vect u^{\{l\}}_i}_2 + \epsilon_1(i)) +  \norm{\tilde{\vect u}_i - \vect u^{\{l\}}_i}_{2} \\
&\leq 2r\norm{U_{l, \cdot}}_{\infty}(\kappa_i\norm{\tilde{\vect{u}}_i - \vect{u}_i}_2 + \epsilon_1(i)) +  4r\norm{\tilde{\vect u}_i - \vect u^{\{l\}}_i}_{2}. 
\end{split}
\end{equation*} Now using Lemma \ref{secondstep} to bound $\norm{\tilde{\vect u}_i - \vect u^{\{l\}}_i}_{2}$, we obtain, with $C_0 = 4*272r^{3/2}$,

\begin{equation*}
\abs{\tilde{u}_{il} - u_{il}} \le  \frac{C_0}{4}\Big[\kappa_i\norm{U_{l, \cdot}}_{\infty} \norm{\tilde{\vect u}_i - \vect u_i}_2 + [\epsilon_1(i) + a_l\kappa_i\epsilon_2(i)]\norm{U_{l, \cdot}}_{\infty} +  [\epsilon_1(i) + a_l\epsilon_2(i)]\abs{\tilde{u}_{il}}\Big] + 128r\frac{\abs{\langle \vect u^{\{l\}}_i, \vect x\rangle}}{\sigma_{i}}.
\end{equation*} By the triangle inequality, we can bound $\abs{\tilde{u}_{il}} \leq \abs{u_{il}} + \abs{\tilde{u}_{il} - u_{il}}$. This gives, letting $b = \epsilon_1(i) + a_l\epsilon_2(i)$, 

\begin{equation*}
\abs{\tilde{u}_{il} - u_{il}} \le  \frac{C_0}{4}\Big[\kappa_i\norm{U_{l, \cdot}}_{\infty} \norm{\tilde{\vect u}_i - \vect u_i}_2 + b\norm{U_{l, \cdot}}_{\infty} +  b(\abs{u_{il}} + \abs{\tilde{u}_{il} - u_{il}})\Big] + 128r\frac{\abs{\langle \vect u^{\{l\}}_i, \vect x\rangle}}{\sigma_{i}}.
\end{equation*} Then, the coefficient in front of $\abs{\tilde{u}_{il} - u_{il}}$ on the RHS is $\frac{C_0b}{4}$. 
Recall that $\epsilon_1(i) = \frac{\norm{E}}{\sigma_i}$ and $\epsilon_2(i) = \frac{1}{\delta_i}$. Therefore, because $\sigma_i \geq C_0\norm{E}$ and $\delta_i \geq C_0 a_l$, it must be the case that $b < \frac{2}{C_0}$. Then we can move all the terms with $\abs{\tilde{u}_{il} - u_{il}}$ to the left with coefficient at most $1/2$. This gives

\begin{equation} \label{bootstrap1}
 \frac{1}{ 2} |u_{ij} - \tilde u_{ij} |  \le  \frac{C_0}{4}\Big[\kappa_i\norm{U_{l, \cdot}}_{\infty} \norm{\tilde{\vect u}_i - \vect u_i}_2 + b\norm{U_{l, \cdot}}_{\infty} +  b\abs{u_{il}}\Big] + 128r\frac{\abs{\langle \vect u^{\{l\}}_i, \vect x\rangle}}{\sigma_{i}}.\end{equation}



\noindent Bounding the term  $\abs{u_{il}} $ on the RHS by $ \norm{U_{l, \cdot}}_{\infty}$,  we obtain

\begin{equation*}
\frac{1}{2}\abs{\tilde{u}_{il} - u_{il}} \le \frac{C_0\norm{U_{l, \cdot}}_{\infty} }{4}\Big[\kappa_i\norm{\tilde{\vect u}_i - \vect u_i}_2 + 2b\Big] + 128r\frac{\abs{\langle \vect u^{\{l\}}_i, \vect x\rangle}}{\sigma_{i}}.
\end{equation*}

\noindent Multiplying both sides  by $2$, we obtain the desired inequality (with room to spare). 

\end{proof}

In the next two sections, we  establish Lemmas \ref{firststep} and \ref{secondstep}, respectively.

\section{Proof of Lemma \ref{firststep}}
\label{sectionfirststep}
Recall that for a matrix $M$, 
$M_{l, \cdot}$ denotes  the  $l$th row of $M$, viewed as a vector. The key point of the leave-one-out analysis is that by definition, $A^{\{l\}}_{l, \cdot} = A _{l, \cdot} $. Furthermore, as $\vect u^{\{l\}}_i $ is a singular vector of $A^{\{l\}}$, 
$A^{\{l\}}\vect u^{\{l\}}_i $ is either $ \sigma^{\{l\}}_i \vect u^{\{l\}}_i$ or  $ -  \sigma^{\{l\}}_i \vect u^{\{l\}}_i$ and we will use the shorthand 
$A^{\{l\}}\vect u^{\{l\}}_i= \pm \sigma^{\{l\}}_i \vect u^{\{l\}}_i$. 
In all estimates where this shorthand appears, the sign does not matter. Since $A^{\{l\}}\vect{u}^{\{l\}}_i = \pm \sigma_{i}^{\{l\}}\vect{u}^{\{l\}}_i$,

$$u^{\{l\}}_{il}= \frac{\langle A^{\{l\}}_{l, \cdot}, \vect u^{\{l\}}_i\rangle}{\pm \sigma_{i}^{\{l\}}}= \frac{\langle A_{l, \cdot}, \vect u^{\{l\}}_i\rangle}{\pm \sigma_{i}^{\{l\}}}.$$

\noindent By the spectral decomposition of $A$, the RHS can be written as 

$$ \frac{\langle A_{l, \cdot}, \vect u^{\{l\}}_i\rangle }{\pm \sigma_{i}^{\{l\}}} = \frac{1}{\pm \sigma_{i}^{\{l\}}}\Big[\pm \sigma_{1} u_{1l}\vect u_1^{T} \pm \sigma_{2}u_{2l}\vect u^{T}_{2} +  \dots 
\pm \sigma_{r}u_{rl}\vect u^{T}_{r}
\Big]\vect u^{\{l\}}_i,$$ which implies, via the triangle inequality, that 

$$\lvert u^{\{l\}}_{il} - u_{il} \rvert \leq \Bigg\lvert \frac{\sigma_{i}}{\sigma_{i}^{\{l\}}}u_{il}\vect u_i^{T}\vect u^{\{l\}}_i- u_{il}\Bigg\rvert + 
\Bigg\lvert\frac{1}{\sigma_{i}^{\{l\}}}\sum_{j \neq i}^{r}\pm \sigma_{j}u_{jl}\vect u_j^T \vect u^{\{l\}}_i\Bigg\rvert.
$$In the first term on the RHS, we have eliminated the signs in front of $\sigma_{i}$ and $\sigma^{\{l\}}_i$. This is because their signs correspond to the signs of the corresponding eigenvalues of $A$ and $A^{\{l\}}$ respectively.  It must be the case that these eigenvalues have the same sign by Fact \ref{weylfacts}.
Studying the second term on the RHS, write $\vect u^{\{l\}}_i = \vect u_i + (\vect u^{\{l\}}_i - \vect u_i)$. By the orthogonality of $\vect u_i$ with $\vect u_j$ for $j \neq i$, it follows that 

\begin{equation}
\label{entrywiseexpansion1}
\lvert u^{\{l\}}_{il} - u_{il} \rvert \leq  \lvert \frac{\sigma_i}{\sigma^{\{l\}}_i}u_{il}\vect u_i^{T} \vect u^{\{l\}}_i- u_{il}\rvert + \frac{2}{\sigma_{i}}\sum_{j \neq i}^{r} \sigma_{j} \lvert u_{jl} \rvert \norm{\vect u^{\{l\}}_i - \vect u_i}_{2}.
\end{equation} To bound  first term on the right hand side of \eqref{entrywiseexpansion1}, write 
\begin{equation}
\begin{split}
\frac{\sigma_{i}}{\sigma_{i}^{\{l\}}}u_{il}\vect u_i^T \vect u^{\{l\}}_i- u_{il} &=  u_{il}\Bigg(\frac{\sigma_{i} - \sigma_{i}^{\{l\}}}{\sigma_{i}^{\{l\}}}\Bigg)\vect u_i^T \vect u^{\{l\}}_i + u_{il}\vect u_i^T(\vect u^{\{l\}}_i- \vect u_i).
\end{split}
\end{equation} Since $\sigma_{i}^{\{l\}} > \sigma_{i}/2$ by Facts \ref{weylfacts} and \ref{interlacingfacts}, the triangle inequality gives


\begin{equation}
    \begin{split}
        \lvert \frac{\sigma_{i}}{\sigma_{i}^{\{l\}}}u_{il}\vect u_i^T\vect u^{\{l\}}_i- u_{il}\rvert &\leq \abs{u_{il}}\frac{\abs{\sigma_{i} - \sigma_{i}^{\{l\}}}}{\sigma_{i}^{\{l\}}}\norm{\vect u_i}_{2}\norm{\vect u^{\{l\}}_i}_{2} + \abs{u_{il}}\norm{\vect u_i}_{2}\norm{\vect u^{\{l\}}_i- \vect u_i}_{2} \\
        & \leq 2\abs{u_{il}}\frac{\norm{H}}{\sigma_{i}} + \abs{u_{il}}\norm{ \vect u^{\{l\}}_i - \vect u_i}_{2} \\
        & \leq 2\epsilon_1(i)\norm{U_{l, \cdot}}_{\infty} + \norm{ \vect u^{\{l\}}_i - \vect u_i}_{2} \norm{U_{l, \cdot}}_{\infty}.
    \end{split}
\end{equation} The second line uses Fact \ref{weylfacts} to bound $\abs{\sigma_{i} - \sigma_{i}^{\{l\}}} \leq \norm{H^{\{l\}}}$, and Fact \ref{interlacingfacts} to get $\norm{H^{\{l\}}} \leq \norm{H}$. To bound the second term on the RHS of \eqref{entrywiseexpansion1}, bound $\abs{u_{il}} \leq \norm{U_{l, \cdot}}_{\infty}$ to obtain

$$\frac{2}{\sigma_{i}}\sum_{j \neq i}^{r} \sigma_{j} \lvert u_{jl} \rvert \norm{\vect u^{\{l\}}_i - \vect u_i }_{2} \leq 2\norm{U_{l, \cdot}}_{\infty}\frac{\norm{ \vect u^{\{l\}}_i - \vect u_i}_{2} \sigma_{1}(r-1)}{\sigma_{i}}.$$ We can bound this last term by $2\kappa_i\norm{U_{l, \cdot}}_{\infty}\norm{ \vect u^{\{l\}}_i - \vect u_i}_{2} (r - 1)$ to conclude that

\begin{equation}
\label{laststeplemma13}
    \begin{split}
\abs{u^{\{l\}}_{il} - u_{il}} &\leq 2\epsilon_1(i)\norm{U_{l, \cdot}}_{\infty} + \norm{ \vect u^{\{l\}}_i - \vect u_i}_{2} \norm{U_{l, \cdot}}_{\infty}  + 2\kappa_i\norm{U_{l, \cdot}}_{\infty}\norm{ \vect u^{\{l\}}_i - \vect u_i}_{2} (r - 1)\\
& \leq 2r\kappa_i\norm{U_{l, \cdot}}_{\infty}\norm{ \vect u^{\{l\}}_i - \vect u_i}_{2} + 2\epsilon_1(i)\norm{U_{l, \cdot}}_{\infty} \\
&\leq 2r\Big[\kappa_i\norm{U_{l, \cdot}}_{\infty}\norm{\tilde{\vect u}_i - \vect u_i}_2 + \kappa_i\norm{U_{l, \cdot}}_{\infty}\norm{\tilde{\vect u}_i - \vect{u}^{\{l\}}_i}_2 + \epsilon_1(i)\norm{U_{l, \cdot}}_{\infty} \Big].
    \end{split}
\end{equation} where the last step uses the triangle inequality. This concludes the proof of Lemma \ref{firststep}.

\section{Proof of Lemma \ref{secondstep}}
\label{sectionsecondstep}

In this section, we will view $\tilde{A}$ as a perturbation of $A^{\{l\}}$ with the perturbing matrix $H_{\{l\}}$. The main idea is that $H_{\{l\}}$ is only supported on one row and one column. By leveraging this and $\norm{U_{l, \cdot}}_{\infty}$, we can obtain a strong bound for $\norm{\tilde{\vect u}_i - \vect u^{\{l\}}_i}_2$. We begin with the following decomposition which will prove useful throughout:

\begin{equation}
\label{decomposition}
H_{\{l\}} = \vect x \vect e_l^{T} + \vect e_l \vect x^T.
\end{equation}

Recall that $\vect x$ is the $l$th row of $H$, but with the $l$th entry set to $H_{ll}/2$, as $H_{\{l\}}$ is  $H - H^{\{l\}}$. Define $p = \min\{j: \sigma_{j +1} < \sigma_{i}/4\}$. Let $P^{\{l\}}$ be the orthogonal projection to the orthogonal complement of the columns of  $U^{\{l\}}_{p}$, and let $V^{\{l\}}_{p}$ be the $n \times (p - 1)$ matrix whose columns are $\vect u^{\{l\}}_{1}, \dots , \vect u^{\{l\}}_{i-1},  \vect u^{\{l\}}_{i+1}, \dots, \vect u^{\{l\}}_{p}$. Expanding $\tilde{\vect u}_i$ in the coordinates of the orthonormal basis $\{\vect u^{\{l\}}_k\}_{1 \leq k \leq n}$,

$$ \tilde {\bf u}_i = \sum_{k=1}^p \alpha_k  {\bf u}_k ^{\{l\}} + P^{\{l\}} \tilde{\vect u}_i.$$ It follows that

\begin{equation}
\label{pythag}
    \begin{split}
\norm{\tilde{\vect u}_i - \vect u^{\{l\}}_i}^2_2 &= \langle \tilde{\vect u}_i, \tilde{\vect u}_i \rangle  + \langle \vect u^{\{l\}}_i, \vect u^{\{l\}}_i \rangle  - 2\langle \tilde{\vect u}_i , \vect u^{\{l\}}_i \rangle \\
&= 2(1 - \alpha_{i}^{2})\\
&= 2\sum_{k \neq i}^{n}\alpha_{k}^{2} \\
&= 2\sum_{k \neq i}^{p} \alpha_{k}^{2} + 2\norm{P^{\{l\}}\tilde{\vect u}_i}_2^{2} \\
&= 2\norm{V^{\{l\}T}_{p}\tilde{\vect u}_i}_2^{2} + 2\norm{P^{\{l\}}\tilde{\vect u}_i}_2^{2}.
    \end{split}
\end{equation} 
Therefore,
\begin{equation}
\label{sinedecomposition}
    \norm{\tilde{\vect u}_i - \vect u^{\{l\}}_i}_{2} \leq \sqrt{2}\Big[\norm{V^{\{l\}T}_{p}\tilde{\vect u}_i}_2 + \norm{P^{\{l\}}\tilde{\vect u}_i}_2\Big].
\end{equation} Proving Lemma \ref{secondstep} reduces to bounding the two terms on the RHS.
It is possible that the first term in the fourth line of \eqref{pythag} is a sum over an empty set (say $i = 1, p = 1$). In this case, $\norm{\tilde{\vect u}_i - \vect u^{\{l\}}_i}_{2} \leq \sqrt{2}\norm{P^{\{l\}}\tilde{\vect u}_i}_2.$ 

\begin{lemma}[$\norm{P^{\{l\}}\tilde{\vect u}_i}_{2} $ Bound]
\label{leaveoneoutprojectionbound}

\begin{equation}\norm{P^{\{l\}}\tilde{\vect u}_i}_{2} \leq 24\epsilon_1(i)(\abs{\tilde{u}_{il}} + \norm{ \tilde{\vect u}_i - \vect u^{\{l\}}_i}_{2}) + 8\frac{\abs{\langle \vect x, \vect u^{\{l\}}_i \rangle} }{\sigma_i}.
\end{equation}
\end{lemma}

As we previously observed, there is no contribution from $\norm{V^{\{l\}T}_{p}\tilde{\vect u}_i}_2$ when the aforementioned sum is empty. So we assume without loss of generality that it is not. We first establish that it is sufficient to bound the quantity $\norm{U^{\{l\}T}_{p}H_{\{l\}}\tilde{\vect u}_i}_2$ by using the perturbation technique of \cite{OVW1}. This is the content of the following proposition, which is where we use the gap stability condition. 

\begin{proposition}
\label{reduction}
\begin{equation}
    \norm{V^{\{l\}T}_{p}\tilde{\vect u}_i}_2 \leq 2\epsilon_2(i)\norm{U^{\{l\}T}_{p}H_{\{l\}}\tilde{\vect u}_i}_2. 
\end{equation}
\end{proposition}Having established this, when we go to bound $\norm{U^{\{l\}T}_{p}H_{\{l\}}\tilde{\vect u}_i}_2$, the structure of $H_{\{l\}}$ will bring the $l$th row of $U^{\{l\}}_p$ into play. It is important that this row not be too large in norm.

\begin{proposition}[The $l$th row of $U^{\{l\}}_p$ is small] Let $\vect r^{\{l\}}$ denote the $l$th row of $U^{\{l\}}_p$ viewed as a column vector. Then 
\label{leaveoneoutrowbound}

\begin{equation}
\begin{split}
\norm{\vect r^{\{l\}}}_2 \le 8r^{1/2}\kappa_i\norm{U_{l, \cdot}}_{\infty}.
\end{split}
\end{equation}
\end{proposition}
The two propositions can be shown to give us a bound for $\norm{V^{\{l\}T}_{p}\tilde{\vect u}_i}_2$.

\begin{lemma}
\label{vbound}
\begin{equation}
   \norm{V^{\{l\}T}_{p}\tilde{\vect u}_i}_2 \leq  16r^{1/2}\epsilon_2(i)\Big[a_l(\abs{\tilde{u}_{il}} + \kappa_i\norm{U_{l, \cdot}}_{\infty})+ \kappa_i\norm{H}\norm{U_{l, \cdot}}_{\infty}\norm{\tilde{\vect u}_i- \vect u^{\{l\}}_i}_{2}\Big].
   \end{equation}
\end{lemma}

We are now ready to prove Lemma \ref{secondstep}.

\begin{proof}[Proof of Lemma \ref{secondstep} given Lemmas \ref{leaveoneoutprojectionbound} and \ref{vbound}]

Lemmas \ref{leaveoneoutprojectionbound} and \ref{vbound} can be used to bound the RHS of \eqref{sinedecomposition}. In particular, temporarily setting $\beta := \norm{\tilde{\vect u}_i - \vect u^{\{l\}}_i}_{2}$ for brevity, 

\begin{equation}
\begin{split}
     \beta &\leq \sqrt{2}\Big[\norm{V^{\{l\}T}_{p}\tilde{\vect u}_i}_2 + \norm{P^{\{l\}}\tilde{\vect u}_i}_2\Big] \\
    &\leq 34r^{1/2}\Big[\epsilon_1(i)\abs{\tilde{u}_{il}} + a_l\epsilon_2(i)(\abs{\tilde{u}_{il}} + \kappa_i\norm{U_{l, \cdot}}_{\infty}) + [\epsilon_1(i) + \kappa_i\norm{H}\norm{U_{l, \cdot}}_{\infty}\epsilon_2(i)]\beta\Big] +  16\frac{\abs{\langle \vect x, \vect u^{\{l\}}_i \rangle} }{\sigma_i}.
\end{split}
\end{equation} The main observation is that $\beta$ appears on both the LHS and RHS of the inequality. The coefficient of $\beta$ in the RHS is $34r^{1/2}[\epsilon_1(i) + \kappa_i\norm{H}\norm{U_{l, \cdot}}_{\infty}\epsilon_2(i)]$. 
By the definition of $\epsilon_1(i) $ and $\epsilon_2(i)$ (see the discussion preceding Theorem \ref{coordinatetheorem}), this equals $34r^{1/2}(\frac{\norm{H}}{\sigma_i} + \kappa_i\norm{H}\norm{U_{l, \cdot}})$. By the definition of $C_0$, and the assumption that $\sigma_i > C_0\norm{E}$ and $\delta_i > C_0\kappa_i\norm{H}\norm{U}_{\infty}$, we have the following estimate.

\begin{equation}
\begin{split}
     \beta &\leq 34r^{1/2}\Big[\epsilon_1(i)\abs{\tilde{u}_{il}} + a_l\epsilon_2(i)(\abs{\tilde{u}_{il}} + \kappa_i\norm{U_{l, \cdot}}_{\infty})\Big] + \frac{1}{2}\beta + 16\frac{\abs{\langle \vect x, \vect u^{\{l\}}_i \rangle} }{\sigma_i}.
\end{split}
\end{equation} Therefore, moving the terms involving $\beta$ to the left and multiplying both sides by $2$ gives

\begin{equation*}
\begin{split}
     \norm{\tilde{\vect u}_i - \vect u^{\{l\}}_i}_{2}&\leq 68r^{1/2}\Big[[\epsilon_1(i) + a_l\epsilon_2(i)]\abs{\tilde{u}_{il}}  + \kappa_ia_l\norm{U_{l, \cdot}}_{\infty}\epsilon_2(i)\Big] + 32\frac{\abs{\langle \vect x, \vect u^{\{l\}}_i \rangle} }{\sigma_i}.
\end{split}
\end{equation*}
    
\end{proof}

\begin{proof}[Proof of Lemma \ref{leaveoneoutprojectionbound}]
Since $\tilde{A} - A^{\{l\}} = H_{\{l\}}$, it follows that 
 
\begin{equation}
\label{projectionexpansion}
(P^{\{l\}}\tilde{\vect u}_i)^{T}\tilde{A}\tilde{\vect u}_i - (P^{\{l\}}\tilde{\vect u}_i )^{T}A^{\{l\}}\tilde{\vect u}_i  = (P^{\{l\}}\tilde{\vect u}_i)^{T}H_{\{l\}}\tilde{\vect u}_i.
\end{equation} By the definition of $P^{\{l\}}$, we have \begin{equation}
\label{lambdarplus1}
\abs{(P^{\{l\}}\tilde{\vect u}_i)^{T}A^{\{l\}}\tilde{\vect u}_i} = \abs{\langle P^{\{l\}}\tilde{\vect u}_i, A^{\{l\}}\tilde{\vect u}_i \rangle} = \abs{\langle \tilde{\vect u}_i, P^{\{l\}}A^{\{l\}}\tilde{\vect u}_i \rangle}  \leq \sigma_{p+1}^{\{l\}}\norm{P^{\{l\}}\vect u_i}_{2}^{2}.
\end{equation} By Fact \ref{weylfacts}, $\sigma_{p+1}^{\{l\}} \le \sigma_{p+1} + \| H^{\{l\}}\| .$ Furthermore, by Fact \ref{interlacingfacts}, $\| H^{\{l\}}\|
\le \| H\| $, so we have $\sigma^{\{l\}}_{p+1} \leq \sigma_{p+1} + \norm{H}$. Because $\tilde{\vect u}_i$ is a singular vector of $\tilde{A}$, we have $$(P^{\{l\}}\tilde{\vect u}_i )^{T}\tilde{A}\tilde{\vect u}_i  = \pm \tilde{\sigma}_i\norm{P^{\{l\}}\tilde{\vect u}_i}_2 ^{2}.$$ It thus follows that 

\begin{equation}
\tilde{\sigma}_{i}\norm{P^{\{l\}}\tilde{\vect u}_i }_{2}^{2} - (\sigma_{p+1} + \norm {H})\norm{P^{\{l\}}\tilde{\vect u}_i }_2^{2} \le | (P^{\{l\}}\tilde{\vect u}_i)^{T}H_{\{l\}}\tilde{\vect u}_i |.
\end{equation} Applying Cauchy-Schwarz on the RHS, we obtain

\begin{equation}
\label{lemma10left}
\tilde{\sigma}_{i}\norm{P^{\{l\}}\tilde{\vect u}_i }_2^{2} - (\sigma_{p+1} + \norm {H})\norm{P^{\{l\}}\tilde{\vect u}_i }_2^{2} \leq \norm{P^{\{l\}}\tilde{\vect u}_i }_2\norm{H_{\{l\}}\tilde{\vect u}_i }_2.
\end{equation} By Fact \ref{weylfacts}, $\tilde{\sigma}_{i}(A) > \frac{\sigma_{i}}{2}$. By
definition of $p$, $\frac{\sigma_{i}}{2} - \sigma_{p + 1} > \frac{1}{4}\sigma_{i}$. So dividing by $\norm{P^{\{l\}}\tilde{\vect u}_i}_2$ gives

\begin{equation}
\norm{P^{\{l\}}\tilde{\vect u}_i}_2  \leq 
\frac{\norm{H_{\{l\}} \tilde { \vect u }_i}_2}{0.25\sigma_{i} - \norm{H}} \le \frac{\norm{H_{\{l\}}\vect u^{\{l\}}_i}_2 + \norm{H_{\{l\}}}\norm{\tilde{\vect u}_i- \vect u^{\{l\}}_i}_2}{0.25\sigma_{i} - \norm{H}} \le
8\frac{\norm{H_{\{l\}}\vect u^{\{l\}}_i}_2 + 2\norm{H}\norm{\tilde{\vect u}_i- \vect u^{\{l\}}_i}_2}{\sigma_{i}}.
\end{equation} We used the triangle inequality in the numerator. It is apparent that $\vect x$ has $\ell_{2}$ norm at most that of the $l$th row of $H$. This gave $\norm{H_{\{l\}}} = \norm{\vect x\vect e_l^T + \vect e_l \vect x^T} \leq 2\norm{\vect x} \leq 2\norm{H}$. We also lower bounded $0.25\sigma_i - \norm{H} \geq \sigma_i/8$ because $\sigma_i > C_0\norm{H}$. To estimate the term $\norm{H_{\{l\}}\vect u^{\{l\}}_i}_2$, write using \eqref{decomposition},

\begin{equation}
H_{\{l\}}\vect u^{\{l\}}_i= \langle \vect x, \vect u^{\{l\}}_i\rangle \vect e_{l} + u^{\{l\}}_{il}\vect x.
\end{equation} Here $\vect e_{l}$ is the $l$th standard basis vector. We obtain \begin{equation}
\begin{split}
\norm{H_{\{l\}}\vect u^{\{l\}}_i}_{2} &\leq \abs{\langle \vect x, \vect u^{\{l\}}_i \rangle} + \norm{\vect x}_{2}\abs{u^{\{l\}}_{il}} \\
&\le \abs{\langle \vect x, \vect u^{\{l\}}_i \rangle} + \norm{H}(\abs{\tilde{u}_{il}} + \abs{\tilde{u}_{il} - u^{\{l\}}_{il}}). \\
\end{split}
\end{equation} So we conclude that 

\begin{equation}
\norm{P^{\{l\}}\tilde{\vect u}_i}_2 \leq \frac{8\abs{\langle \vect x, \vect u^{\{l\}}_i \rangle}  + 24\norm{H}(\abs{\tilde{u}_{il}}  + \norm{\vect u^{\{l\}}_i- \tilde{\vect u}_i}_{2})}{\sigma_{i}},
\end{equation} proving the lemma. 
\end{proof}

\begin{proof}[Proof of Proposition \ref{reduction}]

By the relation 
$ \tilde{A} - A^{\{l\}} = H_{\{l\}}$, we have

\begin{equation}
\label{orthonormalexpansion}
V^{\{l\}T}_{p}\tilde{A}\tilde{\vect u}_i - V^{\{l\}T}_{p}A^{\{l\}}\tilde{\vect u}_i= V^{\{l\}T}_{p}H_{\{l\}}\tilde{\vect u}_i.
\end{equation}

First, we observe that because $\tilde{\vect u}_i$ is a singular vector of $\tilde{A}$, the first term on the LHS of \eqref{orthonormalexpansion} is $\pm \tilde{\sigma}_{i}V^{\{l\}T}_{p}\tilde{\vect u}_i$. Since the columns of $V^{\{l\}}_{p}$ are  singular vectors of $A^{\{l\}}$, the second term on the left hand side of \eqref{orthonormalexpansion} is $D^{\{l\}}V^{\{l\}T}_{p}\tilde{\vect u}_i$. $D^{\{l\}}$ is the $(p-1) \times (p-1)$ diagonal matrix with entries $\pm \sigma_1^{\{l\}}, .. \pm \sigma_{i-1}^{\{l\}}, \pm \sigma_{i+1}^{\{l\}}...  \pm \sigma_{p}^{\{l\}}$. Then, we have

\begin{equation}
\label{V1bound}
\norm{U^{\{l\}T}_{p}H_{\{l\}}\tilde{\vect u}_i}_2 \geq \norm{V^{\{l\}T}_{p}H_{\{l\}}\tilde{\vect u}_i}_2 = \norm{(\pm \tilde{\sigma}_{i} I - D^{\{l\}})V^{\{l\}T}_{p}\tilde{\vect u}_i}_2 \ge \frac{\delta_i}{2}\norm{V^{\{l\}T}_{p}\tilde{\vect u}_i}_2.
\end{equation}

The first inequality uses that $V^{\{l\}T}_{p}H_{\{l\}}\tilde{\vect u}_i$ is a sub-vector of $U^{\{l\}T}_{p}H_{\{l\}}\tilde{\vect u}_i$. The equality uses \eqref{orthonormalexpansion}, and the last inequality uses that the smallest singular value of the diagonal matrix $\pm\tilde{\sigma}_{i}I - D^{\{l\}}$ is at least $\delta_i/2$ by the assumption of Theorem \ref{coordinatetheorem}. 

\end{proof}

\begin{proof}[Proof of Proposition \ref{leaveoneoutrowbound}]

Since the columns of $U^{\{l\}}_p$ are singular vectors, if $\Sigma^{\{l\}}_{p}$ is a diagonal matrix with entries $\pm \sigma_{1}^{\{l\}}, ... \pm \sigma_{p}^{\{l\}}$, we have by definition of $\vect r^{\{l\}}$,
\begin{equation}
\begin{split}
    \norm{\vect r^{\{l\}}}_2 &= \norm{\vect e_{l}^{T}U^{\{l\}}_{p}}_2 = \norm{\vect e_{l}^{T}A^{\{l\}}U^{\{l\}}_{p}\Sigma^{\{l\}-1}_{p}}_2.\\
\end{split}
\end{equation} Because $A^{\{l\}}$ and $A$ have the same $l$th row,\begin{equation}
    \norm{\vect r^{\{l\}}}_2 = \norm{\vect e_{l}^{T}AU^{\{l\}}_{p}\Sigma^{\{l\}-1}_{p}}_2.
\end{equation} By the spectral decomposition of $A$, if $\Sigma$ is the diagonal matrix whose entries are the eigenvalues of $A$, then the the RHS is $\norm{\vect e_{l}^{T}U\Sigma U^{T} U^{\{l\}}_{p}\Sigma^{\{l\}-1}_{p}}_2  = \norm{U_{l, \cdot}\Sigma U^{T}U^{\{l\}}_{p}\Sigma^{\{l\}-1}_{p}}_2$. By the triangle inequality,
\begin{equation}
    \norm{\vect r^{\{l\}}}_2 \leq \norm{U_{l, \cdot}}_{2}\norm{\Sigma}\norm{U^{T}U^{\{l\}}_{p}}\norm{\Sigma^{\{l\}-1}_{p}}_2.
\end{equation} Because $\sigma_1 > C_0\norm{E}$ and the definition of $p$, $\Sigma^{\{l\}-1}_{p}$ has norm at most $\sigma^{\{l\}}_p \leq \sigma_{p}/2$ by Facts \ref{weylfacts} and \ref{interlacingfacts}. It is clear that $\norm{\Sigma} = \sigma_1$ and $\norm{U^{T}U^{\{l\}}_{p}} \leq 1$.  Therefore, since $\norm{U_{l, \cdot}}_2 \leq \sqrt{r}\norm{U_{l, \cdot}}_{\infty}$, 
\begin{equation}
    \norm{\vect r^{\{l\}}}_2 \leq 2r^{1/2}\Big(\frac{\sigma_{1}}{\sigma_{p}}\Big)\norm{U_{l, \cdot}}_{\infty} \leq 8\kappa_i r^{1/2}\norm{U_{l, \cdot}}_{\infty},
\end{equation} where the last inequality uses the definition of $p$.

\end{proof}

\begin{proof}[Proof of Lemma \ref{vbound}]

By Proposition \ref{reduction}, it suffices to upper bound $\norm{U^{\{l\}T}_{p}H_{\{l\}}\tilde{\vect u}_i}_2$. Using the decomposition for $H_{\{l\}}$ \eqref{decomposition}, we can write

$$U^{\{l\}T}_{p}H_{\{l\}}\tilde{\vect u}_i= (H_{\{l\}}U^{\{l\}}_{p})^{T}\tilde{\vect u}_i = (X + Y)\tilde{\vect u}_i.$$ Thus, we must bound $\norm{X\tilde{\vect u}_i}_2 + \norm{Y\tilde{\vect u}_i}_2$, where $X = U^{\{l\}T}_{p}\vect x \vect e_l^T$ and $Y = U^{\{l\}T}_{p}\vect e_l\vect x^T$. Therefore, \begin{equation}
\label{Xubound}
\norm{X\tilde{\vect u}_i}_2 \leq \abs{\tilde{u}_{il}}\norm{U_{p}^{\{l\}T}\vect x}_2 = \abs{\tilde{u}_{il}}a_l.
\end{equation} On the other hand,  for $\norm{Y\tilde{\vect u}_i}_2$, the triangle inequality gives
\begin{equation}
\label{Yexpansion}
\norm{Y\tilde{\vect u}_i}_2 \leq \norm{Y\vect u^{\{l\}}_i}_2 + \norm{Y}\norm{\tilde{\vect u}_i- \vect u^{\{l\}}_i}_{2}. 
\end{equation} By definition of $Y$, we can write $Y\vect u^{\{l\}}_i = \langle \vect u^{\{l\}}_i, \vect x \rangle U^{\{l\}T}_{p} \vect e_{l} = \langle \vect u^{\{l\}}_i, \vect x \rangle \vect r^{\{l\}}$. As before, $\vect e_{l}$ is the $l$th standard basis vector. Invoking Proposition \ref{leaveoneoutrowbound}, it follows immediately that \begin{equation}
\label{Yubound1}
\norm{Y\vect u^{\{l\}}_i}_2 \leq \abs{\langle \vect u^{\{l\}}_i, \vect x \rangle }\norm{\vect r^{\{l\}}}_2 \leq 8r^{1/2}\kappa_i a_l\norm{U_{l, \cdot}}_{\infty}.
\end{equation} Now, we need to deal with the  term $\norm{Y}\norm{\tilde{\vect u}_i- \vect u^{\{l\}}_i}_{2}$ in \eqref{Yexpansion}. We will bound $\norm{Y}$. It is easy to see that $Y$ can be written as $Y = \vect r^{\{l\}} \vect x^{T}$. In particular, $Y$ is rank $1$ so we can calculate the spectral norm of $Y$ directly as $\norm{Y} = \norm{\vect x}_2\norm{\vect r^{\{l\}}}_2$. We have previously observed that $\norm{\vect x}_2 \leq \norm{H}$ and we can bound $\norm{\vect r^{\{l\}}}_2$ as we did before using Proposition \ref{leaveoneoutrowbound}. Therefore,
\begin{equation}
\label{Yubound2}
    \norm{Y} \leq 8r^{1/2}\kappa_i\norm{H}\norm{U_{l, \cdot}}_\infty.
\end{equation} Combining estimates \eqref{Xubound}, \eqref{Yexpansion}, \eqref{Yubound1}, and \eqref{Yubound2} gives
\begin{equation*}
\begin{split}
\norm{U^{\{l\}T}_{p}H_{\{l\}}\tilde{\vect u}_i}_2 &\leq \norm{X\tilde{\vect u}_i}_2 + \norm{Y\tilde{\vect u}_i}_2 \\
&\leq \abs{\tilde{u}_{il}}a_l + 8r^{1/2}\kappa_i a_l\norm{U_{l, \cdot}}_{\infty} + 8r^{1/2}\kappa_i\norm{H}\norm{U_{l, \cdot}}_\infty\norm{\tilde{\vect u}_i- \vect u^{\{l\}}_i}_{2}.
\end{split}
\end{equation*}
\end{proof}

\section {Proof of Theorem \ref{main-result} via Theorem \ref{coordinatetheorem}}
\label{mainresultproof}

In this section, we  deduce Theorem \ref{main-result} from the deterministic Theorem \ref{coordinatetheorem}. 
The task is basically checking that the conditions of Theorem \ref{coordinatetheorem} hold with high probability. 
The hardest part is the stability condition, and for  this, we will need to appeal to singular value perturbation  bounds from
\cite{OVW1}. We will show that on the complement of a bad event $\mathcal{B}$, the conditions for Theorem \ref{coordinatetheorem} hold for $H = E$. Next, the bad event holds with small probability. 

Define the event $\mathcal{B} := \mathcal{B}_1 \cup \mathcal{B}_2 \cup \mathcal{B}_{E}$ where 

\begin{equation}
\begin{split}
\mathcal{B}_1 &:= \cup_{1 \leq l \leq n}\{\min\{\abs{\tilde{\sigma}_i - \sigma^{\{l\}}_{i+1}}, \abs{\tilde{\sigma}_i - \sigma^{\{l\}}_{i-1}}\} < \delta_i/2\}, \\ 
\mathcal{B}_2 &:= \cup_{1 \leq l \leq n}\Big\{\norm{U^{\{l\}T}\vect x(l)}_2 \geq \sqrt{2r(c_0 + 1)\log n}\Big\}, \text{ and } \\
\mathcal{B}_E &:= \{\norm{E} > T\}. \\
\end{split}
\end{equation}

\begin{lemma}
\label{lemma:B-bound}
\begin{equation}\mathbb{P}(\mathcal{B}) \leq C(r)n^{-c_0} + 2\tau.
\end{equation}
\end{lemma}
\begin{proof}[Proof of Theorem \ref{main-result} given Theorem \ref{coordinatetheorem} and Lemma \ref{lemma:B-bound}]

By the construction of $\mathcal{B}$, we can check that on $\overline{\mathcal{B}}$ (complement of $\mathcal {B})$, the conditions for Theorem \ref{coordinatetheorem}  with $H = E$ and $H^{\{l\}} = E^{\{l\}}$ are satisfied for all $l$.

\vspace{3mm}
{\noindent \it Verification of the conditions for Theorem \ref{coordinatetheorem}.} Let $1 \leq l \leq n$. We first have to check that on $\overline{\mathcal{B}}$,  $$\min\{\abs{\tilde{\sigma}_i - \sigma^{\{l\}}_{i+1}}, \abs{\tilde{\sigma}_i - \sigma^{\{l\}}_{i-1}}\} \geq \delta_i/2.$$The definition of $\mathcal{B}_1$ makes this condition trivial. Then, we have to check that $$\sigma_i > C_0\norm{E}.$$On $\overline{\mathcal{B}}$, the event $\mathcal{B}_E$ guarantees that $\norm{E} \leq T$. Therefore, by $(c, \tau, 1)$ stability,  $$\sigma_i > cT \geq c\norm{E} > C_0\norm{E}. $$ See condition $(a)$ of Definition \ref{stable}. Finally, we verify that 
\begin{equation}
\label{delta-cond-first}
\delta_i > C_0\max\Big\{\norm{U^{\{l\}T}\vect x(l)}_2, \kappa_i\norm{E}\norm{U}_{\infty}\Big\}.
\end{equation} On $\overline{\mathcal{B}}$, the event $\mathcal{B}_2$ guarantees that $$\norm{U^{\{l\}T}\vect x(l)}_2 \leq K\sqrt{2r(c_0 + 1)\log n}.$$ Because of the $(c, \tau, 1)$ stability assumption, $$\delta_i > c\max\{K\sqrt{\log n}, T\kappa_i\norm{U}_{\infty}\}.$$Since $\norm{E} \leq T$, this ensures that \eqref{delta-cond-first} holds. See conditions $(b)$ and $(c)$ of Definition \ref{stable}. 

\vspace{3mm}

{\noindent \it Conclusion of the proof.}
By Lemma \ref{lemma:B-bound},  $\mathcal{B}$ has probability at most $C(r)n^{-c_0} + 2\tau$ (recall that $C(r) = 1000 \times 9^{2r})$. Furthermore, we have checked that, if $\overline{\mathcal{B}}$ occurs, then for all $1 \leq l \leq n$, the conditions for Theorem \ref{coordinatetheorem} hold. For each $l$, Theorem \ref{coordinatetheorem} gives that

\begin{equation}
\abs{\tilde{u}_{il} - u_{il}} \le C_0\norm{U_{l, \cdot}}_{\infty}\Big[\kappa_i\norm{\tilde{\vect u}_i - \vect{u}_i}_2 + \epsilon_1(i) + a_l\kappa_i\epsilon_2(i)\Big] + 256r\frac{\abs{\langle \vect u^{\{l\}}_i, \vect x \rangle}}{\sigma_i}, \end{equation}
where we recall $a_l = \norm{U^{\{l\}T}\vect x}_2$.
Taking the maximum over $l$ on both sides gives us that on $\overline{\mathcal{B}}$,

    \begin{equation}
        \norm{\tilde{\vect u}_i- \vect u_i }_{\infty}\le C_0\norm{U}_{\infty}\Big[\kappa_i\norm{\tilde{\vect u}_i - \vect{u}_i}_2 + \epsilon_1(i) + (\max_{l}a_l)\kappa_i\epsilon_2(i)\Big] + 256r\frac{\max_l\abs{\langle \vect u^{\{l\}}_i, \vect x \rangle}}{\sigma_i}.
    \end{equation} Notice that $\max_l\abs{\langle \vect u^{\{l\}}_i, \vect x \rangle} \leq \max_l a_l$. 
Then we can use the bound for $\max_l a_l \leq \sqrt{2r(c_0 + 1)\log n}$ on $\overline{\mathcal{B}_2}$ to conclude that 

\begin{equation}
\norm{\tilde{\vect u}_i - \vect u_i }_{\infty}\le c\norm{U}_{\infty}\Big[\kappa_i\norm{\tilde{\vect u}_i - \vect{u}_i}_2 + \epsilon_1(i) + (K\sqrt{\log n})\kappa_i\epsilon_2(i)\Big] + c\frac{K\sqrt{\log n}}{\sigma_i},
\end{equation} which is precisely \eqref{main-result-eq}. \end{proof}
Thus, what remains is to prove Lemma \ref{lemma:B-bound}.

\subsection{Proof of Lemma \ref{lemma:B-bound}}
We will bound $\mathbb{P}(\mathcal{B}_E), \mathbb{P}(\mathcal{B}_2)$, and $\mathbb{P}(\mathcal{B}_1)$ separately, and use the union bound.

\vspace{2mm}
{\noindent \bf Probability of $\mathcal{B}_E$.} Recall that $$\mathcal{B}_E = \{\norm{E} > T\}.$$By the definition of $T$, it must be the case that 

\begin{equation} \label{bound1}
\mathbb{P}(\mathcal{B}_E) \leq \tau.  
\end{equation}

\vspace{2mm}
{\noindent \bf Probability of $\mathcal{B}_2$.} Recall that $$\mathcal{B}_2 = \cup_{1 \leq l \leq n}\Big\{\norm{U^{\{l\}T}\vect x}_2 \geq \sqrt{2r(c_0 + 1)\log n}\Big\}.$$Let $1 \leq l \leq n$. Observe that $\norm{U^{\{l\}T}\vect x}_2 $ is the length of the projection of a random vector onto a subspace from which it is independent. Consider the vector $U^{\{l\}T}\vect x$, which has $r$ entries. We use Corollary \ref{hoeffcor} to bound each entry $[U^{\{l\}T}\vect x]_j$ of this vector. We can bound for any $1 \leq j \leq r$,
\begin{equation}
\mathbb{P}\{\abs{[U^{\{l\}T}\vect x]_j} \geq K\sqrt{2(c_0+1)\log n}\} 
\leq 2n^{-c_0 - 1}.
\end{equation} By taking the union bound over the $r$ entries, 

\begin{equation}
\mathbb{P}\Big\{\norm{U^{\{l\}T}\vect x}_2 \geq K\sqrt{2r(c_0 +1)\log n}\Big\} \leq 2rn^{-c_0 -1}.
\end{equation} Thus $\mathcal{B}_2$ holds with probability at most $2rn^{-c_0}$ by taking a union bound over $1 \leq l \leq n$.
\begin{equation} \label{bound2}
\mathbb{P}(\mathcal{B}_2) \leq 2rn^{-c_0}.
\end{equation}

\vspace{2mm}
{\noindent \bf Probability of $\mathcal{B}_1$. } What remains is to bound the probability of $\mathcal{B}_1$. This is the hardest step, so we treat it separately.

\begin{lemma}
\label{lemma:B1-bound}
\begin{equation}
\label{eq:B1-bound}
\mathbb{P}(\mathcal{B}_1) \leq 64 \times 9^{2r}n^{-c_0}+ \tau.
\end{equation}
\end{lemma}

By using the union bound,  \eqref{bound1}, \eqref{bound2}, and \eqref{eq:B1-bound} imply Lemma \ref{lemma:B-bound}. We dedicate the remainder of the section to proving Lemma \ref{lemma:B1-bound}, which completes the proof of Lemma \ref{lemma:B-bound}.

\subsection{Proof of Lemma \ref{lemma:B1-bound}}

Define

\begin{equation}
\begin{split}
\mathcal{G}_{i-1} &:=
    \bigcap_{1 \leq l \leq n}\Bigg\{\max_{k = i-1, i} \max\{\abs{\sigma^{\{l\}}_k - \sigma_k}, \abs{\tilde{\sigma}_k - \sigma_k}\}  \leq 24r\Big[K\sqrt{r(c_0 + 1)\log n}+ \frac{\norm{E}^2}{\tilde{\sigma}_k} + \frac{\norm{E}^{3}}{\tilde{\sigma}_{k}^{2}}\Big]\Bigg\}, \\
\mathcal{G}_{i+1} &:=
    \bigcap_{1 \leq l \leq n}\Bigg\{\max_{k = i, i+1} \max\{\abs{\sigma^{\{l\}}_{k} - \sigma_{k}}, \abs{\tilde{\sigma}_{k} - \sigma_{k}}\}  \leq 24r\Big[K\sqrt{r(c_0 + 1)\log n}+ \frac{\norm{E}^2}{\tilde{\sigma}_{k}} + \frac{\norm{E}^{3}}{\tilde{\sigma}_{k}^{2}}\Big]\Bigg\}.
\end{split}
\end{equation} These good events essentially guarantee that  the relevant perturbed singular values are close to the original ones. 

\begin{proposition}
\label{boundB1set}
\begin{equation}
\mathbb{P}(\mathcal{B}_1) \leq \mathbb{P}(\overline{\mathcal{G}_{i-1}}) + \mathbb{P}(\overline{\mathcal{G}_{i+1}}) + \mathbb{P}(\mathcal{B}_E). 
\end{equation}
\end{proposition}

\begin{proof}[Proof of Lemma \ref{lemma:B1-bound} given the proposition]

We apply the results of \cite{OVW1} on the perturbation of singular values to determine the probability of $\overline{\mathcal{G}_{i+1}}$ and $\overline{\mathcal{G}_{i-1}}$. This is because for all $l$, both $E$ and $E^{\{l\}}$ satisfy the conditions of Theorem \ref{ovw-singular-K}. The bound obtained from Theorem \ref{ovw-singular-K} for both $E$ and $E^{\{l\}}$ will be the same because $\norm{E^{\{l\}}} \leq \norm{E}$ by Fact \ref{interlacingfacts}. 

We will bound $\mathbb{P}(\overline{\mathcal{G}_{i+1}})$, and the exact same bound will hold for $\mathbb{P}(\overline{\mathcal{G}_{i-1}})$. 
Let $1 \leq l \leq n$. We apply Theorem \ref{ovw-singular-K} to $k \in \{i, i+1\}$ with $t = K\sqrt{128(c_0 + 1)\log n}$. For each such $k$, 

\begin{equation}
\label{apply-OVW-1}
\begin{split}
\mathbb{P}\Bigg\{\max\{\abs{\sigma^{\{l\}}_k - \sigma_k}, \abs{\tilde{\sigma}_k - \sigma_k}\} \geq K\sqrt{128r(c_0 + 1)} + 2\sqrt{r}\frac{\norm{E}^2}{\tilde{\sigma}_k} + r\frac{\norm{E}^3}{\tilde{\sigma}_k^2}\Bigg\} &\leq 16\times9^{2r}\exp[-(c_0 + 1)\log n].\\
\end{split}
\end{equation}

This implies that 

\begin{equation}
\mathbb{P}\Bigg\{\max_{k = i, i+1} \max\{\abs{\sigma^{\{l\}}_{k} - \sigma_{k}}, \abs{\tilde{\sigma}_{k} - \sigma_{k}}\}  \leq 24r\Big[K\sqrt{r(c_0 + 1)\log n}+ \frac{\norm{E}^2}{\tilde{\sigma}_{k}} + \frac{\norm{E}^{3}}{\tilde{\sigma}_{k}^{2}}\Big]\Bigg\} \leq 32\times 9^{2r}n^{-c_0 - 1}.
\end{equation}

Since $\mathcal{G}_{i+1}$ is an intersection over $l$, we have to take a union bound over $1 \leq l \leq n$ to bound the complement. Then, $\overline{\mathcal{G}_{i+1}}$ holds with probability at most $32\times 9^{2r}n^{-c_0}$. The same bound holds for $\overline{\mathcal{G}_{i-1}}$. 

Recall that in \eqref{bound1}, we already bounded $\mathbb{P}(\mathcal{B}_E)$. Putting the bounds for $\overline{\mathcal{G}_{i-1}}$, $\overline{\mathcal{G}_{i+1}}$, and $\mathbb{P}(\mathcal{B}_E)$ together, Proposition \ref{boundB1set} implies

\begin{equation}\label{bound4}
\mathbb{P}(\mathcal{B}_1) \leq 64\times9^{2r}n^{-c_0} + \tau. 
\end{equation}
\end{proof}

\begin{proof}[Proof of Proposition \ref{boundB1set}]

Recall that 

$$\mathcal{B}_1 = \cup_{1 \leq l \leq n}\min\{\abs{\tilde{\sigma}_i - \sigma^{\{l\}}_{i+1}}, \abs{\tilde{\sigma}_i - \sigma^{\{l\}}_{i-1}}\} \leq \delta_i/2\}, \text{ and }$$
$$\mathcal{B}_E =\{\norm{E} \geq T\}.$$Let $$\mathcal{B}_{1, i-1} = \cup_{1 \leq l \leq n}\{\abs{\tilde{\sigma}_i - \sigma^{\{l\}}_{i-1}} \leq \delta_i/2\}, \text{ and }$$$$\mathcal{B}_{1, i+1} = \cup_{1 \leq l \leq n}\{\abs{\tilde{\sigma}_i - \sigma^{\{l\}}_{i+1}} \leq \delta_i/2\}.$$ Observe that $\mathcal{B}_{1} \subset \mathcal{B}_{1, i-1} \cup \mathcal{B}_{1, i+1}$. We will show that \begin{equation}
\label{eq:Bii-1}
\mathcal{B}_{i, i-1} \subset \overline{\mathcal{G}_{i-1}} \cup \mathcal{B}_E, \text{ and }
\end{equation}\begin{equation}
\label{eq:Bii+1}
\mathcal{B}_{i, i+1} \subset \overline{\mathcal{G}_{i+1}} \cup \mathcal{B}_E.
\end{equation}
 Then, the union bound will imply the proposition. Recall that $\Delta_i = \sigma_{i} - \sigma_{i+1}$, and $\delta_i = \min(\Delta_{i-1}, \Delta_i)$. 
Since the proof of \eqref{eq:Bii-1} is virtually identical, we will only give the proof of \eqref{eq:Bii+1}. We break the analysis up into cases depending on if $\Delta_{i} > 4T$ or not.

{\it Case 1.}  $\Delta_{i} > 4T$. \newline Suppose $\overline{\mathcal{B}_E}$ holds. Let $1 \leq l \leq n$. By Fact \ref{interlacingfacts}, $\norm{E^{\{l\}}} \leq \norm{E}$, and on $\overline{\mathcal{B}_E}$, $\norm{E} \leq T$. Therefore, by applying Weyl's inequality to $\sigma^{\{l\}}_{i+1}$ and $\tilde{\sigma}_i$, we find that $$\abs{\tilde{\sigma}_i - \sigma^{\{l\}}_{i+1}} \ge \sigma_i - \sigma_{i+1} - 2T >  \Delta_{i}/2 \geq \delta_i/2.$$Therefore, $\overline{\mathcal{B}_E} \subset \overline{\mathcal{B}_{1, i+1}}$, so this implies that $\mathcal{B}_{1, i+1} \subset \mathcal{B}_E \subset \mathcal{B}_E \cup \overline{\mathcal{G}_{i+1}}$ in this case. 

{\it Case 2.} $\Delta_{i} \leq 4T$.\newline We will show that 
$\overline{\overline{\mathcal{G}_{i+1}} \cup \mathcal{B}_{E}} =  \mathcal{G}_{i+1} \cap \overline{\mathcal{B}_E} \subset \overline{\mathcal{B}_{i,i+1}}$, implying \eqref{eq:Bii+1}. Suppose $\mathcal{G}_{i+1} \cap \overline{\mathcal{B}_E}$ holds. Let $1 \leq l \leq n$.  Since $\sigma_i > 100T$ by $(c, \tau, 1)$ stability (and choice of c) and $\Delta_i \leq 4T$, $$\sigma_{i+1}  = \sigma_i - \Delta_{i} \geq \frac{9}{10}\sigma_i.$$ We can use Weyl's inequality to obtain $$\tilde{\sigma}_{i+1} \geq \sigma_{i+1} - \norm{E} \geq \frac{9}{10}\sigma_i - T \geq \frac{8}{10}\sigma_i.$$With this lower bound for $\tilde{\sigma}_{i+1}$, we can upper bound the right hand side of the inequality defining $\mathcal{G}_{i+1}$ for both $k = i$ and $k = i+1$. On $\mathcal{G}_{i+1} \cap \overline{\mathcal{B}_E}$,  we have that
\begin{equation}
\forall k \in \{i+1, i\}: \max\{\abs{\sigma^{\{l\}}_k - \sigma_k}, \abs{\tilde{\sigma}_k - \sigma_k}\}  \leq 32r\Big[K\sqrt{r(c_0 + 1)\log n}+ \frac{T^2}{\sigma_{i}} + \frac{T^{3}}{\sigma_{i}^{2}}\Big].
\end{equation} Since  $\sigma_{i} > 100T$, the third term on the right hand side of the above inequality is at most $1/100$ of the second term. This implies that on $\mathcal{G}_{i+1} \cap \overline{\mathcal{B}_E}$,

\ \begin{equation}
\forall k \in \{i+1, i\}: \max\{\abs{\sigma^{\{l\}}_k - \sigma_k}, \abs{\tilde{\sigma}_k - \sigma_k}\}  \leq 33r\Big[K\sqrt{r(c_0 + 1)\log n}+ \frac{T^2}{\sigma_{i}}\Big].
\end{equation} By the definition of $c$, it must be the case that $\Delta_i$ is much larger than the RHS, because  $\Delta_{i} > c(K\log^{1/2}n + \sigma_{i}^{-1}T^{2})$ by assumption $(b)$ of $(c, \tau, 1)$ stability.\begin{equation}
\forall k \in \{i+1, i\}: \max\{\abs{\sigma^{\{l\}}_k - \sigma_k}, \abs{\tilde{\sigma}_k - \sigma_k}\}  \leq \frac{\Delta_i}{4}.
\end{equation} Thus, on $\mathcal{G}_{i+1} \cap \overline{\mathcal{B}_E}$,both $\tilde{\sigma}_i$ and $\sigma^{\{l\}}_{i+1}$ are at most $\frac{\Delta_i}{4}$ away from their original values (which are $\sigma_i$ and $\sigma_{i+1})$, so the gap between them is at least $\Delta_i/2 \geq \delta_i/2$. Since this is true for all $l$, this implies that $\overline{\mathcal{B}_{i,i+1}}$ holds.
Thus, in either case, $\mathcal{B}_{1, i+1} \subset \overline{\mathcal{G}_{i+1}} \cup \mathcal{B}_{E}.$

As we have stated, a virtually identical argument using a case analysis for $\Delta_{i-1}$ gives that $\mathcal{B}_{1, i-1} \subset \overline{\mathcal{G}_{i-1}} \cup \mathcal{B}_E$. Therefore,  $$\mathcal{B}_1 \subset \overline{\mathcal{G}_{i-1}} \cup \overline{\mathcal{G}_{i+1}} \cup \mathcal{B}_E,$$ which implies that

\begin{equation}
\mathbb{P}(\mathcal{B}_1) \leq \mathbb{P}(\overline{\mathcal{G}_{i-1}}) + \mathbb{P}(\overline{\mathcal{G}_{i+1}}) + \mathbb{P}(\mathcal{B}_E). 
\end{equation}
\end{proof}

\section{ Proof of the Delocalization Lemma \ref{delocalization-result-1}}
\label{section: 
delocalizationproof1}

We now prove Lemma \ref{delocalization-result-1}, using the iterative leave-one-out argument, discussed briefly in Section \ref{mainideas}.

Recall the bound from Theorem \ref{coordinatetheorem} with $A$, $H = E$, on coordinate $l$:

\begin{equation}
        \abs{\tilde{u}_{il} - u_{il}} \le C_0\norm{U_{l, \cdot}}_{\infty}\Big[\kappa_i\norm{\tilde{\vect u}_i - \vect{u}_i}_2 + \epsilon_1(i) + a_l\kappa_i\epsilon_2(i)\Big] + 256r\frac{\abs{\langle \vect u^{\{l\}}_i, \vect x \rangle}}{\sigma_i},
\end{equation}  where we recall that $C_0$ is the value from Theorem \ref{coordinatetheorem}, $C_0 = 272\times4r^{3/2}$. The term in the brackets will be shown to be less than $\frac{5}{2}\kappa_i$. Further, $\norm{U_{l, \cdot}}_{\infty} \leq \norm{U}_{\infty}.$ We can thus write

\begin{equation}
\begin{split}
        \abs{\tilde{u}_{il} - u_{il}} &\le \frac{5}{2}C_0\kappa_i\norm{U}_{\infty}+ 256r\frac{\abs{\langle \vect u^{\{l\}}_i, \vect x \rangle}}{\sigma_i}, \text{ so } \\
        \abs{\tilde{u}_{il}} &\le 3C_0\kappa_i\norm{U}_{\infty}+ 256r\frac{\abs{\langle \vect u^{\{l\}}_i, \vect x \rangle}}{\sigma_i}.
\end{split}
\end{equation}

\vskip2mm 

The last term on the RHS is the inner product of the singular vector of a leave-one-out matrix with a random vector from which it is independent.  Appyling the Hoeffding inequality here is somewhat wasteful, as we can apply the stronger Bernstein inequality, given that we have 
control on the infinity norm of $\vect u^{\{l\}}_i$. Thus, the problem reduces to bounding the infinity norm of eigenvectors of a minor. On the surface, this makes the problem harder, as there are $n$ minors. But we observe  that the bound for the minor is slightly weaker than what we need for the whole matrix. This gain is critical and we are able to exploit it in a full iterative argument. The details now follow. 

{\it Notation.} Let $\alpha$ be an index set. Set $E^\alpha$ to be the random matrix equal to $E$, but with the rows and columns indexed by $\alpha$ set to zero. This is a generalization of the leave-one-out construction from Theorem \ref{coordinatetheorem}. If  $\abs{\alpha} = j$, we have a  leave-$j$-out matrix. Let $A^\alpha = A + E^\alpha$. $U^\alpha$ will be the matrix of $r$ leading singular vectors of $A^\alpha$. Similar notations, such as $\vect u^\alpha_i$ and $\sigma^\alpha_i$, are self-explanatory.

\vskip2mm 

Roughly speaking, the iterative leave-one-out argument uses a weaker bound on $\max_{\abs{\alpha} = j}\norm{\vect u^\alpha_i}_\infty$ to obtain a stronger bound for $\max_{\abs{\alpha} = j-1}\norm{\vect u^\alpha_i}_\infty.$ We will define a deterministic, increasing sequence $\{f_j\}_{0 \leq j \leq j^*}$, where $j^{*}$ is a number smaller than $\log n$. The sequence will be defined so that that $f_j^{*} = 1$, and $f_1$ and $f_0$ will be of size at most $4C_0\kappa_i\norm{U}_{\infty}$. We will show that under the $(c, \tau, 2)$ strong stability assumption, the leave-$j$-out singular vectors satisfy the following for $0 \leq j \leq j^*$.\begin{equation}\label{descent} \text{ For all } \abs{\alpha} = j, \norm{\vect u^\alpha_i}_\infty \leq f_{j}.\end{equation} Bounding the probability of failure of this statement is tricky.  It involves conditioning on what happens at step $j$ (leaving-$j$-out)  to control what happens at step $j-1$ (leaving-$(j-1)$-out). This iterative bound is proven in Lemma \ref{lemma:iteration}. The low probability of failure of \eqref{descent} for $j = 0$ will conclude the proof of Lemma \ref{delocalization-result-1}.

We now formally define our parameters. For $0 \leq j \leq j^{*} := \lceil \frac{50\log n}{\log(\log n)} + 3\rceil$, define $f_j$ in the following fashion. Start with $f_{j^*} = 1$. For $0 \leq j \leq j^*-1$, 
\begin{equation}
\label{f_jdefinition}
f_j := 3C_0\kappa_i\norm{U}_{\infty} + \frac{f_{j+1}}{\log^{0.01}n}.
\end{equation} By the choice of $j^*$ and the fact that $C_0 \ge 1$, it is easy to check that $f_0$ and $f_1$ are both less than $4C_0\kappa_i\norm{U}_{\infty}$.

\begin{lemma} [Iterative Lemma] \label{lemma:iteration}
For $0 \le j \le j^{\ast}$, define
$$\gamma_j := \mathbb {P}\{\max_{\alpha, |\alpha| = j }  \| \vect u_i ^{\alpha} \| _{\infty} >  f_j\}. $$ Then 
$\gamma_{j ^{\ast} } =0$. Further, for $1 \leq j \leq j^*$, $$\gamma_{j-1} \le  \gamma_j + 2\tau + \epsilon_j ,   $$  with $\epsilon_j= 66 \times 9^{2r}n^{j}\exp(-c_2\log^2 n)$, and $c_2 = c_0 + 1$.
\end{lemma}

The first conclusion that $\gamma_{j^{\ast} } =0$ is trivial, as a coordinate of a unit vector is at most 1, and we defined 
$f_{j^{\ast}} = 1$. The important content of this lemma is thus the iterative bound for the $\gamma_j$.

\begin{proof}[Proof of Lemma \ref{delocalization-result-1} given Lemma \ref{lemma:iteration}] We have the relation $$\sum_{j=1} ^{j^{\ast}-1 } \epsilon_{j} \leq n\epsilon_{j^* - 1} = \epsilon_{j^*}.$$ Lemma \ref{lemma:iteration} implies that \begin{equation}
\begin{split}
\gamma_0 &\le 2j^{*}\tau + \epsilon_{j^*} +  \sum_{j=1} ^{j^{\ast}-1 } \epsilon_{j} \le 2j^{*}\tau + 2\epsilon_{j^*}.
\end{split}
\end{equation} Restating this with the definition of $\gamma_0$ and $\epsilon_{j^*}$,

\begin{equation}
\begin{split}
\label{delocalization-proof-bound}
\mathbb{P}\{\norm{\tilde{\vect u}_i}_{\infty} > f_0\} &\leq 2j^{*}\tau +  132 \times 9^{2r} \times n^{j^*}\exp(-c_2\log^2 n) \\
&= 2j^*\tau + 132 \times 9^{2r} \times \exp((j^* - c_2\log n)\log n)\\
&\leq \tau\log n + 132\times 9^{2r}\exp (- \omega (\log n)) \\
&= \tau \log n + 132\times9^{2r}n^{-\omega (1) }\\
&\le \tau \log n + 132\times 9^{2r}n^{-c_0}
\end{split}
\end{equation} for any constant $c_0 >0$, thanks to  the fact that $j^{\ast} = O(\log n/ \log \log n) =o(\log n)$.
We have previously observed that $f_0 \leq 4C_0\kappa_i\norm{U}_{\infty}$; therefore \eqref{delocalization-proof-bound}  implies Lemma \ref{delocalization-result-1}.
\end{proof} 

\vskip2mm 

We now prove Lemma \ref{lemma:iteration}.

\subsection*{Preliminaries.}
Recall that for an index $l$, we defined $\vect x(l)$ as the $l$th row of $E$ with the $l$th entry divided by 2. We now let $\vect x(\alpha, l)$ be the $l$th row of $E^\alpha$ with its $l$th entry divided by $2$. By definition, any entry of $\vect x(\alpha, l)$ is either zero, an entry of $E$, or an entry of $E$ divided by $2$. In particular, the entries of $\vect x(\alpha, l)$ are mean zero, $K$-bounded, independent random variables. We consider  this vector because in the deterministic Theorem \ref{coordinatetheorem}, the $l$th row of $H$ plays an important role when we bound the perturbation of the $l$th coordinate of $\vect u_i$. We will apply Theorem \ref{coordinatetheorem} with $H = E^\alpha$.

The proof of Lemma \ref{lemma:iteration} requires bounding the probability of various failure events, which we now define.

\vskip2mm 
{\noindent \bf The event $\mathcal{B}_{\alpha, l}$.}\newline  Let $\alpha$ be an index set. Let $1 \leq l \leq n$, and set $\beta = \alpha \cup \{l\}$. Let 

\begin{equation}
\begin{split}
\mathcal{B}_{\alpha, l} &:= \mathcal{B}_{\alpha, l, 1} \cup \mathcal{B}_{\alpha, l, 2} \cup \mathcal{B}_{\alpha, E}, \text{ where } \\
\mathcal{B}_{\alpha, l, 1} &:= \{ \min\{\abs{\sigma^{\alpha}_i - \sigma^{\beta}_{i+1}}, \abs{\sigma^{\alpha}_i - \sigma^{\beta}_{i-1}}\} < \delta_i/2\}, \\
\mathcal{B}_{\alpha, l, 2} &:=  \Big\{\norm{U^{\beta T}\vect x(\alpha, l)}_{2} \geq K\sqrt{2c_2r\log^{2} n}\Big\}, 
\\
\mathcal{B}_{\alpha, E} &:= \{\norm{E^\alpha} > T\}. \\
\end{split}
\end{equation}
For $0 \leq j \leq j^*$, let
$$\mathcal{B}_j := \bigcup_{\substack{\abs{\alpha} = j\\ l \not\in \alpha}  }(\mathcal{B}_{\alpha, l, 1} \cup \mathcal{B}_{\alpha, l, 2}) \cup \bigcup_{\abs{\alpha} = j}B_{\alpha, E}.$$

\begin{lemma}[Probability of $\mathcal{B}_j$]
\label{probBj}
Let $0 \leq j \leq j^*$. Under the conditions of Lemma \ref{delocalization-result-1}, 
$$\mathbb{P}(\mathcal{B}_j) \leq 65\times9^{2r}n^{j+1}\exp(-c_2\log^2 n) + 2\tau.$$
\end{lemma}

\vspace{2mm}
{\noindent \bf The event $\mathcal{F}_{\alpha, l}$.}\newline For an index set $\alpha$ with $\abs{\alpha} = j$, and a coordinate $1 \leq l \leq n$, define
$$\mathcal{F}_{\alpha, l} := \{\abs{u^\alpha_{il}} > f_{j} \}.$$ We wish to show that the entries of $\vect u^\alpha_i$ are small. When $\mathcal{F}_{\alpha, l}$ holds, it means that a coordinate of $\vect u^\alpha_i$ is too big, and represents a failure at level $j$.

We will need the following lemma about the $\mathcal{F}_{\alpha, l}$ for those $\alpha, l$ where $l \in \alpha$. 

\begin{lemma}
\label{l-in-alpha}
\begin{equation}
\bigcup_{\substack{\abs{\alpha} = j\\ l \in \alpha} }\mathcal{F}_{\alpha, l} \subset \bigcup_{\substack{\abs{\alpha} = j}}\mathcal{B}_{\alpha, E}.
\end{equation}
\end{lemma} Having discussed how to control $\mathcal{F}_{\alpha, l}$ when $l \in \alpha$, we move to the case where $l \not\in \alpha.$ The remaining events are for controlling the probability of $\mathcal{F}_{\alpha, l}$ for such $\alpha, l$.

\vspace{2mm}
{\noindent \bf The event $\mathcal{K}_{\alpha, l}$.} \newline
For an index set $\alpha$ with $\abs{\alpha} = j$, and a coordinate $1 \leq l \leq n$, set $\beta = \alpha \cup \{l\}$. Define
\begin{equation}
\label{Kalphal-definition}
\mathcal{K}_{\alpha, l} := \Big\{\abs{u^\alpha_{il}} > 3C_0 \kappa_i \| U\| _{\infty} + 256 r \frac {\abs{\langle \vect u^\beta_i, \vect x(\alpha, l) \rangle}}{\sigma_i }\Big\}. 
\end{equation} If the success event $\overline{\mathcal{K}_{\alpha, l}}$ occurs, we can show that $\abs{u^\alpha_{il}}$ is small provided the inner product $\langle \vect u^\beta_i, \vect x(\alpha, l) \rangle$ is small. The following lemma shows that $\mathcal{K}_{\alpha, l}$ is unlikely.

\begin{lemma}[$\mathcal{K}_{\alpha, l} \subset \mathcal{B}_{\alpha, l}$ when $l 
\not\in \alpha$.]
\label{KinB}
Let $0 \leq j \leq j^{*}$. Let $\alpha$ be an index set such that $\abs{\alpha} = j$, and let $l \not\in \alpha$.  Then, 

\begin{equation}
\mathcal{K}_{\alpha, l} \subset \mathcal{B}_{\alpha, l}.
\end{equation}
\end{lemma}

If $l \not\in \alpha$, $\abs{\beta} = j+1$.  Showing that the inner product $\langle \vect u^\beta_i, \vect x(\alpha, l) \rangle$ is small (thus showing that $\abs{u^\alpha_{il}}$ is small) requires information about the infinity norm of $\vect u^\beta_i$. This information will be provided by the following events.

\vspace{2mm}
{\noindent \bf The events $\mathcal{L}_{\alpha, l}$ and $\mathcal{I}_{\alpha, l}$.} \newline
For an index set $\alpha$ with $\abs{\alpha} = j$, and a $1 \leq l \leq n$, let $\beta = \alpha \cup \{l\}$. Define $$\mathcal{L}_{\alpha, l} := \{\| \vect u_i^{\beta}  \| _{\infty } > f_{j+1}\}, $$
$$\mathcal{I}_{\alpha, l} := \{\abs{\langle \vect u^\beta_i, \vect x(\alpha, l) \rangle} \geq c_2\sqrt{2Kn}f_{j+1}\log^2 n\}, \text{ and }$$
$$\mathcal{I}_{j} := \bigcup_{\substack{\abs{\alpha} = j\\ l \not\in \alpha}  } (\mathcal{I}_{\alpha, l} \cap \overline{\mathcal{L}_{\alpha, l}}).$$

\begin{lemma}[Probability of $\mathcal{I}_j$]
\label{probIj}
Let $0 \leq j \leq j^*$. Under the conditions of Lemma \ref{delocalization-result-1}, 

\begin{equation}
\mathbb{P}(\mathcal{I}_j) \leq 2n^{j+1}\exp(-c_2\log^2n).
\end{equation}
\end{lemma}

\begin{proof}[Proof of Lemma \ref{lemma:iteration} given the lemmas]
We have previously observed that the statement is trivially true for $j = j^{*}$ because $\gamma_j^{*} = 0$. Having handled this, we move to the proof of the iterative bound. Let $1 \leq j \leq j^*$. Consider a set $\alpha$ with $j-1$ elements, which defines matrices $E^{\alpha}, A^{\alpha} = A + E^{\alpha}$. Let $l$ be a coordinate such that $l \not\in \alpha$, and let $\beta = \alpha \cup \{l \}. $ We aim to apply Theorem \ref{coordinatetheorem} for the pair $A, E^{\alpha}$ on coordinate $l$ with  $E^{\alpha} $ playing the role of $H$. The theorem obtains a bound on $\abs{u^\alpha_{il} - u_{il}}$. This bound implies that

\begin{equation}
\label{iterative-eq}
\abs{u^\alpha_{il}} \le 3C_0 \kappa_i \| U\| _{\infty} + 256 r \frac {\abs{\langle \vect u^\beta_i, \vect x(\alpha, l) \rangle}}{\sigma_i }. 
\end{equation} Looking at the second term on the RHS, if 

\begin{equation}
\label{inner-product-eq}
\abs{\langle \vect u^\beta_i, \vect x(\alpha, l) \rangle} \le c_2\sqrt{2Kn}f_{j}\log^2 n,
\end{equation} then the RHS of $\eqref{iterative-eq}$ is at most $f_{j-1}$. This is because by the $(c, \tau, 2)$ strong stability assumption, $\sigma_i > c\sqrt{Kn}\log^{2.01}n$. Since $c > 256rc_2\sqrt{2}$, \eqref{iterative-eq} and \eqref{inner-product-eq} imply that

\begin{equation}
\begin{split}
\abs{u^\alpha_{il}} &\le 3C_0 \kappa_i \| U\| _{\infty} + 256 r \frac {\abs{\langle \vect u^\beta_i, \vect x(\alpha, l) \rangle}}{\sigma_i }\\
&\le 3C_0 \kappa_i \| U\| _{\infty} + f_{j}\frac {256 rc_2\sqrt{2Kn}\log^2n}{\sigma_i }\\
&\leq 3C_0 \kappa_i \| U\| _{\infty} + \frac {f_j}{\log^{0.01}n} \\
&= f_{j-1},
\end{split}
\end{equation} and we thus have \begin{equation}
\label{target-F-alpha-l}
\abs{u^\alpha_{il}} \le f_{j-1}. 
\end{equation} The occurrence of \eqref{target-F-alpha-l} is the event we are interested in. Recall that $\mathcal{F}_{\alpha, l}$ is the event where \eqref{target-F-alpha-l} fails. What we have shown is that for those $\alpha, l$ such that $l \notin\alpha$, the failure probability of \eqref{target-F-alpha-l} and thus the bound for $\mathcal{F}_{\alpha, l}$ consists of two components. The first component is that the inequality \eqref{iterative-eq} does not hold, which we recall is the failure event $\mathcal{K}_{\alpha, l}$. By Lemma \ref{KinB}, $\mathcal{K}_{\alpha, l} \subset \mathcal{B}_{\alpha, l}$. The second component is that the inner product bound \eqref{inner-product-eq} fails. We called this failure event $\mathcal{I}_{\alpha, l}$. To analyze $\mathcal{I}_{\alpha, l}$, we will condition on the event that $\| \vect u_i^{\beta}  \| _{\infty } \le f_j$ fails. We called this failure event $\mathcal{L}_{\alpha, l}$. 
The union of the failure events gives for $\alpha, l$ satisfying $l \not\in \alpha$,
\begin{equation}
\begin{split}
\mathcal{F}_{\alpha, l} &\subset (\mathcal{K}_{\alpha, l}) \cup  (\mathcal{I}_{\alpha, l})  \\
 &\subset (\mathcal{B}_{\alpha, l}) \cup  (\mathcal{I}_{\alpha, l} \cap \mathcal{L}_{\alpha, l}) \cup (\mathcal{I}_{\alpha, l} \cap \overline{\mathcal{L}_{\alpha, l}}).
\end{split}
\end{equation} Obviously, the intersection of two sets is contained in both sets, so $\mathcal{I}_{\alpha, l} \cap \mathcal{L}_{\alpha, l} \subset \mathcal{L}_{\alpha, l}$. By taking the union of both sides over $\abs{\alpha} = j-1, l \not\in \alpha$,

$$\bigcup_{\substack{\abs{\alpha} = j-1\\ l \not\in \alpha}  } \mathcal{F}_{\alpha, l} \subset \bigcup_{\substack{\abs{\alpha} = j-1\\ l \not\in \alpha}  }(\mathcal{B}_{\alpha, l}) \cup  \bigcup_{\substack{\abs{\alpha} = j-1\\ l \not\in \alpha}  }(\mathcal{L}_{\alpha, l}) \cup \bigcup_{\substack{\abs{\alpha} = j-1\\ l \not\in\alpha}  } (\mathcal{I}_{\alpha, l} \cap \overline{\mathcal{L}_{\alpha, l}}).$$ For a given $\alpha$, we have only considered the coordinates $l$ such that $l \not\in \alpha$. In order to bound $\gamma_{j-1}$, we must also consider the other coordinates (the ones included in $\alpha$). This case is easy to handle, as by Lemma \ref{l-in-alpha}, 

\begin{equation}
\label{l-in-alpha-case}
\bigcup_{\substack{\abs{\alpha} = j-1\\ l \in \alpha} }\mathcal{F}_{\alpha, l} \subset \bigcup_{\abs{\alpha} = j-1 }\mathcal{B}_{\alpha, E}.
\end{equation} By definition of the events  $\mathcal{B}_j$,
$$\mathcal{B}_{j-1} = \bigcup_{\substack{\abs{\alpha} = j-1\\ l \not\in \alpha}  }(\mathcal{B}_{\alpha, l}) \cup \bigcup_{\abs{\alpha} = j-1 }\mathcal{B}_{\alpha, E}.$$ Therefore,

$$\bigcup_{\substack{\abs{\alpha} = j-1\\ 1 \leq l \leq n} } \mathcal{F}_{\alpha, l} \subset \mathcal{B}_{j-1} \cup  \bigcup_{\substack{\abs{\alpha} = j-1\\ l \not\in\alpha}  }(\mathcal{L}_{\alpha, l}) \cup \bigcup_{\substack{\abs{\alpha} = j-1\\ l \not\in\alpha}  } (\mathcal{I}_{\alpha, l} \cap \overline{\mathcal{L}_{\alpha, l}}).$$ By the definition of $\mathcal{I}_j$,
\begin{equation}\label{Fsplit}\bigcup_{\substack{\abs{\alpha} = j-1\\ 1 \leq l \leq n} } \mathcal{F}_{\alpha, l} \subset \mathcal{B}_{j-1} \cup \mathcal{I}_{j-1} \cup \bigcup_{\substack{\abs{\alpha} = j-1\\ l \not\in\alpha}  }\mathcal{L}_{\alpha, l}.
\end{equation} We make two observations about \eqref{Fsplit}. First, the event on the LHS of the inclusion is the event $$\bigcup_{\substack{\abs{\alpha} = j-1\\ 1 \leq l \leq n} } \mathcal{F}_{\alpha, l} = \bigcup_{\substack{\abs{\alpha} = j-1\\ 1 \leq l \leq n} }\{\abs{u^\alpha_{il}} > f_{j-1}\} = \{\max_{\alpha, |\alpha| = j-1 }  \| \vect u_i ^{\alpha} \| _{\infty} >  f_{j-1}\}.$$ Second, because we are considering $l \notin \alpha$, the union of the $\mathcal{L}_{\alpha, l}$ on the RHS is precisely   $$\bigcup_{\substack{\abs{\alpha} = j-1\\ l \not\in\alpha}  }\mathcal{L}_{\alpha, l} = \bigcup_{\substack{\abs{\alpha} = j-1\\ l \not\in\alpha}  }\Big\{\norm{\vect u^{\alpha \cup \{l\}}_i}_{\infty} > f_{j}\Big\} = \{\max_{\alpha, |\alpha| = j }  \| \vect u_i ^{\alpha} \| _{\infty} >  f_{j}\}.$$ Therefore, by the definition of the $\gamma_{j}$,

\begin{equation}
\mathbb{P}\Big(\bigcup_{\substack{\abs{\alpha} = j-1\\ 1 \leq l \leq n} } \mathcal{F}_{\alpha, l} \Big) = \gamma_{j-1}, \text{ and }
\mathbb{P}\Big(\bigcup_{\substack{\abs{\alpha} = j-1\\ l \not\in \alpha}  } \mathcal{L}_{\alpha, l}\Big) = \gamma_j.
\end{equation} By the union bound, \eqref{Fsplit} implies \begin{equation}
\label{gamma-inductive}
\gamma_{j-1} \leq \mathbb{P}(\mathcal{B}_{j-1}) + \mathbb{P}(\mathcal{I}_{j-1}) + \gamma_j.
\end{equation} By Lemma \ref{probBj}, \begin{equation}
\label{probBj-cite}
\mathbb{P}(\mathcal{B}_{j-1}) \leq 65 \times 9^{2r}n^{j}\exp(-c_2\log^2 n) + 2\tau.
\end{equation} By Lemma \ref{probIj}, 
\begin{equation}
\label{probIj-cite}
\mathbb{P}(\mathcal{I}_{j-1}) \leq 2n^{j}\exp(-c_2\log^2 n).
\end{equation} The inequalities \eqref{gamma-inductive}, \eqref{probBj-cite}, and \eqref{probIj-cite} imply the lemma, since $$\mathbb{P}(\mathcal{I}_{j-1}) + \mathbb{P}(\mathcal{B}_{j-1}) \leq 66\times 9^{2r}n^{j}\exp(-c_2\log^2n) + 2\tau.$$ 
\end{proof} 

\noindent In the next two sections, we  prove Lemmas \ref{probBj}, \ref{l-in-alpha}, \ref{KinB},  and \ref{probIj}.

\section{Proof of Lemmas \ref{l-in-alpha} and \ref{KinB}}
\label{section: delocalizationproof2}

To prove Lemma \ref{l-in-alpha}, we begin with a proposition, stated for deterministic $H$.

\begin{proposition}
\label{left-out-small}
Suppose $H^{\alpha}$ is a matrix equal to $H$, but whose rows and columns indexed by the index set $\alpha$ are set to zero. Suppose $l \in \alpha$, and that $\sigma_{i} > 2\norm{H^\alpha}$. Let $\vect u^{\alpha}_i$ be the $i$th singular vector of $A^{\alpha} = A + H^{\alpha}$.  Then

\begin{equation}\abs{u^{\alpha}_{il}} \leq 2\sqrt{r}\kappa_i\norm{U}_{\infty}.
\end{equation}
\end{proposition}

\begin{proof}[Proof of Lemma \ref{l-in-alpha} given the proposition]
Recall that we wish to show that 

\begin{equation}
\bigcup_{\substack{\abs{\alpha} = j\\ l \in \alpha} }\mathcal{F}_{\alpha, l} \subset \bigcup_{\substack{\abs{\alpha} = j} }\mathcal{B}_{\alpha, E}.
\end{equation}

Consider an index set $\alpha$ such that $\abs{\alpha} = j$, and an index $l$ such that $l \in \alpha$. We will show that $\overline{\mathcal{B}_{\alpha, E}} \subset \overline{\mathcal{F}_{\alpha, l}}$. Suppose $\overline{\mathcal{B}_{\alpha, E}}$ holds, which means $\norm{E^\alpha} \leq T$. We need to show that this implies $\abs{u^\alpha_{il}} \leq f_{j}$, which means that $\overline{\mathcal{F}_{\alpha, l}}$ holds. By the $(c, \tau, 2)$ strong stability assumption (see $(a)$ in Definition \ref{stable}), $\sigma_i > cT$, and $c$ is much larger than $2$. Therefore,
$$\sigma_i \geq cT \geq c\norm{E^\alpha} > 2\norm{E^\alpha}.$$ Since $l \in \alpha$, the conditions of Proposition \ref{left-out-small} are satisfied with the index set $\alpha$ and $H = E$.   Proposition \ref{left-out-small} implies that

\begin{equation}
\abs{u^{\alpha}_{il}} \leq 2\sqrt{r}\kappa_i\norm{U}_{\infty} \leq C_0\kappa_i\norm{U}_{\infty} \leq f_{j},
\end{equation} since $C_0 > 2\sqrt{r}$.
\end{proof}

\begin{proof}[Proof of Proposition \ref{left-out-small}]
Recall that $u^{\alpha}_{il}$ is the $l$th entry of a singular vector $\vect u^{\alpha}_i$ of $A^{\alpha}$, corresponding to a singular value $\sigma_i^{\alpha}$. By definition, we have 

$$ \abs{u^{\alpha}_{il}} = \abs{\vect e_{l}^{T}\vect u^{\alpha}_i} = \frac{1}{\sigma_i^{\alpha}}\abs{\vect e_l^TA^{\alpha}\vect u^{\alpha}_i}. $$


\noindent Since the $l$th row of $H^{\alpha}$ is zero, $A^{\alpha}$ and $A$ have the same $l$th row, which implies that we can replace $A^{\alpha}$ by $A$ to obtain $$ \abs{u^{\alpha}_{il}} =  \frac{1}{\sigma^\alpha_i}\abs{\vect e_l^TA \vect u^{\alpha}_i} .$$ By Fact \ref{weylfacts}, Fact \ref{interlacingfacts}, and the assumption of the proposition, it is easy to deduce that $\sigma_i^{\alpha} > \sigma_i/2$. Writing, $A =U\Sigma U^{T}$ using the spectral decomposition, we obtain 

$$ \abs{u^{\alpha}_{il}} \le 2  \sigma_i^{-1} \abs{\vect e_l^{T}U\Sigma U^{T}\vect u^{\alpha}_i} .$$

\noindent Notice that 

$$ \abs{\vect e_l^{T}U\Sigma U^{T}\vect u^{\alpha}_i} \le \| \vect e_l ^T U\|_2  \| \Sigma \| \| U^T \vect u_i^{\alpha} \| _2 \le  \sigma_1 \| U_{l,.}\|_2 , $$ 
This is because $\| \Sigma\| =\sigma_1$, $\vect e^T_l U = U_{l, .}$, the $l$th row of $U$, and $\| U \vect u_i^{\alpha} \|_2 \le 1$. 
This and the previous bound imply 

$$ \abs{u^{\alpha}_{il}} \le 2  \frac{\sigma_1} {\sigma_i } \| U_{l, .} \|_2 = 2 \kappa _i \| U_{l, .} \|_2 \leq 2\kappa_i\sqrt{r}\norm{U}_{\infty}.$$ where the last line uses that $\norm{U_{l, \cdot}}_2 \leq \sqrt{r}\norm{U}_{\infty}$ by the Cauchy-Schwarz inequality. 


\end{proof}

We move to the proof of Lemma \ref{KinB}. The following lemma is used to prove Lemma \ref{KinB} by showing that the conditions for Theorem \ref{coordinatetheorem} hold on the complement of $\mathcal{B}_{\alpha, l}$. We use Theorem \ref{coordinatetheorem} to show that $\mathcal{K}_{\alpha, l} \subset \mathcal{B}_{\alpha, l}$.

\begin{lemma}[Theorem \ref{coordinatetheorem} can be applied on $\overline{\mathcal{B}_{\alpha, l}}$]
\label{Tjinclusion}
Let $0 \leq j \leq j^{*}$. 
Let $\alpha$ be an index set satisfying $\abs{\alpha} = j$ and let $l$ be an index such that $l \not \in \alpha$. 
Under the conditions of Lemma \ref{delocalization-result-1}, if $\overline{\mathcal{B}_{\alpha, l}}$ occurs, the conditions for Theorem \ref{coordinatetheorem} hold with $A$, $H = E^{\alpha}$, and coordinate $l$.
\end{lemma}

\begin{proof}[Proof of Lemma \ref{KinB} given Lemma \ref{Tjinclusion}]
 Let $\alpha$ be an index set such that $\abs{\alpha} = j$, and let $l \not\in \alpha$. Recall that we wish to show that $\mathcal{K}_{\alpha, l} \subset \mathcal{B}_{\alpha, l}$. Set $\beta = \alpha \cup \{l\}$. Suppose $\overline{\mathcal{B}_{\alpha, l}}$ holds. We wish to show that $\overline{\mathcal{K}_{\alpha, l}}$ holds, which is equivalent to showing that 

\begin{equation}
\label{KinBbound}
\abs{u^\alpha_{il}} \le 3C_0 \kappa_i \| U\| _{\infty} + 256 r \frac {\abs{\langle \vect u^\beta_i, \vect x(\alpha, l) \rangle}}{\sigma_i }.
\end{equation} By Lemma \ref{Tjinclusion}, we can apply Theorem \ref{coordinatetheorem} on $\overline{\mathcal{B}_{\alpha, l}}$ with $A$ and $H = E^\alpha$ for coordinate $l$. This gives

\begin{equation}
\label{apply-coordinate-thm-alpha}
        \abs{u^{\alpha}_{il} - u_{il}} \le C_0\norm{U_{l, \cdot}}_{\infty}\Big[\kappa_i\norm{\vect u^{\alpha}_i - \vect{u}_i}_2 + \epsilon_1^{\alpha}(i) + a^{\alpha}_l\kappa_i\epsilon_2(i) \Big] + 256r\frac{\abs{\langle \vect u^{\beta}_i, \vect x(\alpha, l) \rangle}}{\sigma_i}, 
\end{equation} where $a^{\alpha}_l = \norm{U^{\beta T} \vect x(\alpha, l)}_2$ and $\epsilon^\alpha_1(i) = \frac{\norm{E^\alpha}}{\sigma_i}$. We will bound the term in the large brackets on the RHS and show that it is smaller than $\frac{5}{2}\kappa_i$. We will show that on $\overline{\mathcal{B}_{\alpha, l}}$, \begin{equation}
\label{epsilon-sum-goal}\epsilon_1^{\alpha}(i) + a^{\alpha}_l\kappa_i\epsilon_2(i) \leq \frac{\kappa_i}{2}.
\end{equation} Assume for now that this holds. By using the trivial bound  \begin{equation}
    \norm{\vect u^\alpha_i - \vect u_i}_2 \leq 2, 
\end{equation} \eqref{apply-coordinate-thm-alpha} and \eqref{epsilon-sum-goal} imply that \begin{equation}
        \abs{u^{\alpha}_{il} - u_{il}} \le \frac{5}{2}C_0\kappa_i\norm{U}_{\infty} + 256r\frac{\abs{\langle \vect u^{\beta}_i, \vect x(\alpha, l) \rangle}}{\sigma_i}. 
\end{equation} Moving $\abs{u_{il}}$, which is smaller than $\norm{U}_{\infty}$, to the right, we obtain
\begin{equation}
        \abs{u^{\alpha}_{il}} \le (\frac{5}{2}C_0 + 1)\kappa_i\norm{U}_{\infty} + 256r\frac{\abs{\langle \vect u^{\beta}_i, \vect x(\alpha, l) \rangle}}{\sigma_i}. 
\end{equation} Since $C_0 > 2$, we obtain \eqref{KinBbound}, proving the lemma.

Therefore, what remains is to verify \eqref{epsilon-sum-goal} on $\overline{\mathcal{B}_{\alpha, l}}$. Assume that $\overline{\mathcal{B}_{\alpha, l}}$ holds.
Then, by definition of $\mathcal{B}_{\alpha, E}$, we have $\norm{E^\alpha} \leq T$, which implies that

\begin{equation}\epsilon^\alpha_1(i) =  \frac{\norm{E^\alpha}}{\sigma_i} \leq \frac{T}{\sigma_i} < \frac{1}{4}.
\end{equation} The last inequality uses the $(c, \tau, 2)$ stability assumption, which gives $\sigma_i > cT > 4T$. For $a^{\alpha}_l\kappa_i\epsilon_2(i)$, recall that $\epsilon_2(i) = \frac{1}{\delta_i}$. Furthermore, on $\overline{\mathcal{B}_{\alpha, l}}$, the definition of $\mathcal{B}_{\alpha, l, 2}$ guarantees that $a^{\alpha}_l \leq K\sqrt{2rc_2\log^2 n}$. Therefore, \begin{equation}
    a^{\alpha}_l \epsilon_2(i) \leq \frac{K\sqrt{2rc_2\log^2 n}}{\delta_i}< \frac{1}{4}.
\end{equation} The last inequality uses $(c, \tau, 2)$ stability, which guarantees $\delta_i > cK\sqrt{\log^2 n}$, and the fact that $c > 4\sqrt{2rc_2}$. We conclude that \begin{equation}
\label{sum-eps-small}
\epsilon_1^{\alpha}(i) + a^{\alpha}_l\kappa_i\epsilon_2(i) \leq \frac{\kappa_i}{2}.
\end{equation} This verifies \eqref{epsilon-sum-goal} and thus completes the proof.
\end{proof}

\begin{proof}[Proof of Lemma \ref{Tjinclusion}]
Let $\alpha$ be an index set $\alpha$ with $\abs{\alpha} = j$, and let $l$ be an index such that $l \not\in \alpha$. Set $\beta = \alpha \cup \{l\}.$ Suppose that $\overline{\mathcal{B}_{\alpha, l}}$ holds.
Recall that we want to show that this implies that the conditions for Theorem \ref{coordinatetheorem} are satisfied with $H = E^\alpha$ and coordinate $l$.\newline
 Assume  $\overline{\mathcal{B}_{\alpha, l}}$ occurs. We first have to check that this implies that
 
 \begin{equation}
 \min\{\abs{\sigma^{\alpha}_i - \sigma^{\beta}_{i+1}}, \abs{\sigma^{\alpha}_i - \sigma^{\beta}_{i-1}}\} \ge \delta_i/2.
 \end{equation} It is clear that this is true by definition of $\mathcal{B}_{\alpha, l, 1}$. Then, we have to check that $\overline{\mathcal{B}_{\alpha, l}}$ implies \begin{equation}
    \sigma_i > C_0\norm{E^{\alpha}}.
 \end{equation} On $\overline{\mathcal{B}_{\alpha, l}}$, $\norm{E^{\alpha}} \leq T$ (see $\mathcal{B}_{\alpha, E})$. By $(c, \tau, 2)$ stability, $\sigma_i > cT$.  Therefore, $$\sigma_i > cT > C_0\norm{E^{\alpha}},$$ as desired. Finally, we verify that 
\begin{equation}
\label{eq:deltaverify}
\delta_i > C_0\max\Big\{\norm{U^{\beta T}\vect x(\alpha, l)}_2, \kappa_i\norm{E^{\alpha}}\norm{U}_{\infty}\Big\}.
\end{equation} On $\overline{\mathcal{B}_{\alpha, l}}$ (see $\mathcal{B}_{\alpha, l, 2}$), $$\norm{U^{\beta T}\vect x(\alpha, l)}_2 \leq K\sqrt{2rc_2\log^2 n}.$$ Because of the $(c, \tau, 2)$ stability assumption,$$\delta_i > c\max\{K\sqrt{\log^2 n}, T\kappa_i\norm{U}_{\infty}\},$$which ensures that \eqref{eq:deltaverify} holds, because $c > C_0\sqrt{2rc_2}$.
\end{proof}

\section{Proof of Lemmas \ref{probBj} and  \ref{probIj}}
\label{section: delocalizationproof3}

\begin{proof}[Proof of Lemma \ref{probBj}]
Let $0 \leq j \leq j^*$.
Recall that 
$$\mathcal{B}_j := \bigcup_{\substack{\abs{\alpha} = j\\ l \not\in \alpha}  }(\mathcal{B}_{\alpha, l, 1} \cup \mathcal{B}_{\alpha, l, 2}) \cup \bigcup_{\abs{\alpha} = j}B_{\alpha, E}.$$ We will continue to use $\beta = \beta(\alpha, l) =  \alpha \cup \{l\}$. In particular, $\beta$ depends on both $\alpha $ and $l$. We observe that $\mathcal{B}_j \subset \mathcal{B}_{j1} \cup \mathcal{B}_{j2} \cup \mathcal{B}_{jE}$, where

\begin{equation}
\begin{split}
\mathcal{B}_{j1} &:= \cup_{\abs{\alpha} = j}\cup_{l \not\in \alpha}\{ \min\{\abs{\sigma^{\alpha}_i - \sigma^{\beta}_{i+1}}, \abs{\sigma^{\alpha}_i - \sigma^{\beta}_{i-1}}\} < \delta_i/2\}, \\
\mathcal{B}_{j2} &:= \cup_{\abs{\alpha} = j}\cup_{l \not\in \alpha}\Big\{\norm{U^{\beta T}\vect x(\alpha, l)}_{2} \geq K\sqrt{2c_2r\log^{2} n}\Big\}, 
\text{ and }\\
\mathcal{B}_{jE} &:= \cup_{\abs{\alpha} = j}\{\norm{E^\alpha} > T\}.
\end{split}
\end{equation} We will bound the probabilities of these three events and use the union bound to conclude.

\vspace{2mm}
{\noindent \bf Probability of $\mathcal{B}_{jE}$.} By Fact $\ref{interlacingfacts}$, for all $\abs{\alpha} = j$, we have the deterministic bound $\norm{E^{\alpha}} \leq \norm{E}$. Therefore, defining 
$$\mathcal{B}_E := \{\norm{E} \geq T\}, \text{ we have } $$\begin{equation}
\label{BjEbound}
\mathbb{P}(\mathcal{B}_{jE}) \leq \mathbb{P}(\mathcal{B}_E) \leq \tau,
\end{equation} by definition of $T$. 

\vspace{2mm}
{\noindent \bf Probability of $\mathcal{B}_{j2}$.} Recall that

$$\mathcal{B}_{j2} = \cup_{\abs{\alpha} = j}\cup_{l \not\in \alpha}\Big\{ \norm{U^{\beta T}\vect x(\alpha, l)}_{2} \geq K\sqrt{2c_2r\log^2 n}\Big\}.$$ Let $\alpha$ be an index set such that $\abs{\alpha} = j$ and let $l \not\in \alpha$. Set $\beta = \alpha \cup \{l\}$. 

We use Hoeffding's inequality. Observe that $\vect x(\alpha, l)$ and $A^{\beta}$ are independent. Therefore, each entry of the length $r$ vector $U^{\beta T}\vect x(\alpha, l)$ is the inner product of a $K$ bounded random vector with independent entries (the vector $\vect x(\alpha, l)$) with a unit vector from which it is independent. By applying Corollary \ref{hoeffcor} to each entry, for $1 \leq k \leq r$,

\begin{equation}
\label{hoeffding-calc}
\mathbb{P}\{
\abs{[U^{\beta T}\vect x(\alpha, l)]_k} \geq K\sqrt{2c_2\log^2 n}\} \leq 2\exp(-c_2\log^2 n). 
\end{equation} Therefore, by taking the union bound over $1 \leq k \leq r$, 

\begin{equation}
\label{hoeffding-B2}
\mathbb{P}\Big\{\norm{U^{\beta T}\vect x(\alpha, l)}_2 \geq K\sqrt{2rc_2\log^2 n}\Big\} \leq 2r\exp(-c_2\log^2 n). 
\end{equation} By taking the union bound of \eqref{hoeffding-B2} over $\abs{\alpha} = j$, and $l \not \in \alpha$,

\begin{equation}
\label{Bj2bound}
\mathbb{P}[\mathcal{B}_{j2}] \leq 2rn^{j+1}\exp(-c_2\log^2n).
\end{equation}

\vspace{2mm}
{\noindent \bf Probability of $\mathcal{B}_{j1}$.}  Recall that
$$\mathcal{B}_{j1} = \cup_{\abs{\alpha} = j}\cup_{l \not\in \alpha}\{ \min\{\abs{\sigma^{\alpha}_i - \sigma^{\beta}_{i+1}}, \abs{\sigma^{\alpha}_i - \sigma^{\beta}_{i-1}}\} < \delta_i/2\}.$$ For an index set $\alpha$ such that $\abs{\alpha} = j$ and a coordinate $l \not\in \alpha$, set $\beta = \alpha \cup \{l\}$. Similar to the proof of Theorem \ref{main-result}, (with $\sigma_i^\alpha$ playing the role of $\tilde{\sigma}_i$ and $\sigma_i^\beta$ playing the role of $\sigma^{\{l\}}_i$), define the events $\mathcal{G}_{j, i+1}$ and $\mathcal{G}_{j, i-1}$ 

\begin{equation}
\begin{split}
\mathcal{G}_{j, i+1} &:=
    \bigcap_{\abs{\alpha} = j}\Bigg\{\max_{k = i, i+1} \max_{l \not \in \alpha}\{\abs{\sigma^{\beta}_k - \sigma_k}, \abs{\sigma^{\alpha}_k - \sigma_k}\} \leq 24r\Big[K\sqrt{rc_2\log^2 n}+ \frac{\norm{E}^2}{\tilde{\sigma}_k} + \frac{\norm{E}^{3}}{\tilde{\sigma}_{k}^{2}}\Big]\Bigg\}, \\
\mathcal{G}_{j, i-1} &:=
\bigcap_{\abs{\alpha} = j}\Bigg\{\max_{k = i, i-1} \max_{l\not\in \alpha}\{\abs{\sigma^{\beta}_k - \sigma_k}, \abs{\sigma^{\alpha}_k - \sigma_k}\} \leq 24r\Big[K\sqrt{rc_2\log^2 n}+ \frac{\norm{E}^2}{\tilde{\sigma}_k} + \frac{\norm{E}^{3}}{\tilde{\sigma}_{k}^{2}}\Big]\Bigg\}.
\end{split}
\end{equation} These events keep the relevant singular values of the perturbations $A^{\alpha}$ and $A^\beta$ for all $\alpha, l$ close to the original singular values to ensure the gap remains large. Recall that we applied Theorem \ref{coordinatetheorem} for $H = E^{\alpha}$ and $H^{\{l\}} = E^{\beta}$. The singular values of both $A^\alpha$ and $A^\beta$ are controlled by $\mathcal{G}_{j, i-1}$ and $\mathcal{G}_{j, i+1}$.

Using the same argument we employed in the proof of Proposition \ref{boundB1set}, it is straightforward to show using case analysis and $(c, \tau, 2)$ stability that
\begin{equation}
\label{splitBj1}
\mathbb{P}(\mathcal{B}_{j1}) \leq \mathbb{P}(\overline{\mathcal{G}_{j, i-1}}) + \mathbb{P}(\overline{\mathcal{G}_{j, i+1}})  + \mathbb{P}(\mathcal{B}_E), \end{equation} where we recall $\mathcal{B}_E = \{\norm{E} \geq T\}.$ $\mathcal{B}_E$ is an event which has probability at most $\tau$ by definition of $T$. Bounding the probabilities of $\overline{\mathcal{G}_{j, i-1}}$ and $\overline{\mathcal{G}_{j, i+1}}$ proceeds virtually identically to the proof of Lemma \ref{lemma:B1-bound}, so we omit the details of the calculation. The probabilities of $\overline{\mathcal{G}_{j, i-1}}$ and $\overline{\mathcal{G}_{j, i+1}}$ will be bounded using the result of \cite{OVW1}, which we recall in Theorem \ref{ovw-singular-K}, applied to the random perturbations $E^\alpha$ and $E^\beta$. Both have norm at most $\norm{E}$ by Fact \ref{interlacingfacts}. By Theorem \ref{ovw-singular-K} applied with $t = K\sqrt{128rc_2\log^2 n}$ to $E^{\alpha}$ and $E^\beta$, and the union bound over $\abs{\alpha} = j$, $l \not\in \alpha$,

\begin{equation}
\label{bound-B1}
\mathbb{P}(\overline{\mathcal{G}_{j, i-1}}) + \mathbb{P}(\overline{\mathcal{G}_{j, i+1}}) \leq 64 \times 9^{2r}n^{j+1}\exp(-c_2\log^2 n). 
\end{equation} Together with \eqref{splitBj1}, this implies \begin{equation}
\label{Bj1bound}
\mathbb{P}(\mathcal{B}_{j1}) \leq 64\times 9^{2r}n^{j+1}\exp(-c_2\log^2 n) + \tau.
\end{equation} We have now bounded $\mathcal{B}_{j1}, \mathcal{B}_{j2}$, and $\mathcal{B}_{jE}$. To conclude the proof, use the union bound and \eqref{BjEbound}, \eqref{Bj2bound}, and \eqref{Bj1bound}. This gives

$$\mathbb{P}(\mathcal{B}_j) \leq (64\times 9^{2r} + 2r)n^{j+1}\exp(-c_2\log^2 n) + \tau \leq 65\times9^{2r}n^{j+1}\exp(-c_2\log^2 n) + \tau.$$
\end{proof}

What remains is to prove Lemma \ref{probIj}, which bounds the probability of $\mathcal{I}_j$. To begin, we reproduce the definitions of the relevant events. Let $\alpha$ be an index set with $\abs{\alpha} = j$ and let $1 \leq l \leq n$. Recall that $\beta = \alpha \cup l$, and we defined the events

$$\mathcal{L}_{\alpha, l} = \Big\{\norm{\vect u^\beta_i}_{\infty} > f_{j+1}\Big\},$$

$$\mathcal{I}_{\alpha, l} = \{\langle \vect u^\beta_i, \vect x(\alpha, l) \rangle \geq c_2\sqrt{2Kn}f_{j+1}\log^2 n \}, \text{ and }$$

$$\mathcal{I}_{j} = \bigcup_{\substack{\abs{\alpha} = j\\l \not\in \alpha}  } (\mathcal{I}_{\alpha, l} \cap \overline{\mathcal{L}_{\alpha, l}}).$$

We will need the following lemma. We first introduce some notation. The notation $Y \in \overline{\mathcal{L}_{\alpha, l}}$ means that $Y$ is a possible realization of $E^\beta$ such that $\overline{\mathcal{L}_{\alpha, l}}$ holds. 

\begin{lemma}
\label{bernstein-bound}
Let $\alpha$ be an index set with $\abs{\alpha} = j$, and let $l$ be a coordinate such that $l \not\in \alpha$. Set $\beta = \alpha \cup \{l\}$. Let $Y \in \overline{\mathcal{L}_{\alpha, l}}$. Then, 

\begin{equation}
\begin{split}
\mathbb{P}(\mathcal{I}_{\alpha, l}\lvert E^{\beta} = Y) &\leq 2\exp(-c_2\log^2 n).
\end{split}
\end{equation}

\end{lemma}

\begin{proof}[Proof of Lemma \ref{probIj} given Lemma \ref{bernstein-bound}.] 
The goal is to bound the the probability of $$\mathcal{I}_j = \bigcup_{\substack{\abs{\alpha} = j-1\\ l \not\in \alpha}  } (\mathcal{I}_{\alpha, l} \cap \overline{\mathcal{L}_{\alpha, l}}).$$ By conditioning on $\overline{\mathcal{L}_{\alpha, l}}$,

\begin{equation}
\begin{split}
\mathbb{P}(\mathcal{I}_{\alpha, l} \cap \overline{\mathcal{L}_{\alpha, l}})  &= \mathbb{P}(\mathcal{I}_{\alpha, l} | \overline{\mathcal{L}_{\alpha, l}})\mathbb{P}(\overline{\mathcal{L}_{\alpha, l}}) \leq \mathbb{P}(\mathcal{I}_{\alpha, l} \lvert \overline{\mathcal{L}_{\alpha, l}}),
\end{split}
\end{equation} where we use the trivial bound $\mathbb{P}(\overline{\mathcal{L}_{\alpha, l}}) \leq 1$. We now bound the RHS to obtain \begin{equation}
\mathbb{P}(\mathcal{I}_{\alpha, l} \cap \overline{\mathcal{L}_{\alpha, l}}) \leq \mathbb{P}(\mathcal{I}_{\alpha, l} \lvert \overline{\mathcal{L}_{\alpha, l}}) \leq \sup_{Y \in \overline{\mathcal{L}_{\alpha, l}}}\mathbb{P}(\mathcal{I}_{\alpha, l} \lvert E^{\beta} = Y).
\end{equation} Lemma \ref{bernstein-bound} thus implies that

$$\mathbb{P}(\mathcal{I}_{\alpha, l} \cap \overline{\mathcal{L}_{\alpha, l}}) \leq 2\exp(-c_2\log^2n).$$
The union bound over $\abs{\alpha} = j$ and $l \not\in \alpha$ gives

\begin{equation}
\mathbb{P}(\mathcal{I}_{j}) \leq 2n^{j+1}\exp(-c_2\log^2 n).
\end{equation}
\end{proof}

\begin{proof}[Proof of Lemma \ref{bernstein-bound}]
Recall that we are considering index set $\alpha$ satisfying $\abs{\alpha} = j$ and a coordinate $l \not\in 
\alpha$. Set $\beta = \alpha \cup \{l\}$.
We are looking for a bound of 
$$\mathbb{P}(\mathcal{I}_{\alpha, l} | E^{\beta} = Y) = \mathbb{P}\Big\{\abs{\langle \vect u^\beta_i, \vect x(\alpha, l)\rangle} \geq c_2\sqrt{2Kn}f_{j+1}\log^2n \Big\lvert E^{\beta} = Y\Big\}r,$$ where $Y \in \overline{\mathcal{L}_{\alpha, l}}$. In other words, $Y$ is a possible realization of $E^{\beta}$ satisfying $\norm{\vect u^\beta_i}_{\infty} \leq f_{j+1}$. Recall that $\vect x(\alpha, l)$ is the $l$th row of $E^\alpha$ (with $l$th entry divided by 2). 
Conditional on $E^\beta$ equalling such a $Y$, $\vect u^\beta_i$ is a deterministic unit vector whose entries have absolute value at most $f_{j+1}$. The only randomness in each event thus comes from the vector $\vect x(\alpha, l)$. Let us name the entries of $\vect x(\alpha, l)$ as $x_k$. It follows that the inner product $\langle \vect u^\beta_i, \vect x(\alpha, l) \rangle$, conditional on $E^{\beta} = Y$,  is the sum of independent, $Kf_{j+1}$ bounded, mean zero random variables $x_ku^\beta_{ik}$. Thus, this quantity can be bounded with Bernstein's inequality (Lemma \ref{bernstein}). 

We are applying Bernstein's inequality conditionally, so we also need to find a bound for the sum of the conditional second moments of the $x_ku^\beta_{ik}$. Since $\vect u^\beta_i$ is deterministic when we condition on $E^\beta$, and $\vect x(\alpha, l)$ is independent of $E^\beta$, we have

\begin{equation}
\label{conditional-variance}
    \sum_{k = 1}^n \mathbb{E}\Big[x_k^2 u^{\beta 2}_{ik} \Big| E^\beta = Y\Big]  = \sum_{k = 1}^{n}u^{\beta 2}_{ik}\mathbb{E}[x_k^2] \leq K\sum_{k = 1}^{n}u^{\beta 2}_{ik} = K.
\end{equation}For the inequality, we use that the second moments of the entries of $E$ are at most $K$ by Assumption \ref{assumption-K}. The last equality uses the fact that $\vect u^\beta_i$ is a unit vector. Applying Bernstein's inequality then gives that 
\begin{equation}
\label{apply-bernstein}
\begin{split}
\mathbb{P}\Big\{\abs{\langle \vect u^{\beta}_i, \vect x(\alpha, l) \rangle} > t \Big| E^{\beta} = Y\Big\} &\leq 2\exp\Bigg(\frac{-t^{2}/2}{\sum_{k = 1}^{n} \E\Big[x^2_k u^{\beta 2}_{ik}\Big|E^{\beta} = Y\Big] + Kf_{j+1}t/3}\Bigg) \\
&\leq 2\exp\Big(\frac{-t^{2}/2}{K + Kf_{j+1}t/3}\Big).
\end{split}
\end{equation} Setting $t = c_2\sqrt{2Kn}f_{j+1}\log^2n$, 

\begin{equation}
\label{bernstein-calc}
\begin{split}
\mathbb{P}\Big\{\abs{\langle \vect u^{\beta}_i, \vect x(\alpha, l) \rangle} > c_2\sqrt{2Kn}f_{j+1}\log^2n \Big| E^{\beta} = Y\Big\} &\leq 2\exp\Bigg(\frac{-c_2^2Knf_{j+1}^{2}\log^4 n}{K + \frac{\sqrt{2}c_2}{3}Kf^2_{j+1}\sqrt{Kn}\log^2n}\Bigg). \\
\end{split}
\end{equation} We now upper bound the terms in the denominator. For the first term, since $n^{-1/2} \leq \norm{U}_{\infty} \leq f_{j+1}$, we obtain that $K \leq Knf_{j+1}^{2}$. For the second term, we use that $K \leq n$ by Assumption \ref{assumption-K}, so $\sqrt{Kn} \leq n$. This gives

\begin{equation}
\begin{split}
\mathbb{P}\Big\{\abs{\langle \vect u^{\beta}_i, \vect x(\alpha, l) \rangle} > c_2\sqrt{2Kn}f_{j+1}\log^2n \Big| E^{\beta} = Y\Big\}&\leq 2\exp\Bigg(\frac{-c_2^2Knf_{j+1}^{2}\log^4 n}{Knf_{j+1}^2 + \frac{\sqrt{2}c_2}{3}Kf_{j+1}^2n\log^2n}\Bigg).\\
\end{split}
\end{equation}

Since $c_2 = c_0 + 1 > 1$, we can upper bound the first term in the denominator on the RHS with $Knf_{j+1}^2 \leq \frac{c_2}{3}Knf_{j+1}^2\log^2n$. This gives

\begin{equation}
\begin{split}
\mathbb{P}\Big\{\abs{\langle \vect u^{\beta}_i, \vect x(\alpha, l) \rangle} > c_2\sqrt{2Kn}f_{j+1}\log^2n \Big| E^{\beta} = Y\Big\} &\leq 2\exp\Bigg(\frac{-c_2^2Knf_{j+1}^{2}\log^4 n}{(\frac{1 + \sqrt{2}}{3})c_2Knf_{j+1}^2\log^2n}\Bigg) \\
&\leq 2\exp(-c_2\log^2n).
\end{split}
\end{equation}
\end{proof}

\section{The proof of Theorem \ref{refined-main}}
\label{refined-proof}
The proof follows fairly easily from the proof of Lemma \ref{delocalization-result-1}. We only need to augment the events  to reduce the wasteful $\log^2 n$ terms in the bounds for inner products. Recall the quantities $U^{\{l\}}$, $\vect u^{\{l\}}_i$, and $\vect x(l)$ from the proof of Theorem \ref{main-result}. Define the events $$\mathcal{I}'_{l} := \{\abs{\langle \vect u_i^{\{l\}}, \vect x(l) \rangle} \geq c_2K\sqrt{2Kn}f_1\log n\}.$$ Let \begin{equation*}\mathcal{I}' := \bigcup_{1 \leq l \leq n} \mathcal{I}'_l. 
\end{equation*} Recall that in the proof of Lemma \ref{delocalization-result-1} we defined the events $\mathcal{B}_j$. We will be considering $\mathcal{B}_0$. Lastly, define $$\mathcal{B}' :=  \cup_{1 \leq l \leq n}\Big\{\norm{U^{\{l\}T} \vect x(l)}_2 \geq K\sqrt{2rc_2\log n}\Big\}.$$ We slightly abuse notation, as $\mathcal{B}'$ was defined and bounded in the proof of Theorem \ref{main-result} (under the name $\mathcal{B}_2$). Define the failure event $$\mathcal{F} = \mathcal{I}'  \cup \mathcal{B}'  \cup \mathcal{B}_0.$$
\begin{lemma}
\label{lemma:Fbound}
\begin{equation}
\mathbb{P}(\mathcal{F}) \leq \tau\log n + 150\times9^{2r}n^{-c_0}.
\end{equation}
\end{lemma}

\begin{proof}[Proof of Theorem \ref{refined-main} given Lemma \ref{lemma:Fbound}]
Let $1 \leq l \leq n$. By Lemma \ref{lemma:Fbound}, $\mathcal{F}$ has probability at most $\tau\log n + 150 \times n^{-c_0}$. By Lemma \ref{Tjinclusion}, the conditions for Theorem \ref{coordinatetheorem} hold on $\overline{F} \subset{\overline{B}_0}$ with $A$, $H = E$, for coordinate $l$. The theorem gives the bound \begin{equation}
\abs{\tilde{u}_{il} - u_{il}} \le C_0\norm{U_{l,\cdot}}_{\infty}\Big[\kappa_i\norm{\tilde{\vect u}_i - \vect{u}_i}_2 + \epsilon_1(i) + a_l\kappa_i\epsilon_2(i)\Big] + 256r\frac{\abs{\langle \vect u^{\{l\}}_i, \vect x(l) \rangle}}{\sigma_i},
\end{equation} where we recall that $a_l = \norm{U^{\{l\}T} \vect x(l)}_2$. On $\overline{\mathcal{F}}$, we have $a_l \leq K\sqrt{2rc_2\log n}$ because of $\mathcal{B}'$. Further, we have, by definition of $\mathcal{I}'$, \begin{equation}
\begin{split}
\abs{\langle \vect u_i^{\{l\}}, \vect x(l) \rangle} &\leq c_2K\sqrt{2Kn}f_1\log n \leq 4C_0c_2\sqrt{2Kn}\kappa_i\norm{U}_{\infty}\log n,
\end{split}
\end{equation} where we use our previous observation that $f_1 \leq 4C_0\kappa_i\norm{U}_{\infty}.$
By definition of $c$, we therefore have $$\abs{\langle \vect u_i^{\{l\}}, \vect x(l) \rangle}  < \frac{c}{256r}\kappa_i\norm{U}_{\infty}\sqrt{Kn}\log n.$$
It follows that
\begin{equation}
\abs{\tilde{u}_{il} - u_{il}} \le c\norm{U_{l, \cdot}}_{\infty}\Big[\kappa_i\norm{\tilde{\vect u}_i - \vect{u}_i}_2 + \epsilon_1(i) + \kappa_i\epsilon_2(i)K\sqrt{\log n}\Big] + \frac{c\kappa_i\norm{U}_{\infty}\sqrt{Kn}\log n}{\sigma_i}.
\end{equation} Since $l$ was arbitary, we have 
\begin{equation}
\norm{\tilde{\vect u}_i - \vect u_i}_{\infty} \le c\norm{U}_{\infty}\Big[\kappa_i\norm{\tilde{\vect u}_i - \vect{u}_i}_2 + \epsilon_1(i) + \kappa_i\epsilon_2(i)K\sqrt{\log n}\Big] + \frac{c\kappa_i\norm{U}_{\infty}\sqrt{Kn}\log n}{\sigma_i}.
\end{equation}
\end{proof}

\subsection{Proof of Lemma \ref{lemma:Fbound}}
We will make use of the fact that $c_2 = (1 + c_0)$, so $\exp(-c_2\log n) = n^{-c_0 - 1}$, for example. We first bound $\mathbb{P}(\mathcal{B}_0)$ and $\mathbb{P}(\mathcal{B}')$. By Lemma \ref{probBj}, 
$$\mathbb{P}(\mathcal{B}_j) \leq 65\times9^{2r}n^{j+1}\exp(-c_2\log^2 n) + 2\tau,$$ so we have
\begin{equation}
\label{B0bound}
\begin{split}
\mathbb{P}(\mathcal{B}_0) &\leq 65 \times 9^{2r}n\exp(-c_2\log^2 n) + 2\tau \\
&\leq 65\times  9^{2r}n^{-c_2\log n} + 2\tau \\
&\leq 65 \times 9^{2r}n^{-\log n}n^{-c_0} + 2\tau \\
&\leq 9^{2r}n^{-c_0} + 2\tau.
\end{split}
\end{equation} As we mentioned,  in the poof of Theorem \ref{main-result}, we bounded (see \eqref{bound2})
\begin{equation}
\label{B2again}
\mathbb{P}(\mathcal{B}') \leq 2rn^{-c_0}.
\end{equation}
We will bound $\mathbb{P}(\mathcal{I}')$ in the proceeding lemma. Recall that $j^{*} = \lceil \frac{50\log n}{\log(\log n)}\rceil + 3.$
\begin{lemma} 
\label{I'bound}
\begin{equation}
\mathbb{P}(\mathcal{I}') \leq 
 2j^{*}\tau + 133\times 9^{2r}n^{-c_0}.
\end{equation}
\end{lemma}
The union bound, Lemma \ref{I'bound}, \eqref{B0bound}, and \eqref{B2again} show that
 $$\mathbb{P}(\mathcal{F}) = \mathbb{P}(\mathcal{I}' \cup \mathcal{B}_0 \cup \mathcal{B}') \leq (134\times 9^{2r} + 2r)n^{-c_0} + (2j^* + 2)\tau.$$ Since $j^* = o(\log n)$, we have
$$\mathbb{P}(\mathcal{F}) \leq 150\times 9^{2r}n^{-c_0} + \tau\log n,$$ as desired (with room to spare).
To complete the proof of Lemma \ref{lemma:Fbound}, we must prove Lemma \ref{I'bound}. 

Before proceeding with the proof, we recall some notation. For an index set $\alpha$ of size $j$ and a coordinate $ 1 \leq l \leq n$, we defined the event $\mathcal{L}_{\alpha, l} = \{\| \vect u_i^{\alpha \cup \{l\}}  \| _{\infty } > f_{j+1}\}$ in the proof of Lemma \ref{lemma:iteration}. We will consider these events with $\alpha = \{\}$ in the following proposition.

\begin{proposition}
\label{prop:Jbound} Let
$$\mathcal{J} = \bigcup_{1 \leq l \leq n} (\mathcal{I}'_{l} \cap \overline{\mathcal{L}_{\{\}, l}}).$$ Then, 
\begin{equation}
\begin{split}
\mathbb{P}(\mathcal{J}) &\leq 2n\exp(-c_2\log n). 
\end{split}
\end{equation}

\end{proposition}
The proof of Proposition \ref{prop:Jbound} is a repetition of the computations in Lemma \ref{probIj}, using Bernstein's inequality conditionally. We place the details in Appendix \ref{appendix:refined-theorem-proof}. 
\begin{proof}[Proof of Lemma \ref{I'bound}]

We complete the task of bounding $\mathbb{P}(\mathcal{I}')$. By conditioning on $\mathcal{L}_{\{\}, l}$, we have \begin{equation}
\mathcal{I}'_{l} \subset \mathcal{L}_{\{\}, l} \cup (\mathcal{I}'_{l} \cap \overline{\mathcal{L}_{\{\}, l}}).
\end{equation} Therefore, 

\begin{equation}\label{I'split}\mathcal{I}' \subset  \bigcup_{1 \leq l \leq n} \mathcal{L}_{\{\}, l}  \cup \mathcal{J}.
\end{equation} 
The probability of $\mathcal{J}$ is bounded using the Proposition \ref{prop:Jbound}. \begin{equation}\label{equation: Jbound} \mathbb{P}(\mathcal{J}) \leq 2n\exp(-c_2\log n) = 2n^{-c_0}.\end{equation} Next, we recognize the other event on the RHS of \eqref{I'split}. \begin{equation} 
    \begin{split}
    \mathbb{P}\Big(\bigcup_{1 \leq l \leq n} \mathcal{L}_{\{\}, l}\Big) = \mathbb{P}
\{\max_{l} \norm{\vect u^{\{l\}}}_{\infty} > f_1\} &= \gamma_1. 
\end{split}
\end{equation} The last equality is the definition of $\gamma_1$. We bounded the $\gamma_j$ iteratively in Lemma \ref{lemma:iteration}. A routine calculation virtually identical to the one done in \eqref{delocalization-proof-bound} gives the following bound for $\gamma_1$.
\begin{equation}\label{equation:Lemptybound}
\mathbb{P}\Big(\bigcup_{1 \leq l \leq n} \mathcal{L}_{\{\}, l}\Big) = \gamma_1 \leq 2j^*\tau  + 132 \times 9^{2r} n^{-c_0}.
\end{equation} The bounds \eqref{equation: Jbound} and \eqref{equation:Lemptybound} show that $$\mathbb{P}(\mathcal{I}') \leq (132 \times 9^{2r} + 2)n^{-c_0} + 2j^{*}\tau \leq 133 \times 9^{2r}n^{-c_0} + 2j^{*}\tau,$$ as desired.

\end{proof}
\vspace{5mm}
\section{Application: A simple algorithm for clustering problems} 
\label{clustering}









A number of clustering problems have the following common form. A
vertex set $V$ is partitioned into $r$ subsets $V_1, \dots V_r$, and between each pair $V_i, V_j$ we draw edges independently with probability $p_{ij} $  (we allow $i=j$). The task is to find 
a particular set $V_j$ or all the parts $V_1, \dots, V_r$  given one instance of the random graph \cite{drineas2004, fernandez, Feige2005SpectralTA, feige2000, feige, mcsherry2001, alon,  blumspencer, Condon1999AlgorithmsFG}.  

The most popular approach to this problem is spectral, which typically 
consists of two steps. In the first step, one considers the coordinates of a singular vector of the adjacency matrix of the graph (or more generally the projection of the row vectors of the adjacency matrix onto a low dimensional singular space), and run a standard clustering algorithm in low dimension. The output of this step is an approximation of the truth. In the second step, one applies adhoc combinatorial techniques to clean up the output, in order to recover the mis-classified  vertices. 

It has been conjectured that in many cases, the cleaning step is not necessary. Our result shows that it is indeed the case. The critical point here is that the existence of misclassified vertices, in many settings, is just an artifact of the analysis in the first step, which typically relies on $\ell_2$ norm estimates. 
Notice that any $\ell_2$ norm estimate, even sharp, could only imply that a majority of the vertices are well classified, and this leads to the necessity of the second step. Once we have a strong $\ell_{\infty}$  norm estimate, 
then we would be able to classify all the vertices at once.

As we stated in Section \ref{applications}, our new infinity norm estimates enable us to overcome the shortcomings in clustering algorithms that rely on $\ell_2$ analysis in a number of settings. This results in fast and simple new algorithms for a wide variety of problems. All matrices in this section are positive semi-definite, so 
there is no difference between singular vectors and eigenvectors.

\subsection{The hidden clique problem} 

 The (simplest form) of the hidden clique problem is the following: Hide a clique $X$ of size $k$ in  the random graph 
 $G(n,1/2)$. Can we find $X$ in polynomial time? 

 Notice that the largest clique in $G(n,1/2)$, with overwhelming probability, has size approximately 
$2 \log n$ \cite{ASbook}. Thus, for any $k$ bigger than $(2+\epsilon) \log n$, 
with any constant $\epsilon >0$, X would be abnormally large and therefore detectable, by brute-force at least.  For instance, one can check all vertex sets of size $k$ to see if any of them form a clique. However, finding $X$ in polynomial time is a different matter, and the best current bound for $k$ is $k \ge c \sqrt n$, for any constant $c >0$. This was first  achieved by Alon, Krivelevich, and Sudakov  \cite{alon}; see also 
 \cite{feige}\cite{dekel} for later developments concerning faster algorithms for certain values of $c$. 

The Alon-Krivelevich-Sudakov algorithm runs as follows. It first finds $X$ when $c$ is sufficiently large, then  uses a simple sampling  trick to 
reduce the case of small $c$ to this case. 

To find the clique for a large $c$, they first compute the second eigenvector of the adjacency matrix of the graph and  locate  the first largest $k$  coordinates in absolute value. Call this set $Y$. This is an approximation of the clique $X$, but not yet totally accurate. The second, combinatorial, step is to define the set $X$ as the  vertices in the graph with at least $3/4k$ neighbors in $Y$. The authors then proved that with high probability, $X$ is indeed the hidden clique. 

With our new results,  we can find $X$ immediately  by a slightly modified version of the first step, omitting the second combinatorial step. Before starting the main step of the algorithm, we change all zeros in the adjacency matrix to $-1$.

 \vskip2mm 

\begin{algorithm}[First singular vector clustering-FSC]

Compute the first singular vector. Let $x$ be the largest value of the coordinates and let $X$ be the set of all coordinates with  value at least $x/2$.
    
\end{algorithm}

\vskip2mm 
This is perhaps the simplest algorithm for this problem, as 
computing the first singular vector of a large matrix 
is a routine operation that appears in all standard numerical linear algebra packages.

 \begin{theorem} \label{hiddenC} 
There is a constant $c_0$ such that for all $k \ge c_0 \sqrt n$, FSC outputs the hidden clique correctly with probability at least $.99$.

 \end{theorem}

{\it Proof of Theorem \ref{hiddenC} } After the switching of zeroes to minus ones, the adjacency matrix $\tilde A$ has an all one block of size $k$ (corresponding the hidden cliques), and the rest are $\pm 1$ bits. For convenience, we assume that the all-one block is at the left-top corner.  Thus, we can write $\tilde A= A +E$, where $A$ has an all-one block on its leading principal sub-matrix of size $k$ and the rest of the entries are zero. $E$ is a random matrix with $\pm 1$ entries with a zero block of size $k$.

Notice that the matrix $A$ has rank 1, with $\sigma_1=k$, and first singular vector   $$(1/\sqrt k, \dots, 1/\sqrt k, 0, 0, \dots, 0). $$  So the large (non-zero) entries of this singular vector reveals the position of the vertices of the clique. The algorithm  computes the leading singular vector of $\tilde A$, and we are going to show that the large entries of this vector, with high probability,  still correspond to 
the vertices of the clique. 

From Theorem \ref{OVW1eigenvector}, it is easy to see that with probability at least $.99$,
the $\ell_2$ error $\norm{\tilde{\vect u}_1 - \vect{u}_1}_2$ is bounded by $O( \sqrt n/k) $. Results from random matrix theory show that $\norm{E}$ is at most $3 \sqrt n$ \cite{Vu2005SpectralNO}, with probability $1-o(1)$.
 Thus, our Theorem \ref{main-result} implies that with probability at least $.99$, the  infinity norm bound between the first singular vector of $A$ and that of $\tilde A$ is 

 $$ O( k^{-1/2}\sqrt {n} / k) \le  k^{-1/2} /4, $$ given that $k /\sqrt n $ is sufficiently large (beating the hidden constant in the big $O$). 
Thus, in the leading singular vector, the entries from the clique are at least $\frac{3}{4} 
k^{-1/2}$, and the rest are at most $\frac{1}{4} k^{-1/2}$ in absolute value. This guarantees that the clustering described in the algorithm reveals all the vertices of $X$. 

While we have made no effort to optimize the value of $c_0$ (indeed, our theoretical constant is quite large), it is an interesting question to determine the values of $c_0$ for which FSC can recover the hidden clique exactly. The optimal $c_0$ is quite small; it is likely close to $1$. See Figure \ref{hiddencliqueplot}.


\begin{remark} \label{onehalf} 
The density $1/2$ is not critical, and can be replaced by any parameter 
$p > n^{-1 +\epsilon} $ (or even $p > n^{-1} \log^c n$, for some 
properly chosen $c$). 
In the case of $p$, one needs to replace a zero entry  by $-p/(1-p)$. The random matrix $E$ now has zero mean and spectral norm at most 
$3 \sqrt {np}$; see again \cite{Vu2005SpectralNO}. Thus, by following our argument, we can see that it is sufficient to assume $k \ge C \sqrt {np} $ for a sufficiently large constant $C$. 

\end{remark} 

\begin{figure}
    \centering
    \includegraphics[scale = 0.30]{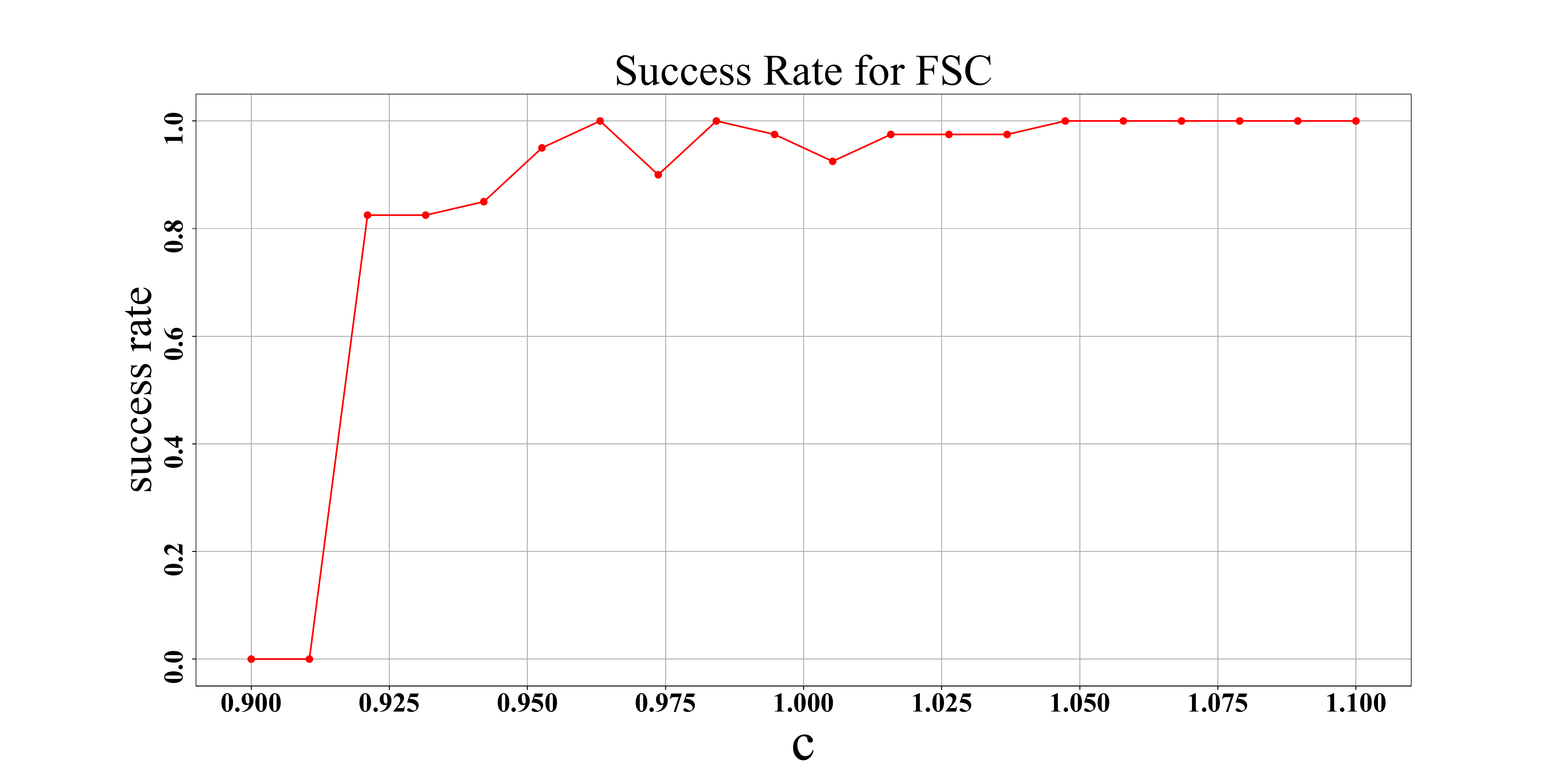}
    \caption[Numerical results for FSC.]{Numerical results for FSC on the Hidden Clique problem. For each $c$, 40 trials of the hidden clique problem are run with $n = 1000$. The step of switching $0$ to $-1$ is omitted. The random graph in which the clique is embedded is $G(n, 1/2)$. Reported are the fraction of trials for which the clique was recovered exactly with no mis-classifications of any vertex.}
    \label{hiddencliqueplot}
\end{figure}

\subsection{Clique partition} \label{hiddencliques}

Let us consider the situation where  one hides many cliques $X_1, \dots, X_r$ of size $k_1 \ge k_2 \ge  \dots \ge  k_r$, which form a partition of the vertex set.  The vertices of different cliques are connected with probability $1/2$, independently. The first  task is to find the $i$th largest clique $X_i$, for any given $1 \le i \le r$, given one instance of the random graph. 

One can do this by finding all $X_i$ and then sorting  them out. However, we can do the task directly by just computing the $i$th singular vector and clustering on its coordinates (the same way as in the last section). Before starting the main step of the algorithm , we change all zeros in the adjacency matrix to $-1$.

 \vskip2mm 
\begin{algorithm}[$i$th clique]
\label{ithclique}
Compute the $i$th singular vector. Let $x$ be the largest  value of a coordinate and let $X_i$ be the set of all coordinates with  value at least $x/2$.
\end{algorithm}

\vskip2mm

One issue is that if
$k_i = k_{i-1} $, then there is no way to differentiate $X_i$ from $X_{i-1}$. 
Thus, the hard instances for the problem are when $|k_i- k_j |$ are small in general. In what follows, we concentrate on that case, and assume that  all $k_i$ are of order $n$. Our theorems enable us to find $X_i$ correctly under the 
assumption that $|k_i - k_{i \pm 1} | = \tilde O( 1) $. 

\begin{theorem} \label{clique1} For any constant $c >0$ there 
is a constant $C$ such that the following holds. 
Assume that $k_r \ge cn $  and 
$k_{i} -k_{i+1} \ge C \log n$  for all $1\le i \le r$. Then with probability at least $.99$, Algorithm \ref{ithclique} recovers $X_i$ correctly, for any $1 \le i \le r$. 
    
\end{theorem}

In what follows, we illustrate the ideas through the case $i=1$. The analysis for a general $i$ is similar. Consider the leading eigenvector of $\tilde A =A +E$, where $A$ now consists of $r$ disjoint diagonal all-one blocks of sizes $k_1, \dots, k_r$. 

\begin{equation*}A= 
\begin{bmatrix}
\mathbbm{1}_{k_{1}}\mathbbm{1}_{k_{1}}^{T} & 0 & \dots & 0 \\
0 & \mathbbm{1}_{k_{2}}\mathbbm{1}_{k_{2}}^{T} & 0 & \vdots  \\
\vdots & 0 & \ddots  & 0\\
0 & \dots  & 0 & \mathbbm{1}_{k_{r}}\mathbbm{1}_{k_{r}}^{T} \\
\end{bmatrix}
\end{equation*}

The leading eigenvalue of $A$ is $k_1$ and the leading eigenvector is $$(1/\sqrt{k_1}, \dots, 1/\sqrt{k_1},0, \dots, 0). $$ 
The next eigenvalue is $k_2$ and the gap 
$\delta_1=k_1 -k_2$. Since the cliques partition the vertex set, $k_1 \ge n/r$. The difference (compared to the previous section) is that, with probability $.99$,  the $\ell_2$ error (from Theorem \ref{OVW1eigenvector})
is now bounded by 
$$O\Big[ \frac{1}{\delta_1}  + \frac{\| E\| }{\sigma} + 
\frac{\| E\|^2}{\delta_1\sigma} \Big] = O \Big[ \frac{1}{k_1-k_2} + 
\frac{\sqrt {n} }{k_1}  + 
\frac{n}{ k_1 (k_1-k_2)}  \Big] = O\Big[\frac{1}{ k_1-k_2} \Big].$$

Since we assume that all $k_i= \Theta (n)$, the infinity norm of $U$ is 
$O(\frac{1}{\sqrt n}) $.
So, our Theorem \ref{main-result}  implies
that with probability at least $.99$, the infinity norm bound for the first eigenvector is 


$$O\Big[\frac{r\sqrt{\log n}}{ (k_1-k_2)\sqrt {n}} + \frac{\log n} {n} \Big] \le \frac{1}{4 k_1^{1/2} } ,$$  given that $(k_1-k_2)/\sqrt{\log n} $ is bounded from below by a sufficiently large constant, proving the claim.

\vskip2mm 

We found a simple, but effective, trick to reduce the general  case (when  the separation 
condition  could be violated, such as when $k_i$ are all the same) to the situation in  Theorem \ref{clique1}.
We call this trick {\it random truncation} and it works as follows. 

\vskip2mm 

\noindent {\it Random truncation.} Select each vertex with probability $\rho :=n^{-1 +\epsilon}$, independently, where $\epsilon $ is a small positive constant. Let $S$ be the set of selected vertices and 
$V'= V \backslash S, X_i'= X_i \backslash S, k_i'= |X_i'|$. If $k_i =\Theta (n)$
then 
$|X_i \cap S|$ is a binomial random variable $\chi$ with mean $k_i \rho= \Theta (n\rho)$ and standard deviation $ \Theta (\sqrt {n\rho}) = \Theta  (n^{\epsilon/2}) $. 
Since $\log n = o( n^{\epsilon/2})$, the following fact is obvious. 

\begin{fact} \label{separation} (Separation Lemma) 
Consider the random variable $\chi$ above and let $\chi'$ be its independent copy. Then for any given interval $I$ of length $O(\log n)$, with probability at least $1-o(1)$,  $\chi - \chi' \not\in I$.
\end{fact}

By the union bound over all pairs $(i,j)$, it 
follows that with probability at least $1-o(1)$, 
$\min_ {i \neq j} |k_i' -k_j '| = \omega (\log n) $. Thus, our separation condition holds on the subgraph spanned by $V'$. We can now run our algorithm on the adjacency matrix of this graph to identify $X_i'$. To finish, define $X_i$ as the union of $X_i'$ with  the vertices in $S$ which are connected to all  the vertices in $X_i'$.

 \vskip2mm
 
 Another natural task is to find all $X_i$, and we  can complete this task by 
consecutive applications of the algorithm FSC from the last section. First, find $X_1$ (or more precisely $X_1'$), then remove it from the graph. Then, find $X_2$ and continue in this way. Here is the formal description of the algorithm. 

\begin{algorithm}[Clique partition]
\begin{enumerate}
    \item[]
    \item Define a set $S$ by choosing each vertex in $S$ with probability $\rho:= n^{-1 +\epsilon}$.
Let $V'= V\backslash S$ and consider the graph spanned by $V'$.
    \item For $i=1, \dots, r-1$, run FSC to get $X_i'$. Let $X'_r= V' \backslash \cup _{i=1}^{r-1}V_i$. 

    \item Define $X_i$ be the union of $X_i'$ and the vertices of $S$ which are adjacent to all of $X_i$. 
\end{enumerate}
\end{algorithm}

\begin{theorem} Assume that $k_i =\Theta (n)$ for all $1 \le i \le r$. 
With probability at least $.9$, the Algorithm Clique Partition recovers $X_1, \dots, X_n$ correctly. 

\end{theorem} 

\subsection{Planted colorings} 

Finding a $r$ coloring of a graph is a notoriously hard problem, even when we
know that the graph is $r$-colorable. A number of researchers have considered  the random instance of this problem. One natural  setting is as follows. Partition the vertex set $V$  into $r$ independent sets 
$X_1, \dots, X_r$ of sizes $k_1 \ge k_2 \dots \ge k_r$ and then connect the vertices between different $X_i$  with probability $1/2$. The task is to recover the proper coloring from one instance of this random graph; see for instance  \cite{alonkahale97}, \cite{blumspencer}.

Notice that if we look at the complement graph, then 
this is exactly the problem considered in the previous section, as independent sets become cliques. Thus, we obtain 

\begin{theorem} 
Assume that $k_i =\Theta (n)$ for all $1 \le i \le r$. Then with probability at least $.9$,  the algorithm in the last section recovers the planted coloring. \end{theorem}

\begin{remark} \label{onehalf2}
In this and the previous problems, the constant $1/2$ again is not important, and can be replaced by a general density $p$. 
The condition that $k_i =\Theta (n)$ for all $i$ can also be weakened. 

\end{remark}

\subsection{Hidden partition} 
We now consider a generalization of the problem in Section \ref{hiddencliques}, where each clique
$X_i$ is replaced by a random graph with edge density $p_i >1/2$. Similar to Section \ref{hiddencliques}, the task  is to locate a particular  $X_i$ or all $X_i$ from one random instance of the graph.

Switch all $0$ in the adjacency matrix to $-1$. The resulting matrix $\tilde A$
can be decomposed into $A+E$, where  $A$ now has the following form 

$$A= 
\begin{bmatrix}
(2p_1-1)\mathbbm{1}_{k_{1}}\mathbbm{1}_{k_{1}}^{T} & 0 & \dots & 0 \\
0 & (2p_2-1)\mathbbm{1}_{k_{2}}\mathbbm{1}_{k_{2}}^{T} & 0 & \vdots  \\
\vdots & 0 & \ddots  & 0\\
0 & \dots  & 0 & (2p_{r}-1)\mathbbm{1}_{k_{r}}\mathbbm{1}_{k_{r}}^{T} \\
\end{bmatrix}
$$

The random matrix $E$ has the following form. The entry $e_{ij}$, $1 \le i \le j \le n$, 
is Rademacher ($\pm 1$) if $i$ and $j$ belong to different $X_k$.  If they belong to the same $X_k$, then let $e_{ij}= (2-2p_k)$ with probability $p_k$ and $-2p_k$ with probability $1-p_k$. It is easy to check that $\tilde A =A+E $ and 
that all entries of $E$ have zero-mean and are $2$-bounded. 

Set $\rho_i :=2p_i-1$. The singular values of $A$ are $k_1\rho_1, \dots, k_r \rho_r$. 
We replace the assumption  $k_1 \ge k_2 \ge  \dots \ge  k_r$ by its weighted version   $k_1\rho_1 \ge  k_2 \rho_2 \ge  \dots \ge k_r \rho_r$. 
The leading singular value of $A$ is now $k_1\rho_1$, the second is $k_2\rho_2$ and the gap is $\delta_1= k_1 \rho_1 -k_2\rho_2$;  the singular vector remains the same. If we follow the proof of Theorem \ref{clique1}, then the condition becomes 

$$ \frac{1}{k_i^{1/2} } \frac{r}{\delta_i} + \frac{\sqrt{\log n}} { k_i \rho_1} \le C^{-1}  \frac{1}{k_i^{1/2}} $$ for some sufficiently large constant $C$. This is equivalent to assuming that both
$\delta_i$ and $\sqrt {k_i} \rho_i /\sqrt {\log n } $ are lower bounded by some sufficiently large constant $C$. The second one is equivalent to 
$\rho_i \ge C \sqrt {\log n/n } $ for some sufficiently large constant $C$.

\begin{theorem}\label{part1}  For any constant $c_1$, there is a constant $C$ such that the following holds. Assume that $k_1 \rho_1/ k_r \rho_r $ is bounded from above  by $c_1$.  If  $\delta_i   \ge C \log n$ and $\rho_i \ge C \sqrt {\log n/n } $ then the $i$th clique algorithm recovers $X_i$ correctly with probability $.9$. 
\end{theorem}

We can again apply the random truncation trick at the beginning to guarantee the separation condition. 
However, the application of this trick on this more general setting is slightly 
more technical than in the case of cliques, since it is less obvious how to assign the vertices from $S$ to $X_i'$. To decide which $X_i'$ a vertex $v \in S$ belongs to, we first choose a subset $Y_i \subset X_i'$ so that all $Y_i$ has the same size  $c n$ for some constant $c>0$. (This is doable because we assume that all $X_i$ have size $\Theta (n)$.) Let $i$ be the index where $v$ has the most edges connected to $Y_i$ and then add $v$ to $X_i'$. 

Notice that if $v \not\in X_i$, then the number of edges between $v$ and $Y_i$
has distribution $\chi_0= \Binom(.5, cn)$. If $v \in X_i$ then it has 
distribution $\chi_i= \Binom (p_i, cn)$. If 
$p_i -.5:= \rho_i/2 > C_0 \sqrt {n \log n}$ for a sufficiently large constant $C_0$
(which may depend on $c$ and $r$, then with probability at least $.99$, 
$\chi_i \ge \chi_0 $ for  all $v \in S$ and $1 \le i \le r$. This leads us to the following algorithm and theorem.

\vskip2mm 

\begin{algorithm}[Hidden partition]
\label{alghiddenpartition}
\begin{enumerate}
\item[]
\item Define a set $S$ by choosing each vertex in $S$ with probability $\rho:= n^{-1 +\epsilon}$.
Let $V'= V\backslash S$ and consider the graph spanned by $V'$.

\item For $i=1, \dots, r-1$, run FSC to get $X_i'$. Let $X'_r= V' \backslash \cup _{i=1}^{r-1}V_i$.

\item Select subsets $Y_i \subset X_i'$ such that $|Y_i| =c n$, for some properly chosen small constant $c>0$. 

\item Define $X_i$ be the union of $X_i'$ and those  vertices $v$ of $S$ where $d_i(v)= \max _j d_j (v)$, where $d_i(v)$ is the number of neighbors of $v$ in $Y_i$ (break ties arbitrarily). 
\end{enumerate}
\end{algorithm}


\begin{theorem}\label{part2}  For any constant $c$ there are
constants $c_0, C$ such that the following holds.  If  $k_i \ge c n$ and $ \rho_i \ge C \sqrt { \log n/n } $ for all $1\le i \le r$, then Algorithm Hidden Partition recovers all $X_i$ correctly with probability $.9$. 
\end{theorem}

\begin{remark} \label{onehalf3}
The density  $1/2$ again is not important, and can be replaced by a general density $q$. 
 
 \end{remark} 

Another well-known instance of this problem is the hidden bipartition problem. In this problem,  $r=2$ and the vertex set is partitioned into two sets of equal size $n/2$. Draw edges with probability $p$ inside $X_i$ and $q< p$ between $X_1$ and $X_2$.  The task is to recover the partition from one instance of the random graph. This particular case has been studied heavily by many researchers through 4 decades; see Table \ref{surveytable}.

\begin{table}[!h]
\resizebox{\textwidth}{!}{%
\begin{tabular}{||c | c | c ||} 
\hline
 \hline
 Bui, Chaudhuri, Leighton, Sipser '84 \cite{BCLS84} & min-cut method &  $p = \Omega(1/n), q = o(n^{-1 - 4/((p+q)n)})$ \\ 
 \hline
 Dyer, Frieze '89 \cite{DF89} & min-cut via degrees & $p-q = \Omega(1)$\\
 \hline
 Boppana '87 \cite{BP87} & spectral method &  $(p-q)/\sqrt{p+q} = \Omega(\sqrt{\log(n)/n})$ \\
 \hline
 Snijders, Nowicki '97 \cite{SN97} & EM algorithm & $p - q = \Omega(1)$ \\
 \hline
 Jerrum, Sorkin '98 \cite{JS98} & Metropolis algorithm & $p - q = \Omega(n^{-1/6 + \epsilon})$\\
 \hline
 Condon, Karp '99 \cite{Condon1999AlgorithmsFG} & augmentation algorithm & $p - q = \Omega(n^{-1/2 + \epsilon})$\\
 \hline
 Carson, Impagliazzo '01 \cite{CI01} & hill-climbing algorithm &  $p - q = \Omega(n^{-1/2}\log^{4}n)$ \\
 \hline
 Mcsherry '01 \cite{mcsherry2001} & spectral method & $(p - q)/\sqrt{p} \ge \Omega(\sqrt{\log(n)/n})$ \\
 \hline
 Bickel, Chen '09 \cite{BC09} & N-G modularity & $(p-q)/\sqrt{p+q} = \Omega(\log(n)/\sqrt{n})$\\ 
 \hline
 Rohe, Chatterjee, Yu '11 \cite{RCY11} & spectral method & $p - q = \Omega(1)$ \\
 \hline
 Abbe, Bandeira, Hull '14 \cite{abbe2014exact} & maximum likelihood & $p - q = \Omega(\sqrt{\log(n)/n})$\\
 [1ex] 
 \hline
 Vu '18 \cite{vuhidden} & spectral method & $(p - q)/p^{1/2} = \Omega(\sqrt{\log(n)/n})$ \\
 \hline
 Abbe, Fan, Wang, Zhong '19 \cite{abbefanentrywise} & spectral method & $p - q = \Omega(\sqrt{\log(n)/n})$ \\
 [1ex] 
 \hline
 \hline
  \end{tabular}}
\caption[A survey of the hidden bipartition problem.]{A recreation of the table in \cite{abbe2014exact} surveying the hidden bipartition problem, with some recent additions.}
\label{surveytable}
\end{table}

In this case, our method (with some obvious modifications to replace $1/2$ by $q$) gives

\begin{theorem}\label{part3}  If $p,q  =\Theta (1)$,
$ p-q \ge C \sqrt { \log n/n } $ for all $1\le i \le r$, then 
Algorithm Hidden Partition recovers all $X_i$ correctly with probability $.9$. 
\end{theorem}

The lower bound $\sqrt {\log n/n}$ is the current best on this problem; see \cite{abbe2014exact}. From our analysis, it is clear that the same conclusion holds for equal partitions with any number of parts (more than 2). The condition $p,q= \Theta (1)$ can also be improved, and with $p$ tending to zero, it becomes  $ (p-q)/\sqrt{p} \ge C \sqrt {\log n/n } $. 

\vspace{5mm}
\section{Application: Exact Matrix Completion from Few Entries}
\label{completion}

In this section, we prove  Theorem \ref{threshould-round-theorem}. 
Let us recall that $A$ is an integer matrix with rank $r = O(1)$, with entries bounded by an absolute constant (so $\norm{A}_{\infty} = O(1)$). Let  $S$ be the sampled version of $A$ where each entry is sampled (independently) with probability $p$, and the un-sampled entries are zeroed out. Then $\tilde{A} := p^{-1}S$ is an unbiased estimate of $A$. We analyze the simple spectral algorithm
in Section \ref{completion} to recover $A$ exactly from $\tilde A$. Set $W = [U, V]$, the concatenated matrix of $U$ and $V$, where $U$ and $V$ are the left and right singular vectors of $A$ respectively. Recall that  we set $s := \max_i \{i: \sigma_i \ge  \frac{1}{16r}\| W\|_{\infty}^{-2}\}$, $\overline{\delta} = \inf_{i \leq s}\delta_i$, and $\tilde{s} = \max_i\{i: \tilde{\sigma}_i \geq \frac{1}{8r}\norm{W}_{\infty}^{-2}\}$. Let $a = \max\{\sqrt{\norm{A}_{\infty}}, 2\}.$ For concreteness, set $c$ to be the constant $(a + 1) \times 2^{18}\times 7r^3$. 

By the rounding step, in order to have exact recovery, we need to show that $$\norm{A - B}_{\infty} < \frac{1}{2}. $$Let us consider  the entry $A_{12}$.  By the singular value decompositions of  $A$ and $B$, 

$$ A_{12} = \sum_{i = 1}^{r}\sigma_i u_{i1}v_{i2}, \hskip2mm B_{12} = \sum_{i = 1}^{\tilde s}\tilde{\sigma}_i \tilde{u}_{i1}\tilde{v}_{i2} . $$   Let $A_{\tilde s}$ be the best rank $\tilde s$ approximation of $A$. Thus, $A= A_{\tilde s}+ \sum_{i=\tilde s+1}^r \sigma_i u_i u_i^T $. Write $A_{\tilde{s}ij}$ for the $ij$ entry of $A_{\tilde s}$. By the triangle inequality, we have \begin{equation}
\label{MCsplit}
\begin{split}
\lvert B_{12} - A_{12}\rvert  &\leq \abs{B_{12} - A_{\tilde{s}12}} + \abs{\sum_{i = \tilde{s}+1}^{r}\sigma_{i}u_{i1}v_{i2}}\\
&= \Big\lvert \sum_{i = 1}^{\tilde s}\tilde{\sigma}_i\tilde{u}_{i1}\tilde{v}_{i2} -  
\sum_{i = 1}^{\tilde s}\sigma_i u_{i1}v_{i2} \Big\rvert + \abs{\sum_{i = \tilde{s}+1}^{r}\sigma_{i}u_{i1}v_{i2}}.
\end{split}
\end{equation} Bounding the second term on the RHS is fairly straightforward, and is handled in Lemma \ref{MC:secondterm}. The choice of the threshold $\tilde{s}$ plays an important role. Before stating the lemma, we will need the following tail bound on the norm of $E$. Its proof relies on a result from \cite{bandeiravanhandel} and is given in Appendix \ref{appendix:normE-MC}. This lemma will be used throughout this section.

\begin{lemma}\label{lemma:normE-MC} There exists an absolute constant $C$ such that 
\begin{equation}\label{normE-mc}\mathbb{P}\{\norm{E} \geq C\sqrt{Np^{-1}}\} \leq N^{-3}.
\end{equation}
\end{lemma}
For convenience, we define
$$\mathcal{E} := \{\norm{E} \leq C\sqrt{Np^{-1}}\},$$ where $C$ is the constant above.

\begin{lemma}\label{MC:secondterm} With probability at least $1 - N^{-3}$,
$$ \abs{\sum_{i = \tilde{s}+1}^{r}\sigma_{i}u_{i1}v_{i2}} \leq \frac{1}{4}. $$
\end{lemma}In the next lemma, we bound the first term on the RHS of \eqref{MCsplit}. This is where we use our refined bounds for the large $K$ case. As we have mentioned,  unlike the clustering problem, $K$ can be quite large in the matrix completion setting.

\begin{lemma}\label{MC:firstterm} Under the conditions of Theorem \ref{threshould-round-theorem}, with probability at least $1 - 3N^{-2}$, 
$$\abs{B_{12} - A_{\tilde{s}12}} < \frac{1}{4}.$$ 
\end{lemma}
Lemmas \ref{MC:secondterm} and \ref{MC:firstterm} applied to the RHS of \eqref{MCsplit} are enough to conclude the result of Theorem \ref{threshould-round-theorem}, because then the error will be strictly less than $\frac{1}{4} + \frac{1}{4} = \frac{1}{2}$, and will thus be rounded away. The choice of considering entry $(1, 2)$ is arbitary; with probability $1 - N^{-1}$, the bound holds for all of the entries. The remainder of the section is dedicated to proving Lemmas \ref{MC:secondterm} and \ref{MC:firstterm}. 
\begin{proof} [Proof of Lemma \ref{MC:secondterm}]
Suppose that $\mathcal{E}$, which has probability $1 - N^{-3}$ by Lemma \ref{lemma:normE-MC}, occurs. Recall that $\tilde{s} = \max_i\{i: \tilde{\sigma}_i \geq \frac{\norm{W}_{\infty}^{-2}}{8r}\}$. Therefore, for all $i \geq \tilde{s}+1$, it must be the case that $\tilde{\sigma}_{i} \leq \frac{1}{8r}\norm{W}_{\infty}^{-2}.$ By Fact \ref{weylfacts} and Lemma \ref{lemma:normE-MC}, for $i \geq \tilde{s} + 1$,
$$\sigma_i \leq \tilde{\sigma}_i + \norm{E} \leq \frac{\norm{W}_{\infty}^2}{8r}  + C\sqrt{Np^{-1}}.$$

By the density assumption, $p > N^{-1}\log^{4.03}N$. By the incoherence assumption, $\norm{W}_{\infty} \leq cN^{-1/2}$. Therefore, $\sqrt{Np^{-1}} = o\Big(\norm{W}_{\infty}^{-2}\Big)$. It follows that for $i \geq \tilde{s} + 1$, 

$$\sigma_i \leq \frac{\norm{W}_{\infty}^2}{4r}.$$
We have

\begin{equation}
\begin{split}
\abs{\sum_{i = \tilde{s}+1}^{r}\sigma_{i}u_{i1}v_{i2}} & \leq \sum_{i = \tilde{s}+1}^{r}\sigma_{i}\abs{u_{i1}}\abs{v_{i2}}\\
&\leq \frac{\norm{W}_{\infty}^{-2}}{4r}\sum_{i = \tilde{s}+1}^r \abs{u_{i1}}\abs{v_{i2}}.
\end{split}
\end{equation} Because both $\abs{u_{i1}} \leq \norm{W}_{\infty}$ and $\abs{v_{il}} \leq \norm{W}_{\infty}$, it follows that \begin{equation}
\begin{split}
\abs{\sum_{i = \tilde{s}+1}^{r}\sigma_{i}u_{i1}v_{i2}} \leq  \frac{1}{4r}\sum_{i = \tilde{s}+1}^r 1\leq \frac{1}{4}.
\end{split}
\end{equation}
\end{proof} In order to establish Lemma \ref{MC:firstterm}, we need the following proposition and lemma. 

\begin{proposition}\label{prop:tildeslessthans}
Suppose that $\mathcal{E}$ occurs. Then, $\tilde{s} \leq s$. 
\end{proposition}

\begin{proof}
 Suppose $\mathcal{E}$ occurs. Assume towards contradiction that $\tilde{s} > s$. By definition of $\tilde{s}$, this means that there are at least $s + 1$ singular values of $\tilde{A}$ larger than $\frac{\norm{W}_{\infty}^{-2}}{8r}$, so \begin{equation}\label{to-contradiction}\tilde{\sigma}_{s + 1} \geq  \frac{\norm{W}_{\infty}^{-2}}{8r}.\end{equation} By definition of $s$, $\sigma_{s + 1} \leq \frac{\norm{W}_{\infty}^{-2}}{16r}$. Therefore, by Fact \ref{weylfacts},

$$\tilde{\sigma}_{s+1} \leq \sigma_{s + 1} + \norm{E} \leq \frac{\norm{W}_{\infty}^{-2}}{16r} + C\sqrt{Np^{-1}} = \frac{\norm{W}_{\infty}^{-2}}{16r} + o(\norm{W}_{\infty}^{-2}) < \frac{\norm{W}_{\infty}^{-2}}{8r}.$$

where we use our previous observation that $\sqrt{Np^{-1}} = o(\norm{W}_{\infty}^{-2})$. This is a contradiction with \eqref{to-contradiction}.
\end{proof}

\begin{lemma}
\label{m-infty-bound}
Set $m_{\infty}(i) = \max\{\norm{\tilde{\vect u}_i - \vect u_i}_{\infty}, \norm{\tilde{\vect v}_i - \vect v_i}_{\infty}\}$ and $m_{2}(i) = \max\{\norm{\tilde{\vect u}_i - \vect u_i}_{2}, \norm{\tilde{\vect v}_i - \vect v_i}_{2}\}$. Under the conditions of Theorem \ref{threshould-round-theorem}, there exists a constant $c_0$ such that with probability at least $1 - 2N^{-2}$,\begin{equation}
   \sup_{i \leq s} m_{\infty}(i) \leq \frac{c_0\norm{W}_{\infty}}{\log N}.
\end{equation}
\end{lemma}

\begin{proof}[Proof of Lemma \ref{MC:firstterm} given Lemma \ref{m-infty-bound}]

Recall that 

$$\abs{B_{12} - A_{\tilde{s}12}} = \Big\lvert \sum_{i = 1}^{\tilde s}\tilde{\sigma}_i\tilde{u}_{i1}\tilde{v}_{i2} -  
\sum_{i = 1}^{\tilde s}\sigma_i u_{i1}v_{i2} \Big\rvert.$$ Letting $\tilde{u}_{i1}= u_{i1} + \Delta u_{i1}$ and $\tilde{v}_{i2}= v_{i2} + \Delta v_{i2}$, we have \begin{equation}
\label{hat-A-expansion}
B_{12} = \sum_{i = 1}^{\tilde s}\tilde{\sigma}_{i}[u_{i1}v_{i2} + u_{i1}\Delta v_{i2} + v_{i2}\Delta u_{i1} + \Delta u_{i1} \Delta v_{i2}].
\end{equation}  Let $c_0$ be the constant from Lemma \ref{m-infty-bound}. Suppose $\mathcal{E}$ and $\{\sup_{i \leq s} m_{\infty}(i) \leq \frac{c_0}{\log N}\}$ both occur. By the lemma and the union bound, this happens with probability at least $1 - 3N^{-2}$. Since $\mathcal{E}$ occurs, $\tilde{s} \leq s$ by Proposition \ref{prop:tildeslessthans}. It follows that 
\begin{equation}
\label{AB-split}
\abs{B_{12} - A_{\tilde{s}12}} \leq \sum_{i = 1}^{s}\abs{\tilde{\sigma}_{i} - \sigma_i}\abs{u_{i1}v_{i2}} + d,
\end{equation} where

$$d := \sum_{i = 1}^{s}\tilde{\sigma}_i[\abs{u_{i1}}\abs{\Delta v_{i2}} + \abs{v_{i2}}\abs{\Delta u_{i1}} + \abs{\Delta u_{i1}}\abs{\Delta v_{i2}}]. $$ Observe that both $\abs{\Delta u_{i1}}$ and $\abs{\Delta v_{i2}}$ are bounded by  $m_{\infty}(i)$. We consider the two terms on the RHS of \eqref{AB-split} separately. For the first term, using Fact \ref{weylfacts}, $\abs{\tilde{\sigma}_i - \sigma_i} \leq \norm{E}$. Since $\abs{u_{i1}}\abs{v_{i2}} \leq \norm{W}_{\infty}^2$, we have \begin{equation}
\sum_{i = 1}^{s}\abs{\tilde{\sigma}_{i} - \sigma_i}\abs{u_{i1}v_{i2}} \leq r\norm{E}\norm{W}_{\infty}^2\end{equation}
By the incoherence assumption, $\norm{W}_{\infty} = O(N^{-1/2})$ so $\norm{W}^2 = O(N^{-1})$. Since $\mathcal{E}$ occurs, $$\norm{E} \leq C\sqrt{Np^{-1}} = o(N),$$ where the last equality uses the assumed lower bound for $p$, $p > \frac{\log^{4.03} N}{N}$. It follows that $\norm{E}\norm{W}_{\infty}^2 = o(1)$.
\begin{equation}
\label{mc-bound-sum1}
\sum_{i = 1}^{s}\abs{\tilde{\sigma}_{i} - \sigma_i}\abs{u_{i1}v_{i2}} = o(1).
\end{equation}

Moving to the term $d$ on the RHS of \eqref{AB-split}, we bound for $i \leq s$, $\tilde{\sigma_i} \leq \sigma_i + \norm{E}$ by Fact \ref{weylfacts}. By the signal-to-noise condition and the fact that $\mathcal{E}$ occurs, $\sigma_s > \norm{E}$. Thus, $\sigma_i + \norm{E} \leq 2\sigma_i$. Therefore, 

\begin{equation}
\begin{split}
d&=\sum_{i = 1}^{s}\tilde{\sigma}_i[\abs{u_{i1}}\abs{\Delta v_{i2}} + \abs{v_{i2}}\abs{\Delta u_{i1}} + \abs{\Delta u_{i1}}\abs{\Delta v_{i2}}]\\
&\leq 2\sum_{i = 1}^{s}\sigma_i[\abs{u_{i1}}\abs{\Delta v_{i2}} + \abs{v_{i2}}\abs{\Delta u_{i1}} + \abs{\Delta u_{i1}}\abs{\Delta v_{i2}}].
\end{split}
\end{equation} Then, using the bounds $\abs{\Delta u_{i1}}, \abs{\Delta v_{i1}} \leq m_{\infty}(i)$ and $\abs{u_{i1}}, \abs{v_{i1}} \leq \norm{W}_{\infty}$, we have 

\begin{equation}
\begin{split}
d&\leq 2\sum_{i = 1}^{s}\sigma_i[2\norm{W}_{\infty}m_{\infty}(i) + m_{\infty}^2(i)].
\end{split}
\end{equation} Because $A$ has rank $r = O(1)$ and has $O(1)$ bounded entries,  $\sigma_1 = O(N)$. Since $\norm{W}_{\infty} = O(N^{-1/2})$ by the incoherence assumption and $\sup_{i \leq s} m_{\infty}(i) = o(\norm{W}_{\infty})$,

\begin{equation}
\label{mc-bound-sum2}
\begin{split}
d = O(\sigma_1\norm{W}_{\infty}\sup_{i \leq s}m_{\infty}(i))
& = o(N\norm{W}_{\infty}^2) = o(1).
\end{split}
\end{equation}  \eqref{mc-bound-sum1} and \eqref{mc-bound-sum2} give that 

\begin{equation}
\abs{B_{12} - A_{\tilde{s}12}}  = o(1) <  \frac{1}{4},
\end{equation}
for a large enough $N$.
\end{proof}


In order to prove Lemma \ref{m-infty-bound}, first establish that we can apply our refined (large $K$ case) results for the $\ell_{\infty}$ perturbation of singular vectors because the strong stability condition holds. 

\begin{lemma}
\label{stability-mc} Recall that $a$ is the absolute constant  $a = \max\{\sqrt{\norm{A}_{\infty}}, 2\}.$ Under the conditions of Theorem \ref{threshould-round-theorem}, for all $i \leq s$, the singular values and gaps $(\sigma_i, \delta_i)$ are all $(\frac{c}{a+1}, N^{-3}, 2)$ strongly stable under $E$.
\end{lemma}

\begin{proof}[Proof of Lemma \ref{m-infty-bound} given Lemma \ref{stability-mc}]
Since the entries of $A$ are $O(1)$, $K = O(p^{-1})$. By Lemma \ref{stability-mc} and the choice of $c$, the conditions for Theorem \ref{rectangular2} to bound $m_{\infty}(i)$  hold with $\frac{c}{a+1}$, $\tau = N^{-3}$, $\nu = 2$, and $K = O(p^{-1})$. Applying this for all $i \leq s$, using the fact that $\bar{\delta}$ is the minimum $\delta_i$ among the first $s$ singular values,

\begin{equation}
\label{infty-completion}
        \sup_{i \leq s} m_{\infty}(i) = O\Bigg(\kappa_s\norm{W}_{\infty}\Big[\sup_{i \leq s}m_2(i) + \frac{\norm{E}}{\sigma_s} + \frac{\sqrt{\log N}}{p\overline{\delta}}\Big]+\frac{\kappa_s\sqrt{p^{-1}N}\norm{W}_{\infty}\log N}{2\sigma_s}\Bigg),
\end{equation} with probability at least $1 - N^{-2}$.  Suppose in addition $\mathcal{E}$ holds, which happens with probability at least $1 - N^{-3}$. 

Recall that we wish to show that the LHS is less than $c\norm{W}_{\infty}\log^{-1}N$. Let us start by bounding $\kappa_s$. We have shown that $\sigma_1 = O(N)$. By definition of $s$, $\sigma_s \geq \frac{\norm{W}_{\infty}^{-2}}{16r} = \Omega(N)$, where we also use the incoherence assumption. It follows that $\kappa_s = O(1)$.
To handle the last term in the RHS of \eqref{infty-completion}, we use that $\frac{\sqrt{Np^{-1}}}{\sigma_s} = O(\log^{-2.01}N)$ by the signal-to-noise assumption, so 
\begin{equation}\label{m-infty-last-term}
\frac{\kappa_s\sqrt{p^{-1}N}\norm{W}_{\infty}\log N}{2\sigma_s} = o(\norm{W}_{\infty}\log^{-1}N).
\end{equation} 
For the remaining term, we first bound $\sup_{i \leq s}m_2(i)$. We appeal to the $\ell_2$ perturbation bounds of \cite{OVW1}. We adapt their results to our situation in Corollary \ref{mc-corollary} in Appendix \ref{l2-appendix}, which we can apply because $\bar{\delta} > cp^{-1}\log N$. This result gives that with probability at least $1 - N^{-3}$, 

\begin{equation}\label{m2-bound}
\sup_{i \leq s} m_2(i) = O\Big[\frac{\sqrt{\log N}}{p\bar{\delta}} + \frac{\norm{E}}{\sigma_s} + \frac{\norm{E}^2}{\sigma_s \bar{\delta}}\Big].
\end{equation} Using this bound and \eqref{m-infty-last-term}, \eqref{infty-completion} becomes 

\begin{equation}\label{squarebrackets}
        \sup_{i \leq s} m_{\infty}(i) = O\Bigg(\kappa_s\norm{W}_{\infty}\Big[ \frac{\sqrt{\log N}}{p\overline{\delta}} + \frac{\norm{E}}{\sigma_s} + \frac{\norm{E}^2}{\sigma_s \bar{\delta}}\Big]\Bigg) + o(\norm{W}_{\infty}\log^{-1}N).
\end{equation}

Since $\kappa_s = O(1)$, the proof of the lemma will be complete once we establish that there is a constant $c_1$ such that the sum of the three terms in the brackets is at most $O(c_1\log^{-1}N)$. Let us start with the third term in the square brackets, $\frac{\norm{E}^2}{\sigma_s \bar{\delta}}$. Since $\mathcal{E}$ holds, $$\norm{E}^2 = O(Np^{-1}).$$ By the gap assumption, $\bar{\delta}^{-1}p^{-1} = O(\frac{1}{\log N})$. Thus, since $\sigma_s = \Omega(N)$, $$\norm{E}^2\bar{\delta}^{-1}\sigma_s^{-1} = O(N\sigma_s^{-1}\bar{\delta}^{-1}p^{-1}) = O(\bar{\delta}^{-1}p^{-1}) = O(\log^{-1}N).$$ 

Moving to the first term in the square brackets in \eqref{squarebrackets}, we have just established that $\bar{\delta}^{-1}p^{-1} = O(\log^{-1}N)$. Lastly, for the second term in the square brackets in \eqref{squarebrackets}, $\norm{E}\sigma_s^{-1} = o(\log^{-1}N)$. This is by the signal-to-noise assumption and because $\mathcal{E}$ occurs. Since all three terms are either $O(\log^{-1}N)$ or $o(\log^{-1}N)$, the existence of $c_1$ can be quickly inferred.

A quick inspection of the proof gives that the total probability of occurrence of the events considered is at least $1 - 2N^{-2}$.

\end{proof}

\begin{proof}[Proof of Lemma \ref{stability-mc}]
Recall that $T = \inf\{t > 0: \mathbb{P}(\norm{E} > t) \le N^{-1/3}\}$. By Lemma \ref{lemma:normE-MC}, $T \le C\sqrt{Np^{-1}}$, where $C$ is the absolute constant from the lemma. Recall that $\bar{\delta}$ is the smallest gap in the first $s$ singular values of $A$. We will use $\sigma_s$ and  $\bar{\delta}$ to show that that the singular values and gaps $(\sigma_i, \delta_i)$ for $i \leq s$ satisfy the $(\frac{c}{a+1}, N^{-3}, 2)$ strong stability condition. Let $1 \leq i \leq s$.

We first verify the three conditions for $(\frac{c}{a+1}, N^{-3}, 2)$ stability, and conclude with verifying strong stability. First, the signal-to-noise condition gives
\begin{equation}\label{eq:signal-noise-MC}
\sigma_i \geq \sigma_s > c\sqrt{Np^{-1}}\log^{2.01}N,
\end{equation} which ensures that $$\sigma_i \geq 
\sigma_s > c\sqrt{Np^{-1}}\log^{2.01} N > \frac{c}{2}T \geq \frac{c}{a+1}T.$$ This shows that condition $(a)$ in Definition \ref{stable} holds for $(\sigma_i, \delta_i)$ with $\frac{c}{a+1}$ and $\tau = N^{-3}$. 

Next, recall the gap condition 
\begin{equation}
\label{gap-condition-MC}
\bar{\delta} > cp^{-1}\log N.
\end{equation} 
In order to verify $(c)$ of Definition \ref{stable}, we will show that 
\begin{equation}\label{delta-req}
\bar{\delta} > \frac{c}{2}T\kappa_i\norm{W}_{\infty},
\end{equation} which will imply $\delta_i > \frac{c}{a+1}T\kappa_i\norm{W}_{\infty}$, as desired. Because $\kappa_s = O(1)$, the bound on $T$ implies $$T\kappa_i\norm{W}_{\infty} = O(\sqrt{p^{-1}}\sqrt{N}\norm{W}_{\infty})= O(\sqrt{p^{-1}})$$ by the incoherence assumption. Thus, \eqref{delta-req} holds by \eqref{gap-condition-MC}.

We examine $(b)$ in Definition \ref{stable}, the final condition to verify $(\frac{c}{a+1}, N^{-3}, 2)$ stability. Recall that $K$ is the bound for the absolute value of the entries of $E$. Observe that $K \leq \norm{A}_{\infty}p^{-1}$.  Since $\nu = 2$, we must show that

\begin{equation}
\label{gap-term-MC-1}
\delta_i > \frac{c}{a+1}\Big(K\log N + \frac{T^2}{\sigma_s}\Big).\end{equation} 
Since $T \le C\sqrt{Np^{-1}}$, it is sufficient to show that 
\begin{equation}
\bar{\delta} > \frac{c}{a+1}\Big(K\log N + C^2\frac{Np^{-1}}{\sigma_s}\Big).\end{equation} Equation \eqref{gap-condition-MC} implies that 

$$\bar{\delta} > cp^{-1}\log N = \frac{c}{\norm{A}_{\infty}}\norm{A}_{\infty}p^{-1}\log N \geq \frac{c}{\norm{A}_{\infty}}K\log N \geq \frac{c}{a}K\log N.$$ Since $\sigma_s = \Omega(N)$, $\frac{C^2Np^{-1}}{\sigma_s} = o(K\log N)$ because $p > N^{-1}\log N$. Therefore, 

$$\bar{\delta} > \frac{c}{a}K\log N > \frac{c}{a+1}\Big(K\log N + C^2\frac{Np^{-1}}{\sigma_s}\Big)$$ as desired.

Having established $(\frac{c}{a+1}, N^{-3}, 2)$ stability, we conclude with the verification of strong stability. The signal to noise condition, equation
\eqref{eq:signal-noise-MC}, implies that \begin{equation}\sigma_i \geq \sigma_s > \frac{c}{\sqrt{\norm{A}_{\infty}}}\sqrt{N\norm{A}_{\infty}p^{-1}}\log^{2.01}N \geq \frac{c}{\sqrt{\norm{A}_{\infty}}}\sqrt{NK}\log^{2.01}N \geq \frac{c}{a+1}\sqrt{NK}\log^{2.01}N. \end{equation}Thus, the conditions for strong stability in Definition \ref{strong-stable} are satisfied for $(\sigma_i, \delta_i)$ with $\frac{c}{a+1}$, $\tau = N^{-3}$, and $\nu = 2$.
\end{proof}

\newpage

\printbibliography
\newpage

\appendix

\section{Perturbation of Singular values}
\label{sing-val-perturb}
We begin with the definition of the {\it concentration property}. The authors of \cite{OVW1} state that a square matrix $E$ satisfies the $(C, c, \gamma)$ concentration property if for all unit vectors $\vect u, \vect v \in \mathbb{R}^{n}$, and $t > 0$, 

\begin{equation}
    \label{concentration-property} 
    \mathbb{P}(\abs{\vect u^{T}E\vect v} > t ) \leq C\exp(-ct^{\gamma}).
\end{equation}

By using Hoeffding's inequality, they show that if $K \geq 1$ and $E$ is an $n \times n$ symmetric matrix with independent, $K$ bounded entries, then 

\begin{equation}
    \label{concentration-property-2} 
    \mathbb{P}(\abs{\vect u^{T}E\vect v} > t ) \leq 2\exp(-\frac{1}{8K^2}t^{2})
\end{equation}

In other words, $E$ satisfies the $(2, \frac{1}{8K^2}, 2)$-concentration property.

A key ingredient in the analysis is the singular values of the perturbed matrices $A^{\{l\}}$ and $\tilde{A}$. It is very easy to verify that if $E$ is a symmetric random matrix with independent, K- bounded entries, then for all $l$, $E^{\{l\}}$ has the $(2, \frac{1}{8K^2}, 2)$ concentration property as well.
Therefore, the results derived in \cite{OVW1} for the perturbation of the singular values applies to $A^{\{l\}}$ for all $l$.

The main result for the perturbation of singular values in \cite{OVW1} is the following theorem.

\begin{theorem}
\label{ovw-singular}
Suppose that $E$ is $(C, c, \gamma)$ concentrated. Suppose that $A$ has rank $r$, and let $1 \leq i \leq r$ be an integer. Then, for any $t \geq 0$, 

\begin{equation}
    \tilde{\sigma}_j \geq \sigma_j - t
\end{equation}

with probability at least 

\begin{equation}
    1 - 2C9^i\exp\Big(-c\frac{t^{\gamma}}{4^{\gamma}}\Big)
\end{equation}

and

\begin{equation}
    \tilde{\sigma}_i \leq \sigma_i + tr^{1/\gamma} + 2\sqrt{i}\frac{\norm{E}^{2}}{\tilde{\sigma}_i} + i\frac{\norm{E}^3}{\tilde{\sigma}^2_i}
\end{equation}

with probability at least 

\begin{equation}
    1 - 2C9^{2r}\exp\Big(-cr\frac{t^{\gamma}}{4^{\gamma}}\Big).
\end{equation}
\end{theorem}Theorem \ref{ovw-singular-K} follows from Theorem \ref{ovw-singular} and the fact that a symmetric, $K$ bounded random matrix with independent (above the diagonal) entries satisfies the $(2, \frac{1}{8K^2}, 2)$ concentration property.

\section{\texorpdfstring{$\ell_2$}{l2} Perturbation of Eigenvectors}
\label{l2-appendix}
\begin{theorem}[\cite{OVW1}]
\label{ovw-cp-thm}
    Suppose that $E$ is $(C, c, \gamma)$ concentrated for a trio of constants $(C, c, \gamma)$ and suppose that $A$ has rank $r$. Then, for any $t > 0$, 

    \begin{equation}
        \norm{\tilde{\vect u}_1 - \vect u_1}_{2} \leq 8\Big(\frac{tr^{1/\gamma}}{\delta} + \frac{\norm{E}}{\sigma_1} + \frac{\norm{E}^2}{\delta\sigma_1}\Big)
    \end{equation}
with probability at least

\begin{equation}
    1 - 54C\exp\Big(-c\frac{\delta^\gamma}{8^\gamma}\Big) - 2C9^{2r}\exp\Big(-cr\frac{t^{\gamma}}{4^{\gamma}}\Big). 
\end{equation}
\end{theorem}

Recall that a symmetric $E$ with independent, mean zero, $K$ bounded entries satisfies the $(2, \frac{1}{8K^2}, 2)$-concentration property. We will be considering only such matrices $E$ in what follows. Applying Theorem \ref{ovw-cp-thm} gives the following theorem. 

\begin{theorem} Let $E$ be a random, $K$-bounded, symmetric matrix with independent entries above the diagonal. For any $t > 0$, 

    \begin{equation}
        \norm{\tilde{\vect u}_1 - \vect u_1}_{2} \leq 8\Big(\frac{tr^{1/2}}{\delta} + \frac{\norm{E}}{\sigma_1} + \frac{\norm{E}^2}{\delta\sigma_1}\Big)
    \end{equation}
with probability at least

\begin{equation}
    1 - 108\exp\Big(-\frac{\delta^2}{8^3K^2}\Big) - 4 \times 9^{2r}\exp\Big(-r\frac{t^2}{128K^2}\Big). 
\end{equation}
\end{theorem}

As a consequence, if $\frac{\delta}{K\sqrt{8^3}} = \tau$, then for all $t \geq 0$, with probability at least $1 - 108\exp(-\tau^2) - 4*9^{2r}\exp\Big(-t\Big)$,

    \begin{equation}
        \norm{\tilde{\vect u}_1 - \vect u_1}_{2} \leq 8\Big(\frac{K\sqrt{128}t}{\delta} + \frac{\norm{E}}{\sigma_1} + \frac{\norm{E}^2}{\delta\sigma_1}\Big)
    \end{equation}

They also obtain the following recursive result for the perturbation of the remaining eigenvectors.

\begin{theorem}

Assume that $E$ is $(C, c, \gamma)$ concentrated. Suppose that $A$ has rank $r$, and let $1 \leq i \leq r$ be an integer. Then for any $t > 0$, 

\begin{equation}
    \norm{\vect u_i - \tilde{\vect u}_i}_{2} \leq 16\Bigg[\sum_{j = 1}^{i -1}\norm{\tilde{\vect u}_j - \vect u_j}_{2} + \frac{tr^{1/\gamma}}{\delta_j} + \frac{\norm{E}}{\sigma_j} + \frac{\norm{E}^2}{\sigma_j\delta_j}\Bigg]
\end{equation}

with probability at least 

\begin{equation}
    1 - 6C9^i\exp\Big(-c\frac{\delta_j^{\gamma}}{8^{\gamma}}\Big) - 2C9^{2r}\exp\Big(-cr\frac{t^{\gamma}}{4^{\gamma}}\Big)
\end{equation}
\end{theorem}

Fix $1 \leq s \leq r$. $s$ represents a cutoff point beyond which the eigenvectors are not of interest. Let $\epsilon_j(t) := 16\Big(\frac{tr^{1/\gamma}}{\delta_j} + \frac{\norm{E}}{\sigma_j} + \frac{\norm{E}^2}{\sigma_j\delta_j}\Big)$. By taking the union bound over the first $s$ singular vectors and iterating this recursive bound, and letting $\overline{\delta}$ be the smallest gap in the first $s$ singular values, we obtain

\begin{theorem} For all $i \leq s$, and $t > 0$,
\begin{equation}
    \norm{\tilde{\vect u}_i - \vect u_i}_{2} \leq \sum_{j = 1}^{i}[\sum_{k = 0}^{i - j}16^{k}]\epsilon_{j}(t)
\end{equation}

with probability at least 

\begin{equation}
        1 - 6sC9^s\exp\Big(-c\frac{\bar{\delta}^\gamma}{8^{\gamma}}\Big) - 2Cs9^{2r}\exp\Big(-cr\frac{t^{\gamma}}{4^{\gamma}}\Big).
\end{equation}
\end{theorem}

As a consequence, letting $\tau = \frac{\overline{\delta}}{K\sqrt{8^3}}$ if $E$ is symmetric and $K$-bounded with independent entries, there exists $C$ depending only on $s$ such that with probability at least $1 - C[\exp(-\tau^{2}) - 9^{2r}\exp(-t)]$,

    \begin{equation}
        \sup_{i \leq s}\norm{\tilde{\vect u}_i - \vect u_i}_{2} \leq C\Big(\frac{Kt}{\overline{\delta}} + \frac{\norm{E}}{\sigma_s} + \frac{\norm{E}^2}{\overline{\delta}\sigma_s}\Big),
    \end{equation}
for any $t > 0$. In the rectangular case, where $A$ is an $m \times n$ matrix, and when $r = O(1)$, we derive the following corollary using the standard symmetrization trick. We encounter this setting in the matrix completion problem. Let $N = m + n$.

\begin{corollary}
\label{mc-corollary}
If $r = O(1)$ and $\overline{\delta} > 400K\sqrt{\log N}$, with probability at least $1 - N^{-3}$, we have

    \begin{equation}
        \sup_{i \leq s}\max\{\norm{\tilde{\vect u}_i - \vect u_i}_{2}, \norm{\tilde{\vect v}_i - \vect v_i}_{2}\}  = O\Big(\frac{Kt}{\overline{\delta}} + \frac{\norm{E}}{\sigma_s} + \frac{\norm{E}^2}{\overline{\delta}\sigma_s}\Big).
    \end{equation}
\end{corollary}

\section{Results for Section \ref{refined-proof}}
\label{appendix:refined-theorem-proof}
\begin{lemma}
\label{bernstein-bound-2}
Let $Y \in \overline{\mathcal{L}_{\{\}, l}}$ holds. Then, 

\begin{equation}
\begin{split}
\mathbb{P}(\mathcal{I}'_{l}\lvert E^{\{l\}} = Y) &\leq 2\exp(-c_2\log n).
\end{split}
\end{equation}

\end{lemma}

\begin{proof}[Proof of Proposition \ref{prop:Jbound} given the lemma.] 
The proposition follows immediately from the lemma once we observe that 

$$\mathbb{P}(\mathcal{I}_l' \cap \overline{\mathcal{L}_{\{\}, l}}) \leq \sup_{Y \in \overline{\mathcal{L}_{\{\}, l}}}\mathbb{P}(\mathcal{I}_l' | E^{\{l\}} = Y) \leq 2\exp(-c_2\log n).$$
\end{proof}

\begin{proof}[Proof of Lemma \ref{bernstein-bound-2}]
We are looking for a bound of 
$$\mathbb{P}(\mathcal{I}'_{ l} | E^{\{l\}} = Y) = \mathbb{P}\{\abs{\langle \vect u^{\{l\}}_i, \vect x(l)\rangle} \geq c_2\sqrt{2Kn}f_{j+1}\log n \Big\lvert E^{\{l\}} = Y\}.$$ The lemma only considers realizations $Y$ of $E^{\{l\}}$ satisfying $\norm{\vect u^{\{l\}}_i}_{\infty} \leq f_{1}$. Recall that $\vect x(l)$ is essentially the $l$th row of $E$ and $l \in \beta$. Therefore, conditional on such $E^{\{l\}}$, $\vect u^{\{l\}}_i$ is a deterministic unit vector whose entries have absolute value at most $f_{1}$. The only randomness in each event thus comes from $\vect x(l)$. It follows that the inner product $\langle \vect u^{\{l\}}_i, \vect x(l) \rangle$, conditional on $E^{\{l\}}$,  is the sum of independent, $Kf_{1}$ bounded, mean zero random variables $x_ku^{\{l\}}_{ik}$. Thus, this quantity can be bounded with Bernstein's inequality. We are applying Bernstein's inequality conditionally, so we also need to find a bound for the sum of the conditional second moments of the $x_ku^{\{l\}}_{ik}$. Since $\vect u^{\{l\}}_i$ is deterministic when we condition on $E^{\{l\}}$, and $\vect x(l)$ is independent of $E^{\{l\}}$, we have
\begin{equation}
\label{conditional-variance-2}
    \sum_{k = 1}^n \mathbb{E}[x_k^2 u^{\{l\} 2}_{ik} | E^{\{l\}} = Y]  = \sum_{k = 1}^{n}u^{\{l\}2}_{ik}\mathbb{E}[x_k^2] \leq K\sum_{k = 1}^{n}u^{\{l\}2}_{ik} = K.
\end{equation} For the inequality, we use that the second moments of the entries of $E$ are at most $K$ by Assumption \ref{assumption-K}. Applying Bernstein's inequality (Lemma \ref{bernstein}) then gives us that
\begin{equation}
\label{apply-bernstein-2}
\begin{split}
\mathbb{P}\Big\{\abs{\langle \vect u^{\{l\}}_i, \vect x(l) \rangle} > t | E^{\{l\}} = Y\Big\} &\leq 2\exp\Big(\frac{-t^{2}/2}{\sum_{k = 1}^{n} \E[x^2_k u^{\{l\}2}_{ik}|E^{\{l\}} = Y] + Kf_{1}t/3}\Big) \\
&\leq 2\exp\Big(\frac{-t^{2}/2}{K + Kf_{1}t/3}\Big).
\end{split}
\end{equation} Set $t = c_2\sqrt{2Kn}f_{1}\log n$.  Since $n^{-1/2} \leq \norm{U}_{\infty} \leq f_{1}$, we obtain that $K \leq Knf_{1}^{2}$. Thus, using this bound for the first term in the denominator of the RHS of \eqref{apply-bernstein-2}, \begin{equation}
\label{bernstein-calc-2}
\begin{split}
\mathbb{P}\Big\{\abs{\langle \vect u^{\{l\}}_i, \vect x(l) \rangle} > c_2\sqrt{2Kn}f_{1}\log n \Big| E^{\{l\}} = Y\Big\} &\leq 2\exp\Big(\frac{-c_2^2Knf_{1}^{2}\log^2 n}{Knf_{1}^2 + \frac{\sqrt{2}c_2}{3}Kf^2_{1}\sqrt{Kn}\log n}\Big) \\
&\leq 2\exp\Big(\frac{-c_2^2Knf_{1}^{2}\log^2 n}{Knf_{1}^2 + \frac{\sqrt{2}c_2}{3}Kf_{1}^2n\log n}\Big) \\
&\leq 2\exp(-c_2\log n).
\end{split}
\end{equation} In the second line, we used that $K \leq n$ by Assumption \ref{assumption-K}.
\end{proof}
\section{Proof of Lemma \ref{lemma:normE-MC}}
In this section, we bound the spectral norm of the matrix $E$ from the matrix completion problem. We will use the following result from \cite{bandeiravanhandel}, which is a tail bound for the norm of $K$ bounded random matrices with independent entries.

\label{appendix:normE-MC}
\begin{theorem}[Remark 3.13 in \cite{bandeiravanhandel}]
\label{theorem:bandeiravanhandel}
Let $X$ be a symmetric, mean zero $n \times n$ random matrix whose entries above the diagonal are independent. Suppose the entries of $X$, $\xi_{ij}$, are $K$ bounded random variables. Let $$v = \max_i\sum_{j}\E[\xi_{ij}^2].$$

Then there exists a universal constant $c > 0$ such that for any $t \geq 0$,

$$\mathbb{P}\{\norm{E} \geq 4\sqrt{v} + t\} \leq n\exp\Big(-\frac{t^2}{cK^2}\Big).$$

\end{theorem}

\begin{proof}[Proof of Lemma \ref{lemma:normE-MC}]
Recall that $E$ is an $m \times n$ random matrix with independent, $K$ bounded entries, where $K = O(p^{-1})$. Form the symmetrization of $E$ as $$S = S(E) = \begin{bmatrix}
0 & E \\
E^T & 0 
\end{bmatrix}.
$$ $S$ is a symmetric $N \times N$ random matrix satisfying the conditions of Theorem \ref{theorem:bandeiravanhandel}. It is easy to check that $\norm{E} \leq \norm{S}$. Let the entries of $S$ be given by $s_{ij}$. We need to calculate $v = \max_{i}\sum_j\mathbb{E}[s_{ij}^2]$ to apply the theorem. Recall that in the matrix completion setting, the entries of $E$ have second moment at most $K$. It follows that $\E[s_{ij}^2] \leq K$, so $v \leq NK$.

By Theorem \ref{theorem:bandeiravanhandel}, 

\begin{equation}
\mathbb{P}\{\norm{S} \geq 4\sqrt{NK} + t\} \leq N\exp\Big(\frac{-t^2}{cK^2}\Big).
\end{equation} Set $t = K\sqrt{4c\log N}$. It follows that
\begin{equation}
\mathbb{P}\{\norm{S} \geq 4\sqrt{NK} + K\sqrt{4c\log N} \} \leq N\exp\Big(-4\log N\Big).
\end{equation}

Since $\norm{E} \leq \norm{S}$, 

\begin{equation}
\mathbb{P}\{\norm{E} \geq 4\sqrt{NK} + \sqrt{4cK\log N}\} \leq N^{-3}.
\end{equation}

Recall that $K = O(p^{-1})$. Under the assumptions of Theorem \ref{threshould-round-theorem}, $p > \frac{\log N}{N}$, so $\sqrt{K\log N} = o(\sqrt{N})$. 

It follows that there exists an absolute constant $C$ such that

\begin{equation}
\mathbb{P}\{\norm{E} \geq C\sqrt{NK}\} \leq N^{-3}.
\end{equation}

\end{proof}

\end{document}